\documentclass{article}

\usepackage[utf8]{inputenc}
\usepackage[T1]{fontenc}
\usepackage[a4paper, margin=2.7cm]{geometry}

\usepackage{bbm}

\usepackage{amsmath, amsthm, amsfonts, amssymb}
\usepackage{mathrsfs}
\usepackage{mathtools}
\usepackage{dsfont}        
\usepackage{bbm}           
\usepackage{esint}         
\usepackage{cancel}
\usepackage{tensor}

\usepackage{tikz}
\usepackage{xcolor}

\usepackage{enumerate}
\usepackage{caption}
\usepackage{booktabs}
\usepackage{thmtools}

\usepackage[colorlinks=true, linkcolor=blue, citecolor=blue,
            urlcolor=blue, breaklinks]{hyperref}
\usepackage[nameinlink]{cleveref}

\usepackage{natbib}
\usepackage{url}

\theoremstyle{plain}
\newtheorem{teorema}{Theorem}[section]
\newtheorem*{teorema*}{Theorem}
\newtheorem{lemma}[teorema]{Lemma}
\newtheorem{proposizione}[teorema]{Proposition}

\newtheorem{corollario}[teorema]{Corollary}

\newtheorem{assumption}[teorema]{Assumption}

\theoremstyle{definition}
\newtheorem{definizione}[teorema]{Definition}

\theoremstyle{remark}

\newtheorem{osservazione}[teorema]{Remark}

\renewcommand{\div}{\operatorname{div}}   

\newcommand{\abs}[1]{\left\vert#1\right\vert}
\newcommand{\cl}[1]{\left\lceil#1\right\rceil}
\newcommand{\norm}[1]{\left\Vert#1\right\Vert}

\def\d{\,\mathrm{d}}

\def\N{\mathbb{N}}
\def\R{\mathbb{R}}
\def\E{\mathbb{E}}
\def\P{\mathbb{P}}

\def\M{\mathcal{M}}
\def\L{\mathcal{L}}

\def\Var{{\textrm{Var}}\,}
\def\hess{{\textrm{Hess}}\,}
\def\trace{{\textrm{Tr}}\,}

\def\Dir{{\textrm{Dir}}\,}
\def\Cov{{\textrm{Cov}}\,}

\def\pois{{\textrm{Pois}}\,}
\def\Id{{\textrm{Id}}\,}

\def\ird{\int_{\R^d}}
\def\:{\colon}
\def\e{\varepsilon}
\def\f{\varphi}

\def\NN{\mathcal{N}}
\def\diffC{\tilde{\kappa}}
\newcommand{\starconv}[2]{\tensor[_{#1}]{\star}{_{#2}}}

\title{Self-similarity and diffusive limits for linear kinetic equations: a Wild sum approach}
\author{
  José A.~Cañizo$^{*}$ \and
  Stéphane Mischler$^{\dagger}$ \and
  Niccolò Tassi$^{\ddagger}$
}

\date{}
\hypersetup{
  pdftitle  = {Self-similarity and diffusive limits for linear kinetic equations: a Wild sum approach},
  pdfauthor = {José A. Cañizo, Stéphane Mischler, Niccolò Tassi},
}

\AtEndDocument{\vspace{2cm}{
		\noindent
		\textsc{José A. Cañizo. Departamento de Matemática Aplicada \& IMAG,
			Universidad de Granada. Avda. Fuentenueva S/N, 18071 Granada, Spain.}
		\textit{E-mail address}: \texttt{\href{mailto:canizo@ugr.es}{canizo@ugr.es}}
        \par
		\addvspace{\medskipamount}
		\noindent
		\textsc{Stéphane Mischler. CEREMADE, UMR CNRS 7534 \& Université Paris-Dauphine -- PSL Research University \& IUF, Place du Mar\'echal de Lattre de Tassigny, 75775 Paris Cedex 16, France.} 
		\textit{E-mail address}: \texttt{\href{mailto:mischler@ceremade.dauphine.fr}{mischler@ceremade.dauphine.fr}}
		\par
		\addvspace{\medskipamount}
		\noindent
		\textsc{Niccolò Tassi. Warwick Mathematics Institute, University of Warwick, Zeeman Building, Warwick CV4 7AL, United Kingdom.}
		\textit{E-mail address}: \texttt{\href{mailto:niccolo.tassi@warwick.ac.uk}{niccolo.tassi@warwick.ac.uk}}}
}

\begin{document}
\maketitle

\maketitle
\footnotetext[1]{$^*$Departamento de Matemática Aplicada \& IMAG,
			Universidad de Granada. Avda. Fuentenueva S/N, 18071 Granada, Spain. E-mail: \texttt{canizo@ugr.es} }
\footnotetext[2]{$^\dagger$CEREMADE, UMR CNRS 7534 \& Université Paris-Dauphine -- PSL Research University \& IUF, Place du Mar\'echal de Lattre de Tassigny, 75775 Paris Cedex 16, France. E-mail: \texttt{mischler@ceremade.dauphine.fr}}
\footnotetext[3]{$^\ddagger$Warwick Mathematics Institute, University of Warwick, Zeeman Building, Warwick CV4 7AL, United Kingdom E-mail: \texttt{niccolo.tassi@warwick.ac.uk}}

\begin{abstract}
  We prove that linear collisional kinetic equations in the whole
  space without a confinement mechanism display a long-time
  diffusive behaviour.  This improves the known results (decay
  estimates) in this context, providing the first result regarding
  this self-similar behaviour. As a consequence, we also establish a
  uniform-in-time convergence of suitably rescaled solutions to their
  diffusion limit, which is also new. The class of equations
  considered includes some BGK-type equations, some kinetic nonlocal
  Fokker--Planck-type equations and some kinetic (possibly fractional)
  Fokker--Planck equations, for which we are able to explicitly write
  solutions through a Wild sum (or Dyson series) or through Fourier
  transform methods. When these representations are manageable the
  results we give are especially clean, with explicit convergence
  rates, natural $L^p$ norms, and weak requirements on the initial
  condition.
\end{abstract}
\setcounter{tocdepth}{2}
\tableofcontents

\section{Introduction}
We study the long-time asymptotic behaviour of the linear kinetic
evolution equation
\begin{equation}\label{eq:general}
  \partial_t f + v\cdot \nabla_x f = \L f,
  \qquad f(0,\cdot )= f_0, 
\end{equation}
where $f = f(t,x,v)$ is a density function that depends on time
$t \geq 0$, space $x \in \R^d$ and velocity $v \in \R^d$.
\smallskip We assume that the collisional operator $\mathcal{L}$ only
acts on the $v$ variable, is linear, that its associated semigroup
in the $v$ variable preserves mass and positivity, and finally that it
admits a unique probability microscopic equilibrium $F$ (that is,
$F= F(v)$ is a probability density satisfying $\L F = 0$). In
particular, the operator $\L$ has only one invariant (mass). In the
sequel, we will consider
\begin{enumerate}[(I)]
\item the {\it BGK operator}
\begin{equation}
\label{eq:defL-BGK}\tag{I}
 \mathcal{L} f  = F  \Big(\ird f(w) \d w \Big)-f;
\end{equation}
   
\item the {\it nonlocal Fokker--Planck operator}
 \begin{equation}
      \label{eq:defL-nonlocalFP}\tag{II}
  \mathcal{L}f= G *_vf - f + \div_v(vf);
\end{equation}
\item the {\it (possibly) fractional Fokker--Planck} operator with $s
  \in (0,1]$,
    \begin{equation}
          \label{eq:defL-fracFP}\tag{III}
        \mathcal{L}f = - (-\Delta_v)^s f + \hbox{\rm div}_v (v f). 
 \end{equation}
\end{enumerate}

Let us make precise some notations and assumptions.

\begin{itemize}

\item In \eqref{eq:defL-BGK}, $F = F(v)$ is a given centred and normalised probability
  distribution on $\R^d$ (and thus the unique probability
  equilibrium), and we assume that either
 \begin{equation}\label{eq:F-diffusecase}
    \ird F(v) |v|^{2 + \delta} \d v < +\infty, 
  \end{equation}
for some $\delta > 0$, setting by definition $s := 1$ in this case, or 
  \begin{equation}\label{eq:F-factionalcase}
    F(v) \sim \frac{1}{|v|^{d + 2s}}
    \qquad \text{as $|v| \to +\infty$,}
  \end{equation}
for some $s  \in (0,1)$. We mainly consider $F=\mathcal{M}$, the standard Gaussian, or $F=\mathcal{M}^s$, the stable law, although other choices are possible (see Section \ref{sec: FNOstable}).

\item In \eqref{eq:defL-nonlocalFP}, $*_v$ denotes the convolution in
  the $v$ variable, i.e. $G*_vf=\int G(w)f(x,v-w)\d w$.  $G$ is always
  a Gaussian distribution in $\R^d$ with covariance matrix $2\Id$:
  \begin{equation*}
    G(v) := (4 \pi)^{-\frac{d}{2}} e^{-\frac{|v|^2}{4}},
    \qquad v \in \R^d.
  \end{equation*}
  By definition, we set $s := 1$ in this case. In particular, one can
  show that $\mathcal{L}$ admits a unique probability equilibrium
  $F = F(v)$ also satisfying \eqref{eq:F-diffusecase}; see for
  instance \cite{mischler_uniform_2017, canizotassi2024}.

\item In \eqref{eq:defL-fracFP}, we assume $s \in (0,1]$. This
  (possibly fractional) Fokker--Planck operator has a unique
  probability equilibrium $F(v) = \mathcal{M}^s(v)$, the stable law of
  parameter $2s$ defined by its Fourier transform
    \begin{equation}
      \label{eq:stable_law}
      \widehat{\mathcal{M}}^s(\eta) = e^{-\frac{|\eta|^{2s}}{2s}},
      \qquad \eta \in \R^d, 
    \end{equation}
    see for instance
    \cite{gentil_levy-fokker-planck_2006}.
\end{itemize}
    
In this paper we
      always use the Fourier transform defined by
  \begin{equation*}
    \widehat{f}(\xi) := \ird f(x) e^{-ix \xi} \d x,
    \qquad \xi \in \R^d.
  \end{equation*}
  We distinguish the notation $G$ (and $G^s$) from $\M$ (and
  $\mathcal{M}^s$) when referring respectively to convolution kernels
  and equilibrium distributions.

  \smallskip The main result in this paper is that these equations
  show a self-similar behaviour for large times which matches the
  (fractional) heat equation. These long-time estimates are new and
  give as a consequence very precise results on the diffusive
  limit. After a short introduction and a summary of related results
  we will give a more precise statement in Theorem \ref{thm: main}.

  \smallskip Kinetic equations of the form \eqref{eq:general} appear
  in several contexts: in the interaction of particles with a
  background gas, $\L$ should be the linear Boltzmann operator with a
  suitable collision kernel; it may be a linear BGK-type operator in
  models for scattering. Similar PDEs find applications in biology,
  such as run-and-tumble models \citep{Alt1980,Evans2023}, medical
  imaging \citep{QinLiPhonon2026}, in the social sciences, including
  opinion formation models \citep{Toscani2006_opinion} and wealth
  distribution \citep{Cordier2005}, and models for vehicular traffic
  flow \citep{Heraty2020_traffic}. The same structure appears in some
  relativistic \citep{marle1965modele} and quantum models
  \citep{markowich1990semiconductor, carlen2026Zeno}. We also mention
  that numerical methods for nonlocal kinetic equations of this type
  have also been studied in \cite{ayi_structure-preserving_2022}.

  \smallskip The evolution kinetic equation \eqref{eq:general} does
  not have any confining mechanism in the variable $x$, and solutions
  are expected to converge to $0$ for large times.  A precise
  description of this long-time behaviour is not easy to give and only
  recently some quantitative results have been proved which provide
  possible rates at which any solution $f$ goes to $0$ in several
  norms, see
  \citet{Bouin_etAl2020HypoWoConf,BouinDolbeaultLafleche_2022FractHypo}.
  More precisely, the solution $f$ to \eqref{eq:general} satisfies
\begin{equation}
  \label{eq:Rate-kineticL2}
  \| f(t, \cdot, \cdot) \| \leq C t^{-q}
  \quad \text{for any $t > 0$,} \quad q = q(s,2) := \frac{d}{4s}, 
\end{equation}
for a weighted $L^2$ norm $\|\cdot\|$ in the variables $x$, $v$ and for some constant $C = C(f_0)$. 
This exponent $q$ is also relevant for the diffusive limit of the
kinetic equation \eqref{eq:general} that we introduce now.

\smallskip If the time scale at which the ``scattering'' $\L$ acts is
much shorter than the transport scale, it is appropriate to consider a
scaled version of equation (\ref{eq:general}) where we add a scaling
parameter $\e > 0$: we define
\begin{equation}\label{eq:scaledf}
  f^\e(t,x,v) := \frac{1}{\e^{d}} f\Big(\frac{t}{\e^{2s}}, \frac{x}{\e}, v\Big).
\end{equation}
Then, 
if $f$ is a
solution to (\ref{eq:general}) with initial condition $f_0$,  $f^\e$
is a solution to the scaled equation \begin{equation}\label{eq:general-scaled}
  \e^{2s}
  \partial_t f^{\e} + \e v\cdot \nabla_x f^{\e}
  =  \mathcal{L}f^{\e},
\end{equation}
associated to $f^\e_0$, with
$f_0^\e(x,v) := \e^{-d} f_0 (\e^{-1} x, v)$. In the (fractional)
diffusion limit $\e \to 0$, it is known that the solution $f^\e$
converges to $F(v) u(t,x)$ on compact time intervals, where $u$ is a
solution to the (fractional) heat equation
 \begin{equation}\label{eq: fracheat}
  \partial_t u=-\kappa(-\Delta)^s u, 
  \qquad u(0,\cdot )= u_0, 
\end{equation}
where $\kappa >0$ is a suitable constant, see for instance
\cite{Degond_Goudon_Poupad2000DiffusionLimit},
\cite{Mellet_Mischler_Mouhot2010Fractional} for the BGK operator, and
\cite{puel2011fractional,lebeau2019diffusion, dechicha2024fractional}
for the (possibly fractional) diffusion operator. From the
probabilistic perspective, \cite{Jara_Komorowski_Olla2009LimThms}
obtained similar results as
\cite{Mellet_Mischler_Mouhot2010Fractional}. In the same spirit,
\cite{basile2010convergence, basile2014kinetic} considered related
kinetic equations with anomalous scaling.  Later,
\cite{Fournier_Tardif2018criticalkinetic, Fournier_Tardif2020Fmultid}
studied the kinetic Fokker-–Planck equation with heavy-tailed local
equilibria for a wider range of tail exponents. All of the previous
results prove convergence but give no rates.  Numerical methods for
linear kinetic equations of type \eqref{eq:defL-BGK} and
\eqref{eq:defL-fracFP} for $s=1$ which are asymptotic-preserving (in
the sense that they behave well under the diffusive rescaling) have
been studied in \cite{bessemoulin2020hypocoercivity}.

Concerning case \eqref{eq:defL-nonlocalFP}, we believe the proof of
the diffusive limit is not strictly contained in the previous
literature, but similar techniques may also apply to this case.

Let us comment on existing quantitative results on rates of
convergence. In the classical case, explicit rates for regular enough
initial data in the diffusion asymptotics ($\e \to 0$ in
\eqref{eq:general-scaled}) were given in \cite{Bardos1984}. For the
fractional case, rates have been given in \cite*{Bouin_Mouhot2022},
namely
\begin{equation}\label{eq:quantitative-diffusion-limit}
\sup_{[0,T]} \|f^\e - u F \| \le C_T \varepsilon^{\alpha}, 
\end{equation}
in a weighted $H^{-\zeta}_xL^2_v$ norm for any $T > 0$, with an
explicit---but hard to write---exponent $\alpha>0$. The constant
$C_T$, however, depends on time $T$ and does not allow one to extend
the result uniformly in time.

\smallskip

The fractional heat equation \eqref{eq: fracheat} can be solved
explicitly on the Fourier side and thus its solutions have an explicit
representation formula through a convolution product with an
appropriate time-dependent kernel, whose natural $L^p$ decay in time
is given by the following exponent $q$, depending on
$p \in [1,\infty]$ and $s \in (0,1]$:
\begin{equation}\label{def:q-gal}
  q  := q(s,p)= \frac{d}{2s}\left(1-\frac{1}{p}\right),
\end{equation}
which is consistent with the previous definition of $q(s,2)$.  One can
show that, when starting from a probability measure $u_0$ with
appropriate moment conditions, the
associated solution $u$ to the (fractional) heat equation \eqref{eq:
  fracheat} satisfies
\begin{equation}\label{eq:FracHeateq-Longtime} 
\| u (t,\cdot) - (1+t)^{-d/(2s)} \mathcal{M}^s_{(2s\kappa)^{1/(2s)}}((1+t)^{-1/(2s)} \cdot)  \|_{L^p} \le \frac{C}{(1+t)^{q + \gamma } }, \quad \hbox{for any} \ t > 0, 
  \end{equation}
for some $\gamma > 0$, and thus in particular 
\begin{equation}\label{eq:FracHeateq-Rate} 
\| u (t,\cdot) \|_{L^p} \sim \frac{c_s}{t^q}, \quad \hbox{as} \ t \to \infty,  \quad c_s > 0, 
\end{equation}
see \cite{vazquez_asymptotic_2018,
  vazquez2017asymptoticbehaviourfractionalheat}.  With this in mind,
our results will also show that the rate provided in
\citet{Bouin_etAl2020HypoWoConf,BouinDolbeaultLafleche_2022FractHypo}
is optimal at least in cases \eqref{eq:defL-BGK} and
\eqref{eq:defL-fracFP}.

\smallskip In our main result below and its corollary, we are able to
improve \eqref{eq:Rate-kineticL2} by obtaining a self-similar estimate
analogue of \eqref{eq:FracHeateq-Longtime}, which leads to a
uniform-in-time version of \eqref{eq:quantitative-diffusion-limit}.
In order to make a precise statement, we introduce some function
spaces.  For $p,r \in [1,+\infty]$ we define the mixed $(p,r)$
Lebesgue space $L^p_xL^r_v(\R^d\times\R^d)$ of measurable functions on
$\R^d \times \R^d$ associated to the norm
\begin{equation}
  \label{eq:LpxLrv}
  \| f \|_{L^p_xL^r_v} := \left(\ird\left(\ird |f(x,v)|^r\d
      v\right)^{\frac{p}{r}}\d x\right)^{\frac1p},
\end{equation}
for $p,r \in [1,+\infty)$, with usual generalisations to the cases $p$
or $r =+\infty$.
 
\begin{teorema}[Long-time behaviour -- summary]
  \label{thm: main}
   Let $\mathcal{L}$ be one of the operators in \eqref{eq:defL-BGK}, \eqref{eq:defL-nonlocalFP} or \eqref{eq:defL-fracFP}, 
let $F = F(v)$ be the unique probability equilibrium of $\L$ and let $s \in (0,1]$ be the associated exponent.  
   There exists a  constructive constant $\gamma > 0$, such that for any  initial data $f_0$ which is a  probability
  distribution on $\R^{2d}$ satisfying
  \begin{equation*}
    M_\nu := \ird \ird f_0(x,v) |v|^\nu \d v \d x < +\infty, 
  \end{equation*}
 with $\nu=1$ if $s\in(1/2,1]$ and $\nu\in(s,2s)$ if $s\in(0,1/2]$, the associated solution $f$ to the kinetic equation \eqref{eq:general} 
 satisfies
\begin{equation}
  \label{eq:main-thm}
  \norm{f(t,\cdot,\cdot) - F(\cdot)u(t,\cdot)}_{L^p_x L^r_v}
  \le C \underbrace{(1+t)^{-q}}_{\text{Self-similar decay}} \ \underbrace{(1+t)^{-\gamma}}_{\text{Second order decay}}, 
  \end{equation}
  where $u$ is the solution to the (fractional) heat equation
  \eqref{eq: fracheat} with diffusivity constant $$\kappa_\mathcal{L}:=
  \begin{cases}
      \Gamma(2s)&\text{if $s\in(0,1)$ and $\mathcal{L}=\mathcal{L}_{BGK}$}\\
      1&\text{otherwise}
  \end{cases}
  $$ and initial macroscopic
  density $$u_0(x):=\int f_0(x,w)\d w, $$ and where we recall the
  definition \eqref{def:q-gal} of $q$.  The constant $C$ depends
  (increasingly) only on $d,\,p,\,r,\,s$, $\|f_0\|_{L^p_x L^r_v}$,
  $\|f_0\|_{L^1_v L^p_x }$, and $M_\nu$. The precise hypotheses, the
  allowed values of $\nu$, $p$, $r$ and the values of $\gamma$ are
  given in Table \ref{tab:placeholder} and explicitly presented in
  Theorems \ref{thm:mainBGKStable}
  (case \eqref{eq:defL-BGK} $F=\mathcal{M}$ or $F=\mathcal{M}^s$), \ref{thm: mainNLFP} (case
  \eqref{eq:defL-nonlocalFP}) and \ref{thm:mainKFP} (case
  \eqref{eq:defL-fracFP}). The case $I$ for a general equilibrium $F$ is sketched in  Theorem  \ref{thm: mainBGKgeneral}.
\end{teorema}

The decay $t^{-q}$ is exactly the time rate of decay of the
$L^p_x L^r_v$ norm of $f(t,\cdot)$ already mentioned for $p=2$ in
\eqref{eq:Rate-kineticL2} (and we refer to
\cite{bouin2025halfspacedecaylinearkinetic} for a possible extension
of this result to any Lebesgue space) and of $u(t,\cdot) F$ from
\eqref{eq:FracHeateq-Rate}.  Because of the additional $t^{-\gamma}$
term, \eqref{eq:main-thm} drastically improves the decay rate
\eqref{eq:Rate-kineticL2} and shows that $u(t,\cdot) F $ is a
first-order approximation to the long-time behaviour of the solution
$f(t,\cdot)$ to \eqref{eq:general}.  We believe that, at least for the
case $s=1$, the optimal $\gamma$ should be $\frac{1}{2}$ in {all}
cases, but we were not able to reach this value for all cases --- see
Table \ref{tab:placeholder}.

\smallskip Alternatively, combining \eqref{eq:main-thm} with
\eqref{eq:FracHeateq-Longtime}, we deduce
\begin{equation}\label{eq:SELFSIMILAR}
 \norm{f(t,\cdot,\cdot) -  (1+t)^{-d/(2s)} \mathcal{M}_{(2s\kappa)^{1/(2s)}}^s((1+t)^{-1/(2s)} \cdot) F(\cdot) }_{L^p_x L^r_v} \le \frac{C}{ (1+t)^{q + \gamma } },
\end{equation}
for any  solution $f$ to the kinetic equation \eqref{eq:general} associated to a probability  distribution $f_0$ and any $t \ge 0$.

Since the constant in \eqref{eq:FracHeateq-Longtime} depends on a
moment in $x$ of $u_0$, the constant in \eqref{eq:SELFSIMILAR} depends
on moments in both $x$ and $v$, making the previous result less
general than \eqref{eq:main-thm}. However, the last formulation
\eqref{eq:SELFSIMILAR} makes the long-time asymptotic self-similar
behaviour of $f$ even clearer.

\smallskip

Recalling the scaling 
\eqref{eq:scaledf},  by a change of variables, Theorem \ref{thm: main} immediately gives the
following seemingly more general result:

\begin{corollario}[Long-time behaviour and diffusive limit]
  \label{thm:main-epsilon}
  Let us make the same assumptions  as in Theorem
  \ref{thm: main} on the operator  $\mathcal{L}$, the equilibrium $F$, the exponent $s \in (0,1]$
  and the initial datum $f_0$. 
Any solution  $f^\e$ to the rescaled kinetic equation 
  \eqref{eq:general-scaled} associated to the initial condition $f^\e_0$ satisfies
    $$
  \norm{f^\e(t,\cdot,\cdot) - F(\cdot)u(t,\cdot)}_{L^p_x L^r_v}
  \le C t^{-q}\left( \frac{\e^{2s}}{t} \right)^{\gamma}, 
  $$
  for any $t > 0$, 
  where $u$ is again the solution to the  (fractional) heat equation \eqref{eq: fracheat} 
  associated to the 
   initial
  macroscopic density $u_0(x):=\int f^\e_0(x,w)\d w$. %
  The constant $C$
  depends (increasingly) only on the norms of
$f_0^\e$ required in Theorem \ref{thm: main} and on 
  $M_\nu$. The precise hypotheses and values of $\nu$, $p$,
  $r$, $\gamma$ are exactly as in Theorem \ref{thm: main}.
\end{corollario}

It is essential to notice that the scaling of the constant $C$ in
Theorem \ref{thm: main} is enough to allow us to deduce this result:
in order to obtain Corollary \ref{thm:main-epsilon} one must apply
Theorem \ref{thm: main} to an initial data
$f_0(x,v) = \e^d f_0^\e(\e x, v)$. The constant $C$ depends only on
$M_\nu$, which stays fixed under this scaling, and depends
increasingly on $\|f_0\|_{L^p_x L^r_v}$, which does not increase as
$\e \to 0$ if $f_0^\e$ stays fixed.
Also, notice that we do not require this initial data to be
well-prepared in the sense of \cite{Bouin_Mouhot2022}. If we add this hypothesis
then the result can clearly be improved for small times, 
obtaining convergence in $L^\infty\left([0,\infty); L^p_xL^r_v\right)$.
 
\begin{table}[ht]
    \centering
    \begin{tabular}{cccccc}
\toprule
\textbf{Equation} & \textbf{Operator $\mathcal{L}$} &
                                                      \textbf{Order $s$} & 
\textbf{$L^p$ Range} & \textbf{Exponent $\gamma$} & \textbf{Theorem} \\
\midrule
      Std. BGK& $\M\rho[f]-f$& $1$ &$[1,\infty]$& $1/3$ &\ref{thm:mainBGKStable} $(a)$\\
      Frac. BGK & $\M^s\rho[f]-f$& $(1/2,1)$ &$[1,\infty]$& $1/3$ &\ref{thm:mainBGKStable} $(a)$\\
      L.O. Frac. BGK & $\M^s\rho[f]-f$& $(0,1/2]$ &$[1,p_{\max})$& $\bar\gamma$ &\ref{thm:mainBGKStable} $(b)$\\
      Gen. BGK& $F\rho[f]-f$& $[1/2,1]$ &$[1,\infty]$& $\min\left\{{\bar{\delta}}/{2s},{1}/{3}\right\}$ &\ref{thm: mainBGKgeneral}\\
      Gen. L.O. BGK& $F\rho[f]-f$& $(0,1/2)$ &$[1,p_{\max})$& $\min\left\{{\bar{\delta}}/{2s},\bar\gamma\right\}$ &\ref{thm: mainBGKgeneral}\\
      Nonlocal  KFP & $G*_vf-f+\div(vf)$& $1$& $[1,\infty]$ & $1/2^-$ &\ref{thm: mainNLFP}\\
      KFP & $\Delta_vf+\div(vf)$& $1$& $[1,\infty]$ & $1/2$ &\ref{thm:mainKFP}\\
      Frac.  KFP & $-(-\Delta_v)^sf+\div(vf)$& $(0,1)$& $[2,\infty]$ & $\nu/(2s)$ &\ref{thm:mainKFP}\\        
      Frac.  KFP & $-(-\Delta_v)^sf +\div(vf)$& $(0,1)$& $(1,2)$ & $\nu/s\left(1-1/p\right)$ &Rmk. \ref{rmk:interpolationPsmall}\\
        \bottomrule
    \end{tabular}
  \caption{
Summary of the results.  L.O.\ =low order; KFP = Kinetic Fokker--Planck. 
The rate denotes the exponent in the second--order decay estimate. 
Here
$\bar\gamma = \frac{1}{3}\!\left(1 -q(1-2s)_+\right)$, 
$p_{\max} = \frac{(1-2s)d}{(1-2s)d -2s}$, 
and $\bar\delta$ denotes the convergence rate in the Central Limit Theorem (see Assumption \ref{ass: BE}).
}
   \label{tab:placeholder}
\end{table}

\paragraph{Additional comments}

It is worth underlining that self-similar behaviour for the kinetic
equation \eqref{eq:general} is very natural and comes from the fact
that a combination of the free transport mechanism and the scattering
mechanism produces a long-time diffusion effect in the $x$ variable.
The fact that a kinetic description is appropriate for many physical
processes, and the fact that spatial diffusion is widely present
suggest that this effect should be indeed quite general.  We expect
this self-similar behaviour to hold much more generally. It is however
an open problem to show this under structural assumptions on $\L$
only, without assuming a specific form as we do here.

In the present paper, our method of proof is based on a (more or less
intricate) explicit representation of the solutions. For that reason,
our results are restricted to simple cases of the collisional operator
$\mathcal L$.  Nevertheless, as far as we know, the present paper is
the first to prove self-similar behaviour for solutions to such
kinetic equations even in these simple cases.

Let us emphasise that the results for case \eqref{eq:defL-nonlocalFP}
can be extended to a general centred probability distribution $J$ with
finite moment of order $2+\delta$, for some $\delta>0$, although we do
not consider this extension in this paper. For any such $J$, indeed,
the operator $J*_v f-f$ is a nonlocal analogue of the Laplacian (see
\cite{AndreuVaillo2010}), so that the operator $\mathcal{L}$ is truly
a nonlocal version of the Fokker--Planck operator.

In the special (and perhaps simplest) case of the standard
Fokker--Planck operator $\L f := \Delta_v f + {\rm div}_v(v f)$, the
solution to \eqref{eq:general} can be written fully explicitly.  In
this case the self-similarity result seems to be expected and it may
have been derived before, but we have not been able to find a
reference for this.

The Wild sum approach, indeed, cannot be directly applied to
\eqref{eq: KFP} since the Laplacian and the fractional operator are
not bounded operators and cannot be split into a positive and a
negative part. A possible alternative way to prove the previous
results is from a uniform-in-time localisation limit from the nonlocal
Fokker--Planck equation \eqref{eq: NLKFP}.
For this reason, we include this case  in Section \ref{sec: KFPandFKFP}.

\paragraph{Strategy of the proof}
The technique for cases \eqref{eq:defL-BGK} and
\eqref{eq:defL-nonlocalFP} is very similar and is at the heart of this
work.  Our method of proof is based on a Wild sum or Dyson series
representation of the solution together with a more or less subtle
exchange lemma (see Lemma~\ref{lem:exchange_BGK} and Lemma~\ref{lemma:
  exchangestar}).

For the original use of Wild sums in the context of the Boltzmann
equation, see \cite{wild_boltzmanns_1951, McKean1967}; more recently,
\cite{Carlen_Carvalho_Gabetta_2000} used a similar representation to
give information on the nonlinear, homogeneous Boltzmann equation for
Maxwell molecules.
In our case we write the solution in Wild sum form and, using a
connection between the latter and the distribution of the order
statistics of uniform random variables, we show that the events that
make $f(t,x,v)$ and $u(t,x) F(v)$ differ significantly are improbable
when $t\to \infty$.  This is made rigorous through the proof of
suitable concentration inequalities. In Section \ref{sec: wild} we
provide a more detailed description of the strategy employed.

As already remarked, even if case \eqref{eq:defL-fracFP} has an
explicit solution (for $s=1$, in standard variables; for $s<1$, in
Fourier variables), we have not been able to find a statement of the
above results in the existing literature. We derive our results for
this case using the explicit representation in Section \ref{sec:
  KFPandFKFP}.

\paragraph{Perspectives}

The rates, especially for case \eqref{eq:defL-BGK}, can be
improved. Indeed, a concentration inequality tighter than Chebyshev's
would lead to better rates. Moreover, the dependence of the exponent
on the fractional index remains to be addressed. We are also
interested in the analysis of the limiting case where the tails of the
equilibrium behave like $F(v)\sim |v|^{-d-2}$.

For case \eqref{eq:defL-nonlocalFP}, the natural extension is to
consider a general convolution kernel $J$, including heavy-tailed
examples.

For case \eqref{eq:defL-fracFP}, we would like to reach a second-order
estimate in the full range $p\in[1,\infty]$ (without interpolation),
possibly by exploiting the multivariate stable law structure and the
associated spectral measure analysis.  We believe it is also possible
to derive the result \eqref{eq:defL-fracFP} for $s=1$ from Theorem
\ref{thm: main} for case \eqref{eq:defL-nonlocalFP}, by adding an
additional scaling parameter $\eta$:
$$\frac{1}{\eta^{2}}\big(G_\eta*f-f\big)\to \Delta f.$$ and make the statement of
Theorem \ref{thm: main} uniform in $\eta$---see \cite{canizotassi2024}
for a similar result in the homogeneous case. Thus, taking the limit
as $\eta\to 0$ obtain our main result also for case
\eqref{eq:defL-fracFP}. In a similar way, following
\cite{tassi2026_fractional}, we expect that replacing $G$ in
\eqref{eq:defL-nonlocalFP} by a heavy-tailed kernel (e.g., stable laws
$G^s$) would allow one to derive \eqref{eq:defL-fracFP} from
\eqref{eq:defL-nonlocalFP} as a uniform-in-time nonlocal-to-local
limit also for the fractional case $s\in(0,1)$.

Finally, we would like to extend the method to other operators, the most natural candidate being the linear Boltzmann equation under the Maxwell molecules assumption.

 \paragraph{Plan of the paper}  Section \ref{sec: wild} is concerned with the stochastic interpretation of cases \eqref{eq:defL-BGK} and \eqref{eq:defL-nonlocalFP} and the associated Wild sum formulation derived from a perturbative semigroup argument. 
 Section~\ref{sec:BGK} is devoted to the asymptotic behaviour of the kinetic linear BGK equation in both the standard and fractional regimes: we first collect preliminary results regarding suitable concentration inequalities and perturbation results for stable and Gaussian distributions. We then present the proof of the main theorem for case \eqref{eq:defL-BGK} for stable equilibria and a sketch of the general equilibrium case. In Section \ref{sec: NLKFP}, we establish the main result  for the nonlocal Fokker--Planck operator \eqref{eq:defL-nonlocalFP} following a similar structure, with additional preliminary estimates on a ``twisted'' convolution arising from the interaction between the transport and the nonlocal term and different concentration inequalities.  Finally, Section \ref{sec: KFPandFKFP} addresses the Kinetic Fokker--Planck and Fractional Kinetic Fokker--Planck operators \eqref{eq:defL-fracFP}, via explicit calculations in Fourier variables, from which we recover the (believed) optimal rates for $s=1$ in the full range $p\in[1,\infty]$,  and for $p\in[2,\infty]$ in the fractional case $s\in(0,1).$

\section{Wild sum, stochastic interpretation, and general strategy}
\label{sec: wild}
We consider equations of the following two forms: either
\begin{equation}
  \label{eq:wild-form1}
  \partial_t f +  v \nabla_x f = c (\L^+ f - f),
\end{equation}
or
\begin{equation}
  \label{eq:wild-form2}
  \partial_t f +v \nabla_x f = c (\L^+ f - f + \div_v(v f)),
\end{equation}
both subject to an initial datum $f(0)=f_0$, and we discuss a
representation formula for the solution $f$ to each of these
equations.  More precisely, we write the solution to a kinetic
equation of the form (\ref{eq:wild-form1}) or (\ref{eq:wild-form2}) as
a \emph{Wild sum}, also known as a \emph{Dyson series} in the context
of semigroup theory \cite[III Thm. 1.10]{EngelNagel2001}.  For this,
we define our transport semigroup $T_t$ by
\begin{equation*}
 T_t  g(x,v) := g\left(x-tv,v\right)
\end{equation*}
in the case of an equation of the form (\ref{eq:wild-form1}), or
\begin{equation*}
 T_t  g(x,v) := e^{c d t}
 g\left(x- \frac{e^{c t} - 1}{c} v, e^{c t} v\right)
\end{equation*}
in the case of an equation of the form (\ref{eq:wild-form2}). These
are, respectively, the semigroups associated with the evolution PDEs
$\partial_t f +  v \nabla_x f = 0$ and
$\partial_t f +  v \nabla_x f = c \div_v(v f)$.  The mild
solution of equation (\ref{eq:wild-form1}) or (\ref{eq:wild-form2}),
which is nothing but Duhamel's formula, is
\begin{equation}\label{eq: mild sol}
  f(t)=e^{-ct}T_tf_0
  +c \int_0^t
  e^{-c({t-\tau})}T_{t-\tau}(\mathcal{L}^+f(\tau) )\d \tau
\end{equation}
Iterating the formula recursively gives the following representation
of the solution:
\begin{teorema}[Wild sum representation]
  The solution to either (\ref{eq:wild-form1}) or
  (\ref{eq:wild-form2}) can be written as
  \begin{equation}
    \label{eq:wild2}
      f(t)=
      e^{-c t} \sum_{n\ge
        0}c^n\int_0^t\int_0^{t_n} \cdots \int_0^{t_2}
      T_{t-t_n} \, \L^+ \, T_{t_n-t_{n-1}}
      \cdots
      T_{t_{2}-t_1} \, \L^+ \, T_{t_1} f_0
      \ \d t_1 \cdots \d t_n.
    \end{equation}
    In this expression it is understood that the term for $n = 0$ does
    not contain any integrals and is equal to $T_t  f_0$.
\end{teorema}

\begin{osservazione}    
In both cases \eqref{eq:wild-form1} and \eqref{eq:wild-form2}, $\L^+$ is a positive operator, and they are always posed
for $(x,v) \in \R^d \times \R^d$. The main restrictions here, apart
from the operators being linear, are that the collision part needs to
be split into a positive and a a negative part, with the negative part
being $-f$ (this is the analogue to Maxwellian molecules in the
Boltzmann case); and that if a drift is present it must be $\div(v
f)$, associated to the standard quadratic potential.
\end{osservazione}

This representation has an illuminating probabilistic interpretation,
which is relatively well-known and can be understood from the above
Wild sum expression. Consider the following stochastic process
$(X_t,V_t)$ on $\R^d \times \R^d$: a particle randomly starts at
$(X_0, V_0)$ with probability distribution $f_0$, and carries an
exponential clock of a given intensity $c$. The particle moves by
deterministic free transport until the clock rings; when the clock
rings, the particle jumps randomly according to the kernel of the
operator $\L^+$, and the clock is restarted. That is, between two
rings of the clock, the random variables $X_t, V_t$ satisfy either
\begin{equation*}
  d X_t = V_t \d t,
  \qquad
  d V_t = 0
\end{equation*}
in the case of equation~(\ref{eq:wild-form1}), or
\begin{equation*}
  d X_t =V_t \d t,
  \qquad
  d V_t = - c V_t \d t
\end{equation*}
in the case of equation~(\ref{eq:wild-form2}). \emph{The solution
  $f = f(t)$ to equation (\ref{eq:wild-form1}) or
  (\ref{eq:wild-form2}) at time $t \geq 0$ is then the law of the
  stochastic process $(X_t, V_t)$ at time $t$.}

Denote by $\Phi_n$ the $n$-th term in the sum (\ref{eq:wild2}). That is, 
\begin{equation}\label{eq:defPhi_n}
\begin{cases}
\Phi_0:=T_tf_0\\
  \Phi_n:= \Phi_n(t,x,v; t_1,\dots,t_n) :=
  T_{t-t_n} \, \L^+ \, T_{t_n-t_{n-1}}
  \cdots
  T_{t_{2}-t_1} \, \L^+ \, T_{t_1} f_0\qquad n\ge 1.
  \end{cases}
\end{equation}
Then $\Phi_n$ represents the probability distribution of the process
$(X,V)$ at time $t$, provided exactly $n$ random jumps (clock rings)
happened at the times $0 \leq t_1 \leq \dots \leq t_n \leq t$.
Formally, the number of random jumps before time $t$, follows a
Poisson distribution of parameter $ct$.  Given $n$, the distribution of
the jump times $t_1, \dots, t_n$ can be obtained by drawing $n$ random
points from the uniform distribution on $[0,t]$ and ordering them. We
thus want to write $f(t)$ as an expectation of a suitable random
variable with respect to some measure $\mu$. The transition from the
deterministic integral expansion \eqref{eq:wild2} to the expectation
that we show in \eqref{eq:wild2-expectation} can be made rigorous.

Let us take $c=1$ to ease the notation.  We denote by  $\mathbf{x}=(x_1,\dots, x_{n+1})\in \R^{n+1}$. The following definition comes to our help:
\begin{definizione}\label{def:simplex}
\begin{enumerate}[(i)]
\item The ordered simplex of size $z > 0$ is defined as
  $\mathcal{S}_n(z) := \{ (t_1, \dots, t_n) \in \mathbb{R}^n : 0 < t_1
  < \dots < t_n < z \}$.
\item We define the gap simplex (or just simplex) of size $z$ as the
  $n$-dimensional surface embedded in $\R^{n+1}$
  $$\Delta_{n}(z):=
  \Big\{\mathbf{x} = (x_1, \dots, x_{n+1})\in \R^{n+1}
  \colon
  \text{$x_i \geq 0$ for all $i$,}\quad \sum_{i=1}^{n+1} x_i=z\Big\}$$
\end{enumerate}
\end{definizione}
The $x_i$ can be considered as the gaps between $n$ points in the interval $[0,z]$.

 \begin{definizione}\label{def: orderANDdirich}  
  \begin{itemize}
  \item Let $U_i\sim\text{Unif}\,[0,1]$, $i=1,\dots, n$ be independent
    random variables. Their order statistics, denoted by
    $U_{(1)} \le U_{(2)} \le \dots \le U_{(n)}$, are the random
    variables sorted from the smallest to the largest, e.g.,
    $U_{(1)}=\min_i{U_i}$.  The vector
    $\mathbf{U}=(U_{(1)},\dots U_{(n)})$ takes values in
    $\mathcal{S}_n(1)$
  \item A random vector $\mathbf{X}=(X_1,\dots, X_{n+1})$ is said to
    follow a (symmetric flat) Dirichlet distribution of parameter
    $(1,\dots, 1)$ if it is uniformly distributed on the unit simplex
    $\Delta_{n}(1)$.

    Equivalently, $\mathbf{X}$ can be described as the successive gaps
    between the order statistics. Specifically, define $ U_{(0)} = 0$
    and $ U_{(n+1)}= 1$, and let $( U_{(i)})_{i=1,\dots, n}$ be the
    order statistics previously introduced. Then,
  $$X_i =  U_{(i)}- U_{(i-1)} \quad \text{for } i=1, \dots, n+1.$$ 

  $\mathbf{X}$ takes values in $\Delta_n(1)$.
\end{itemize}
\end{definizione}

We construct an explicit measure space that captures all possible
collision histories up to time $t$. We define the space $\Omega_t$ as
\begin{equation*}
  \Omega_t := \{\omega=(n, t_1, \dots, t_n) \mid
  0 \leq n \in \N,
  \
  0 < t_1 < \dots < t_n < t \}
\end{equation*}
For any continuous and bounded function $\f: \Omega_t \to \R$, we
define the probability measure $\mu$ on $\Omega_t$ by
\begin{equation*}
  \int_{\Omega_t} \f \d \mu
  :=
  e^{- t} \sum_{n\ge
    0}\int_{\mathcal{S}_n(t)} \f(n,t_1,\dots, t_n)\d t_1 \cdots \d t_n
\end{equation*}
Since the volume of $\mathcal{S}_n(t)$ is $\dfrac{t^n}{n!}$, it is
clear that $\mu(\Omega_t)=1$, making $(\Omega_t,\mu)$ a probability
space.  For each $\omega\in\Omega_t$, the random variable $N_t$ is the
projection map
$$N_t:\Omega_t\to\N\qquad\text{such that}\qquad N_t(\omega)= n,$$
representing the number of collisions up to time $t$.  In particular,
under the measure $\mu$, $N_t$ follows a Poisson distribution of
intensity $1$, i.e., the probability that exactly $n$ jumps happen in
the time interval $[0,t]$ is
$$
\P(N_t=n)=e^{-t}\frac{t^n}{n!}.
$$
For each $\omega\in\Omega_t$, the random variable $U_{(i)}$ is the projection map
$$
U_{(i)}:\Omega_t\to [0,t]\qquad \text{such that} \qquad U_{(i)}(\omega) := t_i,
$$
representing the $i$-th collision time.
Given $N_t=n$, the random variables $U_{(1)}, \dots, U_{(n)}$ behave  as the order statistics sampled from $n$ independent uniform distributions on $[0,t]$ and the joint probability density function is 
$$\rho_{\mathbf{U} | N_t=n}(t_1,\dots, t_n) = \frac{n!}{t^n} \mathds{1}_{\mathcal{S}_n(t)}(t_1,\dots, t_n).$$

Similarly, assuming $t_0=0$ and $t_{N_t+1}=t$, the random variable
$X_i$ represents the $i$-th waiting time between two jumps and
$$X_i: {\Omega_t}\to[0,t]\qquad \text{such that}\qquad X_i(\omega)=x_i:= t_i -t_{i-1}.$$

Following Definition \ref{def: orderANDdirich}, conditional on $N_t=n$, the scaled variables $\frac{1}{t}X_1, \dots, \frac{1}{t}X_{n+1}$  follow a flat Dirichlet distribution, which implies
$$\rho_{\mathbf{X} | N_t=n}(x_1,\dots x_n,x_{n+1}) = \frac{n!}{t^n} \mathds{1}_{\Delta_n(t)}(x_1,\dots,x_{n+1}).$$

We rewrite the term in the Wild sum by normalising the integral: 
\begin{equation*}
    \int_{\mathcal{S}_n(t)} 
    \Phi_n  \d t_1 \dots \d t_n 
    = \frac{t^n}{n!} \int_{\mathcal{S}_n(t)}  \frac{n!}{t^n}\Phi_n  \d t_1 \dots \d t_n 
    = \frac{t^n}{n!} \, \mathbb{E}\big[ \Phi_n \mid N_t = n \big].
\end{equation*}
Thus recalling the factor $e^{-t}$, summing over all possible values
of $n$, by the total probability formula one has
\begin{equation}\label{eq:wild2-expectation}
    f(t) = \sum_{n\ge 0} \P(N_t=n) \E[\Phi_{N_t} \mid N_t=n] = \E[\Phi_{N_t}],
\end{equation}
where the expectation is taken according to $\mu$. This expression is
equivalent to (\ref{eq:wild2}), but it includes more cleanly both the
sum and the integrals, and clearly shows that $f(t)$ is an averaging
of the values of $\Phi_n$ over all possible collision histories.

In order to carry out some calculations we will also use the
expression
\begin{equation}
  \label{eq:wild3}
  f(t) = \int_{\Omega_t} \Phi_n (t, x, v; t_1, \dots, t_n) \d \mu(n, t_1,
  \dots, t_n)
  \equiv \int_{\Omega_t} \Phi_n  \d \mu.
\end{equation}

Our proofs rely on the following observations:
\begin{itemize}
\item By introducing a time-dependent operator
  $\tilde{\mathcal{L}}^+_t$, defined by
  $\mathcal{L}^+ T_t =T_t\tilde{\mathcal{L}}^+_t $, we can ``push''
  the transport operators to the left and combine them thanks to the
  semigroup property. This yields
  $$
  f(t)=e^{-ct}T_t\Big[\sum_{n\ge
    0}c^n\int_0^t\int_0^{t_n} \cdots \int_0^{t_2}
  \, \tilde{\mathcal{L}}_{t-t_n}^+ 
  \cdots
  \tilde{\mathcal{L}}_{t_2-t_1}^+\tilde{\mathcal{L}}_{t_1}^+ \,f_0
  \ \d t_1 \cdots \d t_n\Big],
  $$
  allowing us to analyse
  $\tilde{\mathcal{L}}_{t-t_n}\tilde{\mathcal{L}}_{t_n-t_{n-1}}\dots
  \tilde{\mathcal{L}}_{t_1}$
\item Since $f(t)$ is an average over all jump histories, we can
  neglect highly improbable paths, and estimate these probabilities
  thanks to suitable concentration inequalities (see Subsections
  \ref{sec: dirichlet} and \ref{sec: concIneq}).  We emphasise a
  crucial feature: for kinetic equations it may happen (and it
  actually does happen: e.g., in case \eqref{eq:defL-nonlocalFP} mass
  may concentrate almost entirely along ``typical trajectories'' in
  one variable, while the fluctuations in the other remain
  non-negligible.
\end{itemize}
As already mentioned in the introduction, the idea of Wild sums goes
back to Wild and McKean, and was further developed e.g. by Carlen,
Carvalho, and Gabetta in the homogeneous setting
\cite{Carlen_Carvalho_Gabetta_2000}.  The addition of free transport
and possibly a drift, creates a mixing effect between the variables
--- particularly clear in the case \eqref{eq:defL-nonlocalFP}---making
the analysis significantly harder.

\section{The kinetic linear BGK equation}
\label{sec:BGK}
In this section, we consider the linear relaxation BGK equation given by:
\begin{equation}
  \label{eq: eBGK}
  \partial_t f+v\cdot \nabla_x f=\mathcal{L^+}f-f
\end{equation}
where $f = f(t,x,v)$, the equation is posed for $t \geq 0$, $x,v \in
\R^d$, and the operator $\L^+$ is given by
$$
\mathcal{L^+}f (x,v) :=  F(v)\rho[f](x),
\qquad x,v \in \R^d.
$$
Here $F$ denotes the  velocity equilibrium, i.e. the probability distribution  which satisfies $\mathcal{L}F=0$, and $\rho[f]$ denotes the macroscopic density
$$
\rho[f]:=\ird f(t,\cdot,w)\d w.
$$
We provide the proof of the main Theorem \ref{thm: main} in the case
where the local equilibrium $F$ is either
\begin{enumerate}[(i)]
    \item a Maxwellian $\mathcal{M}$, where $$\M(v) = (2\pi)^{-d/2}e^{-|v|^2/2},$$
   
    \item a stable law $\mathcal{M}^s$, defined by the Fourier
      transform
      $$\widehat{\mathcal{M}}^s(\eta) = e^{-|\eta|^{2s}/(2s)},$$
    \end{enumerate}
    and give as well in Section \ref{sec: FNOstable} an outline of the following case:
    \begin{enumerate}[(i)]
      \setcounter{enumi}{2}
    \item $F$ is in the basin of attraction for the generalised
      Central Limit Theorem of either a Maxwellian or a stable law in
      the sense of Assumption \ref{ass: BE}.
\end{enumerate}

The main theorem of this section, which is a particular case of
Theorem \ref{thm: main}, is the following:
\begin{teorema}\label{thm:mainBGKStable}
Let $s\in(0,1]$ and $f_0:\R^d\times\R^d\to\R^+$ be a probability initial density (the initial datum), such that $$M_\nu:=\ird\ird f_0(x,w)|w|^\nu\d x \d w<\infty,$$
with $\nu=1$ if $s\in(1/2,1]$ and $\nu\in(s,2s)$ if $s\in(0,1/2]$.
Let $u$  be the solution of the (fractional) heat equation \eqref{eq: fracheat}  with  diffusivity constant $\kappa=\Gamma(2s)$ and initial data $u_0(x):= \ird f_0(x,w) \d w$. 
\begin{enumerate}[(a)]
    \item For $s\in(1/2,1]$, let  $p\in[1,\infty]$, $f_0\in L^p_{x,v}\cap L_v^1L_x^p$, and $f$ be a solution of \eqref{eq: GBGK} with initial data $f_0$.
Then,  there exists a constant $C > 0$
  depending only on $\norm{f_0}_{L^p}$, $\norm{f_0}_{L^1_vL^p_x}$, $d$, $s$, $p$ and $M_1$  such that
    $$
 \norm{f(t)-\M^s(v)u(t,x)}_{L^p}\le Ct^{-q}t^{-\frac13},
$$
\item For $s\in(0,1/2]$: 
let
 \begin{itemize}
       \item $p \in [1, p_{\max})$ if $s \in \left(0, \frac{d}{2(d+1)}\right]$
            \item $p \in [1, \infty]$ otherwise,       
        \end{itemize}
let $f_0\in L^p_{x,v}\cap L_v^1L_x^p$ and $f$
be a solution of \eqref{eq: GBGK} with initial data $f_0$.
Then, there exists a constant $C > 0$
  depending on $d$,$s$, $\nu$, $p$, $M_\nu$, $\norm{f_0}_{L^p}$ and $\norm{f_0}_{L_v^1L_x^p}$,  such that
    $$
 \norm{f(t)-\M^s(v)u(t,x)}_{L^p}\le Ct^{-q}t^{-\bar\gamma},
$$
where $\bar\gamma=\frac13(1-q(1-2s)_+)$.
Notice in particular that if $p=1$, then $\bar\gamma=\frac13$ for all $s\in(0,1]$.
\end{enumerate}
\end{teorema}

\begin{osservazione}
  One can relax the hypothesis on the $\nu$-moment up to
  $\nu>\frac{2}{3}s(1-q(1-2s)_+)$, and, if we allow a slower decay exponent
  $\bar\gamma<\gamma$, up to $\nu>2s\bar\gamma$.
\end{osservazione}

We start by presenting some preliminary results and computations

\subsection{Preliminary results}

A prominent role in the proof of the main theorem will be played by Dirichlet random variables.
In particular we will use the following well-known lemma:

\begin{lemma}\label{lemma: dirichlet_equivalence}
Let $\Delta_n(z)$ be the $n$-dimensional simplex of size $z>0$ introduced in Definition \ref{def:simplex} $(ii)$. Let $\mathbf{x}\in \Delta_n(z)$ and take $\Psi: \Delta_n(z) \to \R$ a bounded measurable function. Then,
\begin{equation}
\int_{\Delta_n(z)} \Psi(\mathbf{x}) \d \mathbf{x}= \frac{z^n}{n!} \mathbb{E}\left[ \Psi(z\mathbf{X}) \right]\end{equation}
where $\mathbf{X} = (X_1, \dots, X_{n+1}) \sim \mathrm{Dir}(1, \dots, 1)$ is a random vector following the flat Dirichlet distribution on the unit simplex.
\end{lemma}
The proof is straightforward but we include it for completeness.
\begin{proof}
    The volume of $\Delta_n(z)$ is $\frac{z^n}{n!}$.
    Since $\mathbf{X}$ is uniformly distributed on the simplex, by definition of the  expected value
    $$
    \E(\Psi(z\mathbf{X}))=\frac{1}{\textrm{Vol}(\Delta_n(z))}\int_{\Delta_n(z)}\Psi(\mathbf{x})\d \mathbf{x}.
    $$
\end{proof}
\subsubsection{Concentration inequalities}\label{sec: dirichlet}

We start by collecting properties of the symmetric Dirichlet distribution, which describes the statistical behaviour of the time gaps between collisions. 
Many of the properties of the Dirichlet distribution can be found in \cite[Ch. 49]{ContMultDistr}.

\begin{proposizione}\label{prop: tail dirichlet_fractional}
    Let $\mathbf{X}=(X_1,\dots, X_n, X_{n+1})\sim\Dir(1,\dots, 1,1)$ and define the random variable $S_{n,s}=\sum_{i=1}^{n+1}X_i^{2s}.$ 
    Let us define the mean of $S_{n,s}$ as
    $$\mathfrak{m}_{n,s}=\Gamma(2s+1)\frac{\Gamma((n+1)+1)}{\Gamma(n+2s+1)}.$$Then, for all $\beta \in (0, 1/2)$, there exists a constant $C_s > 0$ such that:
    \begin{equation*}
        \P\left( \left| S_{n,s} - \mathfrak{m}_{n,s} \right| \ge n^{1-2s-\beta} \right) \le \frac{C_s}{(n+2)^{1-2\beta}}.
    \end{equation*}
    In particular, if $s=1$,
    $$\P\left(\abs{S_{n,1}-\frac{2}{n+2}} \ge \frac{2}{n^{1+\beta}}\right) \le \frac{C_1}{(n+2)^{1-2\beta}}$$
\end{proposizione}

\begin{proof}

The result is obtained by applying Chebyshev's inequality:
    \begin{equation}\label{eq:cheby} \P(\abs{S_{n,s}-\E(S_{n,s})}>\lambda)\le \frac{\Var(S_{n,s})}{\lambda^2} \end{equation}
    For a Dirichlet distribution with parameters
    $(\alpha_1,\dots, \alpha_{n+1})$ and
    $\alpha_0=\sum_{i=1}^{n+1}\alpha_i$ (see \cite[Ch. 49,
    Sec. 2]{ContMultDistr}), the moments are given by:
    $$ \E(X_i^r)=\frac{\Gamma(\alpha_i+r)\Gamma(\alpha_0)}{\Gamma(\alpha_i)\Gamma(\alpha_0+r)}\qquad \text{for } r>-\alpha_i,$$
and the cross-moments are given by 
$$
\E(X_i^rX_j^q)=\frac{\Gamma(\alpha_i+r)\Gamma(\alpha_j+q)\Gamma(\alpha_0)}{\Gamma(\alpha_i)\Gamma(\alpha_j)\Gamma(\alpha_0+r+q)}\qquad \text{for }r>-\alpha_i, \;q>-\alpha_j, \,i\ne j.
$$
    For the flat symmetric Dirichlet distribution ($\alpha_i=1, \alpha_0=n+1$), this yields:
    $$ \E(X_i^r)=n!\frac{\Gamma(r+1)}{\Gamma(n+r+1)} $$
    and    
$$
\E(X_i^rX_j^q)=n!\frac{\Gamma(r+1)\Gamma(q+1)}{\Gamma(n+r+q+1)}\qquad i\ne j
$$
    Hence,  we compute the mean:
    $$ \E(S_{n,s})=\sum_{i=1}^{n+1}\E(X_i^{2s})=\Gamma(2s+1)\frac{\Gamma(n+2)}{\Gamma(n+2s+1)} $$
     To bound the variance, we use the known property that the components of a Dirichlet distribution are negatively associated. Since $x \mapsto x^{2s}$ is an increasing function for $x \ge 0$, we have $\Cov(X_i^{2s}, X_j^{2s}) \le 0$ for $i \neq j$. Thus, we can  drop the covariance terms; while doing so yields a looser constant, using the cross-moments formula above, one can check that both $\Var(S_{n,s})$ and $\sum \Var(X_i^{2s})$ have the same asymptotic order. For the case $s=1,$ the computations are easier and one could include the covariance term to obtain the sharp constant $C_1=1$.

     For our analysis, however, we are satisfied with the  following upper bound:
    \begin{align*}		
        \Var(S_{n,s}) &= \sum_{i=1}^{n+1}\Var(X_i^{2s}) + \sum_{i=1}^{n+1}\sum_{j\ne i}\Cov(X_i^{2s},X_j^{2s}) \\
        &\le \sum_{i=1}^{n+1}\Var(X_i^{2s}) \\
        &= (n+1)\left[ n!\frac{\Gamma(4s+1)}{\Gamma(n+4s+1)} - \left(n!\frac{\Gamma(2s+1)}{\Gamma(n+2s+1)}\right)^2 \right]
    \end{align*}
We use the following asymptotic expansion 
        $$
        \frac{\Gamma(n+a)}{\Gamma(n+b)}\sim n^{a-b}\Big(1+ \frac{(a-b)(a+b-1)}{2n}+O(n^{-2})\Big).
        $$
        
        This implies that the order of the variance is  {$n^{1-4s}$}, that is there exists $c>1$ such that 
        $$
        \Var(S_{n,s})\le c^2(n+2)^{1-4s}
        $$
	We finally apply Chebyshev's inequality \eqref{eq:cheby} with $\lambda=n^{1-2s-\beta}$:
	$$	\P\left(|S_{n,s}-\mathfrak{m}_{n,s}|>n^{(1-2s-\beta)}\right)\le \frac{\Var(S_{n,s})}{\left(n^{(1-2s-\beta)}\right)^2}\le \frac{c^2(n+2)^{1-4s}}{n^{2-4s-2\beta}}=C_s(n+2)^{-(1-2\beta)}.
	$$
\end{proof}

 Let $\beta\in(0,1/2)$ to be fixed later. We define $\mathscr{E}_n (z)$ as the subset of the simplex
  \begin{equation}\label{eq: defE2n}
\mathscr{E}_n (z):=\left\{\mathbf{x}\in \Delta_{n}(z) : \abs{|\mathbf{x}|^2-\frac{2z^2}{n+2}} \le \frac{1}{n^{\beta}} \frac{2z^2}{n} \right\}  
\end{equation} 
Note that for large $n$, this condition requires $|\mathbf{x}|^2$ to stay within a relative error of $n^{-\beta}$ from its expected value.
An analogous definition for the fractional case reads:
   \begin{equation}
        \label{eq:defE_stable}
        \mathscr{E}_{n,s}(z) := \left\{ \mathbf{x} \in \Delta_n(z) : \left| \sum_{i=1}^{n+1} x_i^{2s} - \dfrac{\Gamma(2s+1)z^{2s}}{(n+2)^{2s-1}} \right| \le \frac{1}{n^{\beta}}\dfrac{\Gamma(2s+1)z^{2s}}{(n+2)^{2s-1}} \right\},
    \end{equation}

\begin{corollario}\label{coroll: dirichlet_frac}
Since  $\mathfrak{m}_{n,s}\sim \Gamma(2s+1)n^{1-2s}$, with relative error of $n^{-1}$, one has that the volume of the complement $\mathscr{E}^c_{n,s}(z)$ is bounded by
    \begin{equation*}
        \int_{\Delta_{n}(z)} \mathds{1}_{\mathscr{E}_{n,s}^c(z)}(\mathbf{x}) \, \d \mathbf{x} 
        \le \frac{z^n}{n!} \frac{C_s}{(n+2)^{1-2\beta}}.
        \end{equation*}
\end{corollario}

We now gather some standard results regarding concentration for Poisson distributed random variables.
\begin{proposizione}\label{prop: boundtailPoiss}
Let $\lambda \ge \lambda_0 > 0$ and let $X$ be a random variable $X\sim \mathrm{Pois}(\lambda)$
For all $\gamma \in \R$, there exists a positive constant $C_\gamma$ (depending on $\gamma$ and $\lambda_0$) 
such that
\begin{equation}\label{eq:poisson2}
\mathbb{E}[X^\gamma \mathds{1}_{X \ge 1}] = e^{-\lambda}\sum_{n\ge 1}^\infty \frac{\lambda^n}{n!} n^\gamma \le C_\gamma\lambda^\gamma.\end{equation}\end{proposizione}

\begin{proof}
Let $X$ be a  random variable such that $X\sim \pois(\lambda)$.

\medskip \noindent
  \textit{Case 1: $\gamma\le 0$}.
When $\lambda\le 2$,  we can easily bound 
$n^{\gamma}\le \lambda^{\gamma}$ and obtain the claim. For larger $\lambda$ we proceed as follows.
Fix now a (small) positive number $\delta$ such that $\lambda\ge\frac{1}{1-\delta}$ and assume $\lambda(1-\delta)$ is an integer.
Then 
$$
 e^{-\lambda}\sum_{n\ge1}\frac{\lambda^n}{n!}n^{\gamma}=e^{-\lambda}\sum_{n=1}^{(1-\delta)\lambda}\frac{\lambda^n}{n!}n^{\gamma}+e^{-\lambda}\sum_{n\ge (1-\delta)\lambda+1}\frac{\lambda^n}{n!}n^{\gamma}:=I_1+I_2
$$
The second term can be easily bounded: since $n\ge(1-\delta)\lambda+1$, then $n^{\gamma}\le C\lambda^{\gamma}$ and therefore
$$I_2\le C\lambda^{\gamma}\P(X\ge (1-\delta)\lambda)\le C\lambda^{\gamma}
$$
To bound the first term, we  use a Chernoff-type bound in the following way:
\begin{multline*}
    I_1=e^{-\lambda}\sum_{n=1}^{(1-\delta)\lambda}\frac{\lambda^n}{n!}n^{\gamma}\le e^{-\lambda}\sum_{n=1}^{(1-\delta)\lambda}\frac{\lambda^n}{n!}=\P(1\le X\le (1-\delta)\lambda)<\P(X\le (1-\delta)\lambda)
    \\
    \le \P(e^{zX}\ge e^{z(1-\delta)\lambda})\le \frac{\E(e^{zX})}{e^{z(1-\delta)\lambda}}=e^{\lambda[e^z-1-(1-\delta)z]}
\end{multline*}
for all $z\le 0$. We have used the fact that for a random variable $X$, $\E(e^{zX})$ is the moment generating function which, in the case of the Poisson distribution, is $e^{\lambda(e^{z}-1)}$.

Now we choose $z=\log(1-\delta)$ in order to obtain
$$
I_1\le e^{-\lambda[\delta+(1-\delta)\log(1-\delta)]}:=e^{-a\lambda}
$$
since for $\delta\in(0,1)$, $\delta+(1-\delta)\log(1-\delta)>0$ (it is easy to see that is larger than or equal to $\delta^2/2$.)

Considering both estimates, since the decay of $I_1$ is exponential, and thus much faster, we conclude the proof for this range of $\gamma$.

\medskip \noindent
  \textit{Case 2: $\gamma\in(0,1)$}
Since the function $x\to x^\gamma$ is concave, Jensen's inequality gives
$$
\E(X^\gamma\mathds{1}_{X \ge 1})\le (\E(X))^\gamma=\lambda^\gamma
$$
and the inequality is satisfied with $C=1$.

  \medskip \noindent
  \textit{Case 3: $\gamma \ge 1$}. Assume without loss of generality that $\lambda_0=1$

 \begin{itemize}
  \item If $\gamma\in \N$, then we have the well-known representation via Stirling numbers of the second kind  $S(n,k)$--- see e.g. \cite{comtet_advanced_1974}
  $$
  \E(X^\gamma)=\sum_{k= 0}^\gamma S(\gamma, k)\lambda^k
  $$
  Since $\lambda \ge \lambda_0=1$, the sum is dominated by its highest degree term
$$
\E(X^\gamma)=\sum_{k= 0}^\gamma S(\gamma, k)\lambda^k\le\lambda^\gamma\sum_{k=0}^\gamma S(\gamma, k)=B_\gamma\lambda^\gamma$$
where $B_\gamma=\sum_{k=0}^\gamma S(\gamma,k)$ is the $\gamma$-th Bell number and is known to be bounded by $\left(\frac{c\gamma}{\log( \gamma+1)}\right)^\gamma$, which is obviously a finite number.

\item If $\gamma\notin\N$, choose any integer $m>\gamma$. H\"older's inequality implies that for $1<\gamma<m<\infty$
$$
\E(|X|^\gamma)\le \E(|X|^{m})^{\gamma/m}\le 
(C_{m}\lambda^{m})^{\gamma/m}\le C_\gamma\lambda^\gamma
$$
\end{itemize}
 
\end{proof}

A second (similar) result is given by the following Proposition, whose proof is again based on the idea of splitting  in two two regions: a central (probable) bulk  region and an exponential decaying tail region.

\begin{proposizione}\label{prop: poiss2}
    Let $s\in(0,1],$ and $\M^s$ be a standard $2s$ stable law.
	For $\lambda \ge 1$, let $X \sim \pois(\lambda)$. Then, there exists a constant $C > 0$ depending only on $p,s$, and the dimension $d$ such that
    $$
	\E\left(\norm{\M^s -\M^s_{\left(\frac{\lambda}{X}\right)^{1-\frac{1}{2s}}}}_{L^p};X\ge 1\right)\le C \left(
    \frac{1}{\sqrt{\lambda}} + \frac{1}{\lambda}
    \right).
	$$
\end{proposizione}
\begin{proof}
    Set 
    $h_p(x):=\norm{\M^s -\M^s_{\left(\frac{\lambda}{x}\right)^{1-\frac{1}{2s}}}}_{L^p}.$ 
    
    We    have
    $$
	\E(h_p(X);X\ge 1)=e^{-\lambda}\sum_{n\ge 1}\frac{\lambda^n}{n!}h_p(n)
    $$
	Let $\delta \in (0,1/2]$ (the proof actually works for $\delta = 1/3$). We  divide this sum into two regions: the ``probable'' one
    $$A_1 := \{ n \in \N \colon |n - \lambda| \leq \delta \lambda \}$$
    and the ``improbable'' one
    $$A_2 := \{ n \in \N \colon |n - \lambda| > \delta \lambda \}.$$

Using a Chernoff-type argument (similar to Proposition \ref{prop: boundtailPoiss}), we bound the tails, observing that the polynomial growth of the norm is dominated by the exponential decay of the Poisson tails.
 Thus, for some constants $\tilde{C}_{p,d}, c_\delta > 0$:
    \begin{equation}
        e^{-\lambda}\sum_{n \in A_2} \frac{\lambda^n}{n!} h_p(n) 
        \le \tilde{C}_{p,d} e^{-c_\delta \lambda}.
    \end{equation}    
	On the other hand, we apply Lemma \ref{lemma: diststable},  with $\lambda_1 := \left(\frac{\lambda}{n}\right)^{1-\frac{1}{2s}}$, $\lambda_2 := 1$ and then the mean-value theorem to obtain
	$$
	h_p(n)\le C_1\abs{1-\left(\frac{\lambda}{n}\right)^{2s-1}}\le C_s\abs{1-\frac{\lambda}{n}}.
	$$
	 Performing a Taylor expansion in $n$ close to $n=\lambda$ we have, for some $\xi \in [\lambda, n]$,
    \begin{equation*}
        \frac{\lambda}{n} - 1
        =
        -\frac{n-\lambda}{\lambda}
        + \frac{\lambda (n-\lambda)^2}{\xi^3}.
    \end{equation*}
    Now since in the probable region $n \geq (1-\delta) \lambda$, we also have $\xi \geq (1-\delta) \lambda$, so
    \begin{equation*}
        \left| \frac{\lambda}{n} - 1 \right|
        \leq
        \frac{|n-\lambda|}{\lambda}
        + \frac{(n-\lambda)^2}{(1-\delta)^3 \lambda^2}.
    \end{equation*}
    Using this,
	$$
	\sum_{n \in A_1} e^{-\lambda}\frac{\lambda^n}{n!}\Big|\frac{\lambda}{n}-1\Big|
    \le
    \sum_{n \in A_1} e^{-\lambda}\frac{\lambda^n}{n!}\Big(\frac{|n-\lambda|}{\lambda}
    +
    \frac{(n-\lambda)^2}{(1-\delta)^3\lambda^2}\Big)\le \frac{\E(|X-\lambda|)}{\lambda}+\frac{\Var(X)}{(1-\delta)^3\lambda^2}\le \frac{1}{\sqrt{\lambda}}+\frac{C}{\lambda}
	$$
	where in the last inequality we used the Cauchy--Schwarz inequality to bound $\E(|X-\lambda|)\le \sqrt{\Var(X)}=\sqrt{\lambda}$.
    We conclude by gathering the two estimates of the the probable and improbable regions.
\end{proof}
\begin{osservazione}\label{rmk:PoissonDependence}
 Let us notice that neither the exponent of the factor 
 $(\lambda/n)$ nor the exact structure of the stable law plays a significant role here; what matters is only that the scaling is correct. This bound, indeed, comes from structural properties of the Poisson distribution.
\end{osservazione}
\subsubsection{Perturbations of Stable laws}
A first trivial result is the following:
\begin{lemma}
 \label{lem:Lp-norm-of-stable}
 For any $1 \leq p \leq +\infty$ and $\lambda > 0$,
 \begin{equation*}
   \| \M^s_\lambda \|_p = C \lambda^{-d (1 - 1/p)},
 \end{equation*}
 where $C = C(p,d,s) > 0$ is an (explicitly computable) constant.
\end{lemma}

\begin{lemma}\label{lemma: diststable}
    Let $\M^s$ be the standard stable law of stability parameter $2s$, and, as usual, denote $\M^s_\lambda(x)=\lambda^{-d}\M^s\big(\frac{x}{\lambda}\big)$.    
    Then, for $\lambda_1,\lambda_2> 0$,
   \begin{equation}\label{eq: diststable2}
  	\norm{\M^s_{\lambda_1}-\M^s_{\lambda_2}}_{L^p}\le C_{d,p} \abs{\lambda_2-\lambda_1}\max\{\lambda_1^{-2sq-1},\lambda_2^{-2sq-1}\} .
  \end{equation}
  Consequently,
  if additionally   $\lambda_1,\lambda_2$ are such that $\abs{1-\frac{\lambda_1^{2s}}{\lambda_2^{2s}}}\le \frac{1}{2}$, we have
     \begin{equation}\label{eq: diststable3}
  	\norm{\M^s_{\lambda_1}-\M^s_{\lambda_2}}_{L^p}\le C_{d,p}\lambda_2^{-2s(q+1)}\abs{\lambda_2^{2s} - \lambda_1^{2s}}.
  \end{equation}
\end{lemma}
\begin{proof}
 Let us prove \eqref{eq: diststable2}.
Consider the map $\lambda\mapsto \M^s_\lambda\in C^1((0,\infty); L^p(\R^d))$. One can easily check that 
 \begin{align*}
     \partial_\lambda \M^s_\lambda&
     =\partial_\lambda(\lambda^{-d}\M^s(x/\lambda)=-d\lambda^{-d-1}\M^s-\lambda^{-d-1}y\cdot \nabla_y \M^s(y)\\
     &=\lambda^{-d-1}(-\div_y(y\M^s))=\lambda^{-d-1}\Delta_y^s\M^s(y)
 \end{align*}
 with  $y=x/\lambda$ and since $\M^s$ is the equilibrium for the (fractional) Fokker--Planck equation 
 $$
 \partial_t \f=-(-\Delta)^s\f+\div(x\f).
 $$
 Taking the $L^p_x$ norm of the previous gives
 $$
 \lambda^{-d-1}\Big(\ird\abs{\Delta^s \M^s(x/\lambda)}^p \d x \Big)^{1/p}=\lambda^{-d-1+d/p}\Big(\ird \abs{\Delta^s \M^s(y)}^p\d y\Big)^{1/p}=\lambda^{-2sq-1}\norm{\Delta^s\M^s}_{L^p}
 $$
 Now since $\Delta^s\M^s\in L^p$ ($\M^s$ is $C^\infty$ and decay at $\infty$ is controlled), the previous estimate gives the final result: indeed,  applying the fundamental theorem of calculus and triangle inequality, gives
 \begin{equation*}
    \begin{split}    \norm{\M^s_{\lambda_2}-\M^s_{\lambda_1}}_{L^p}
    \le \int_{[\lambda_1,\lambda_2]}\norm{\partial_\lambda \M_\lambda^s}_{L^p}\d \lambda\lesssim \int_{[\lambda_1,\lambda_2]}\lambda^{-2sq-1}\d \lambda\lesssim |\lambda_2-\lambda_1| \max\{\lambda_1^{-2sq-1},\lambda_2^{-2sq-1}\},    \end{split}
\end{equation*}
which is nothing but \eqref{eq: diststable2}.

\noindent
It remains to prove that equation \eqref{eq: diststable3} is a consequence of \eqref{eq: diststable2} under the assumption $\abs{1-\frac{\lambda_1^{2s}}{\lambda_2^{2s}}}\le \frac{1}{2}$. We observe that $\max\{\lambda_1^{-2sq-1},\lambda_2^{-2sq-1}\}
\lesssim
\lambda_2^{-2sq-1}.$ If $s=1$ it is indeed a simple consequence of the identity $(a^2-b^2)=(a+b)(a-b)$
   \begin{equation*}  	 \lambda_2^{-2sq-1}\abs{\lambda_2-\lambda_1}=\lambda_2^{-2sq-1}\frac{\abs{\lambda_2^2-\lambda_1^2}}{\abs{\lambda_1+\lambda_2}}\le\lambda_2^{-2sq-2}\abs{\lambda_1^2-\lambda^2_2}.
  \end{equation*}
  If $s\in(0,1)$, notice that by the additional hypothesis $$ \frac{\lambda_1}{\lambda_2}\in \Big[\Big(\frac{1}{2}\Big)^{1/(2s)},\Big(\frac{3}{2}\Big)^{1/(2s)}  \Big]:=I.
  $$
  Then,  the function $h(\theta)=\theta^{2s}$ has derivative $h'(\theta)=2s\theta^{2s-1}$
  which is bounded away from zero when $\theta\in I$.
  Hence, by the mean-value theorem, for some $\xi\in [1,\lambda_1/\lambda_2]\subset I$,
$$
\abs{h(1)-h\Big(\frac{\lambda_1}{\lambda_2}\Big)}=\abs{1-\frac{\lambda_1^{2s}}{\lambda_2^{2s}}}=\abs{1-\frac{\lambda_1}{\lambda_2}}h'(\xi)\ge 2s\abs{1-\frac{\lambda_1}{\lambda_2}} \min_{\xi\in I}\xi^{2s-1}\gtrsim\abs{1-\frac{\lambda_1}{\lambda_2}}. 
$$
Therefore,
   \begin{equation*}  	
   \lambda_2^{-2sq-1}\abs{\lambda_2-\lambda_1}
   =\lambda_2^{-2sq}\abs{1-\frac{\lambda_1}{\lambda_2}}
   \lesssim\lambda_2^{-2sq}\abs{1-\frac{\lambda_1^{2s}}{\lambda_2^{2s}}}   \le\lambda_2^{-2s(q+1)}\abs{\lambda_1^{2s}-\lambda^{2s}_2}.
  \end{equation*}

\end{proof}

\begin{lemma}\label{dist: deltas}
Let $\rho\in L^1$ and let $\M^s$ be the $2s$ standard stable law. Then, there exists a constant $C={C}(d,p,s)$ such that for all $h\in \R^d$
\begin{equation}
    \norm{\M_{{(2st)^{1/(2s)}}}*\rho(x-h) -\M_{{(2st)^{1/(2s)}}}*\rho(x)}_{L^p_{x}}\le Ct^{-q-\frac{1}{2s}}|h|.
\end{equation}

\end{lemma}
\begin{proof}
By the mean value theorem, for some $\theta\in[0,1]$
\begin{equation*}
   \M^s_{{(2st)^{1/(2s)}}}*\rho(x-h)-\M^s_{{(2st)^{1/(2s)}}}*\rho(x)=h\cdot\nabla_x\big(\M_{{(2st)^{1/(2s)}}}*\rho(x-\theta h)\big).
\end{equation*}
For a function $\f$ it is easy to check that
\begin{equation}\label{eq: scaling_temporal}
\norm{\f_\lambda}_{L^p}= \lambda^{-d(1-1/p)}\norm{\f}_{L^p},\qquad \norm{\nabla_x\f_\lambda}_{L^p}= \lambda^{-d(1-1/p)-1}\norm{\nabla_x\f}_{L^p}.
\end{equation}
Collecting the previous identities 
\begin{equation*}
    \begin{split}
        &\norm{\M_{{(2st)^{1/(2s)}}}*\rho(x-h)-\M_{{(2st)^{1/(2s)}}}*\rho(x)}_{L^p}\le |h|\norm{\abs{\nabla_x\big(\M_{{(2st)^{1/(2s)}}}*\rho(\cdot-\theta h ) }\big)}_{L^p}\\        &=|h|\norm{\abs{\nabla_x\M_{{(2st)^{1/(2s)}}}*\rho}\big)}_{L^p}\le|h|\norm{\nabla_x\M_{{(2st)^{1/(2s)}}}}_{L^p}\norm{\rho}_{L^1}\le C|h|t^{-\frac{d}{2s}(1-\frac{1}{p})-\frac{1}{2s}},
    \end{split}
\end{equation*}
where we used the invariance under translation for the $L^p$ norm in the second equation and Young's inequality in the third.\end{proof}

\begin{lemma}\label{lemma: distinitialdata}
Let $s\in(0,1]$ and let $f:\R^d\times \R^d \to \R^+$, be a probability density such that 
  $$\int\int | w |^\nu f(x,w)\d w \d x:=M_\nu<\infty$$
  where $\nu=1$ if $s\in(1/2,1]$ or $\nu\in(0,2s)$ if $s\in(0,1/2]$. Define
  $\rho_h(x) := \int f(x-hw,w)\d w$. Then,  for $p\in[1,\infty]$, there exists a constant $C=C(d,p,s,M_\nu)$  such that
    \begin{equation}\label{eq:64}
        \norm{\M_{{(2st)^{1/(2s)}}}*\rho_h-\M_{{(2st)^{1/(2s)}}}*\rho_0}_{L^p}\le \tilde{C}_{d,p}|h|^\nu t^{-q-\frac{\nu}{2s}}
    \end{equation}
\end{lemma}
\begin{proof} For readability  we denote $\M_{{(2st)^{1/(2s)}}}$ simply by $\M$ when there is no possibility of confusion.   Let $\tau_h$ be the translation operator such that $\tau_hf(x)=f(x-h)$.
We define 
    \begin{multline}D(x)=\M*\rho_h-\M*\rho_0=\ird \M(x-\tilde{y})\ird f(\tilde{y}-hw,w)\d w\d \tilde{y}-\ird \M(x-y)\ird f(y,w)\d w
    \\
    =\ird\ird\left[\tau_{hw}\M(x-y)-\M(x-y)\right] f(y,w)
    \end{multline}
   after  the change of variable $ y=\tilde y-hw$ in the first integral. 

   Employing the invariance of the $L^p$ under translation, we have
   \begin{multline*}
       \norm{D}_{L^p_x}\le \ird\ird \norm{\tau_{h w}\M(\cdot-y)-\M(\cdot-y)}_{L^p_x}f(y,w)\d y \d w\\
       =\ird\ird \norm{\tau_{h w}\M-\M}_{L^p_x}f(y,w)\d y \d w
   \end{multline*}
     We bound $\norm{\tau_{h w}\M-\M}_{L^p_x}$ by the  following two estimates
     \begin{itemize}
         \item The trivial bound
         $$\norm{\tau_{h w}\M-\M}_{L^p_x}\le 2\norm{\M}_{L^p_x}$$
         \item Applying the fundamental theorem of calculus and again the invariance of the $L^p$ norm under translation, one can show the following bound with the gradient 
         \begin{multline*}
             \norm{\tau_{h w}\M-\M}_{L^p_x}|=
             \norm{\int_0^1\frac{\d}{\d \theta}\Big(\M(\cdot-\theta hw\Big)\d \theta}_{L^p_x}\\
             =\norm{- h\int_0^1w\cdot\nabla_x\M(\cdot-\theta hw)\d \theta}_{L^p_x}
             \le |hw|\norm{\nabla\M}_{L^p_x}
         \end{multline*}     \end{itemize}
   For any $\nu\in(0,1]$ interpolating between the two bounds one has 
   $$
   \norm{\tau_{h w}\M-\M}_{L^p_x}\le \left(2\norm{\M}_{L^p_x}\right)^{1-\nu} |hw|^\nu \norm{\nabla\M}^\nu_{L^p_x},
   $$
   and notice that for the case $s\in(1/2,1]$, $\nu=1$, and therefore the bound collapses to the pure gradient bound.
   Substituting in the previous, this gives
    \begin{multline}
        \norm{D(\cdot)}_{L^p_x}\le 
        \ird\ird  \norm{\tau_{hw}\M-\M}_{L^p_x}f(y,w)\d y\d w
        \\
        \le\ird\ird\left(2\norm{\M}_{L^p}\right)^{1-\nu}\left( |h w|\norm{\nabla \mathcal{M}}_{L^p}\right)^\nu f(y,w)
        \\
        =2^{1-\nu}\norm{\M}_{L^p}^{1-\nu}\norm{\nabla\M}_{L^p}^\nu |h|^{\nu}\ird\ird f(y,w) |w|^{\nu}\d w \d y
    \end{multline}
    Thus, recalling the scaling \eqref{eq: scaling_temporal}, we get
    $$
  \norm{D(\cdot)}_{L^p_x}\lesssim |h|^{\nu}(t^{-q})^{1-\nu} (t^{-q-\frac{1}{2s}})^{\nu}M_{\nu}(f)=M_\nu|h|^{\nu}t^{-q}t^{-\frac{\nu}{2s}}
    $$
    which is nothing but \eqref{eq:64}.
    \end{proof}

\subsection{Main argument when the equilibrium \texorpdfstring{$F$}{F} is a stable law}\label{sec: stableBGK}

We first consider the solution of the BGK equation with equilibration towards  either a Maxwellian $\M$ or a stable law $\mathcal{M}^s$.
\begin{equation}
  \label{eq: GBGK}\tag{BGK}
  \partial_t f+v\cdot \nabla_x f=\M^s(v) \rho[f]-f
\end{equation}
where the re-equilibration is towards $\M^s$, the $2s$ stable law
defined in \eqref{eq:stable_law}. Before proceeding with the proof,
let us first point out some basic facts.  The Fourier expression
\eqref{eq:stable_law} means that for all $s<1$, $\M^s$ has heavy tails
i.e.,
$$
\M^s(x)\sim |x|^{-d-2s}\qquad\text{ for large } |x|.
$$
This means that for such indices, it does not possess finite variance
and, in case $(ii)$ of Theorem \ref{thm:mainBGKStable}, it does not
possess a finite mean either, hence the looser hypothesis on the $\nu-$moment for this second case.

Furthermore, except for some specific values of $s$ --- for example,
$s=1/2$ which corresponds to the Cauchy distribution --- stable laws do
not have an explicit form.  Nonetheless, they share an important property: they can serve as attractors for sums of i.i.d. random variables with infinite variance, in contrast to the finite-variance case, when the sum converges to a Gaussian, which corresponds to the  $s=1$ case.  The name stable comes from the fact that these distributions are the only ones that are stable under convolution, namely,  convolution of two $2s$-stable laws is again a $2s$ stable law. More explicitly
\begin{equation}\label{eq: propStable}
\M^s_a*\M^s_b=\M^s_{(a^{2s}+b^{2s})^{1/(2s)}}.
\end{equation}
Moreover, it is a well-known fact that the solution of \eqref{eq: fracheat} can be written as a convolution between the initial data and the fractional heat kernel $\M^s_{(2st)^{1/(2s)}}$, suitably scaled to match the diffusivity constant, i.e.
$$
u(t,\cdot)=\M^s_{(\Gamma(2s+1)t)^{1/(2s)}}*\rho_0,
$$
where we used the identity $2s\Gamma(2s)=\Gamma(2s+1)$. Let us use $\diffC$ as a shorthand for $\Gamma(2s+1)$.
 This stability under convolution ensures that the Wild sum formulation of the solution is
particularly simple.

The equation \eqref{eq: eBGK} can be viewed as a linear transport equation perturbed by a nonlocal jump term. In this case, $T_t$ is the semigroup generated by the free transport
$
-v\cdot\nabla_x,
$
namely
$$
T_t f (x,v) = f(x-tv,v).
$$
Following \eqref{eq:wild2} or, equivalently, \eqref{eq:wild2-expectation}, we write the Wild sum formulation
\begin{align*}
  f(t) &=
  e^{-t} \sum_{n\ge
    0}\int_0^t\int_0^{t_n} \cdots \int_0^{t_2}
  \Phi_n (t, x, v; t_1, \dots, t_n)
  \ \d t_1 \cdots \d t_n
  \\
  &=
  \E[ \Phi_{N_t} (t, x, v; t_1, \dots, t_n)],
\end{align*}
where $N_t$ is the number of jumps up to time $t$ for a Poisson process of intensity $1$ and
\begin{equation*}
  \Phi_n \equiv \Phi_n (t, x, v; t_1, \dots, t_n)
  = T_{t-t_n} \, \L^+ \, T_{t_n-t_{n-1}}
  \cdots
  T_{t_{2}-t_1} \, \L^+ \, T_{t_1} f_0.
\end{equation*}
We have the following lemma
\begin{lemma}[Exchange lemma]
  \label{lem:exchange_BGK}
  For any function $f = f(x,v)$ of the form
  \begin{equation*}
    f(x,v) = F(v) \rho(x)
  \end{equation*}
  we have
  \begin{equation}\label{eq:aux2}
    \L^+ T_t f (x,v) = F(v) \left(F_t * \rho\right) (x),
    \qquad (x,v) \in \R^d \times \R^d.
  \end{equation}
  where $F_t$ is, as usual, the rescaled $F$, given by 
  $F_t(x)=t^{-d}F\left(\frac{x}{t}\right)$.
  \end{lemma}
  \begin{proof}
Since the function is already factorised,  $T_t f(x,v)=\rho(x-tv)F(v)$. Applying $\mathcal{L}^+$ gives 
\begin{align*}
 \L^+ T_t f (x,v)&=F(v)\ird F(w)\rho(x-tw)\d w
 \\
 &=F(v)\ird t^{-d}F\left(\frac{w'}{t}\right)\rho(x-w')\d w '
 \\
 &=F(v) \left(F_t * \rho\right) (x),
\end{align*}
which is nothing but \eqref{eq:aux2}.
  \end{proof}
  The previous result clearly shows how a diffusion effect in the $x$   variable, seen as the convolution $F_t * \rho$ in the variable   $x$, can arise from the interplay between a deterministic transport   operator $T_t$ and a random operator $\L^+$ acting only in $v$. As a   consequence   of this, when $F\equiv \M^s$, one can find a more explicit expression for
  $\Phi_n$:
  \begin{teorema}
  The solution of \eqref{eq: eBGK} for cases $(i)$ and $(ii)$ is given by 
  \begin{equation*}
      f(t)=
  e^{-t} \sum_{n\ge
    0}\int_0^t\int_0^{t_n} \cdots \int_0^{t_2}
  \Phi_n (t, x, v; t_1, \dots, t_n)
  \ \d t_1 \cdots \d t_n
  \end{equation*}
  where 
  \begin{equation}
    \label{eq:An-BGK}
    \Phi_n (t, x, v; t_1, \dots, t_n)
    =  \M^s(v)  \left( \M^s_{\sigma_n} *
      \rho[T_{t_1} f_0] \right)
    \left( x - (t-t_n)v \right),
  \end{equation}
  and \begin{equation}\label{eq: sigma}
  \sigma_n^{2s}=\sum_{i=2}^{n}(t_i-t_{i-1})^{2s}\qquad\text{and}\qquad \M^s_\lambda(x)=\lambda^{-d}\M^s\big(\frac{x}{\lambda}\big)\end{equation}
\end{teorema}

Thanks to Lemma \ref{lem:exchange_BGK},
 we can also write in the expected value sense (omitting the variables of $\Phi_{N_t}$) 
\begin{equation*}
  f(t) - \M(v) u(t,x)
  = \E\big[  \Phi_{N_t} - \M(v) u(t,x) \big] = \E \big[ D_{N_t} \big],
\end{equation*}
where $\Phi_n$ is 
\begin{equation*}
  \Phi_n (t, x, v; t_1, \dots, t_n)
  =  \M^s(v)  \left( \M_{\sigma_n} *
    \rho[T_{t_1} f_0] \right)
  \left( x - (t-t_n)v \right).
\end{equation*}
and $\sigma_n^{2s}=\sum_{i=2}^{n}(t_i-t_{i-1})^{2s}$. 
In order to shorten the expression we call the difference $D_n$:
\begin{equation*}
  D_{n} := \Phi_n(t, x, v; t_1, \dots, t_n) - \M(v) u(t,x).
\end{equation*}

To prove our result we need to estimate
$\| \E\big[ D_{N_t} \big] \|_{L^p}$, where the $L^p$ norm is taken in
the variables $(x,v)$. To do this we will divide the expectation (or
equivalently, the sum and integrals in the expression
(\ref{eq:wild2})) into two regions: a ``probable'' event
$E\subseteq {\Omega_t}$ and an ``improbable'' event $E^{\mathrm c}$,
the complement of $E$ in ${\Omega_t}$. The region $E$ is given by
\begin{equation}
  \label{eq:defE}
  E :=  E_1 \cap E_2 \cap E_3.
\end{equation}
The set $E\subseteq {\Omega_t}$ and the sets
$E_i\subseteq {\Omega_t}$, $i=1, \dots, 3$ depend on the values of
$t$. It is understood that the elements of $E$, $E_1$, $E_2$, $E_3$
are always lists of the form $(n, t_1, \dots, t_n)$,
where $n \in \N$, $n \geq 1$, and
$0 \leq t_1 \leq \cdots  \leq t_n \leq t$. Let $\alpha\in(0,1), \beta\in(0,1/2)$  be chosen later.
Define
$\mathfrak{c}_{n,s}=\Gamma(2s+1)\dfrac{(t_n-t_1)^{2s}}{n^{2s-1}}$ and
\begin{equation}\label{eq: defEi}
\begin{gathered}
E_2=\Big\{ \abs{\sigma_n^{2s}-\mathfrak{c}_{n,s}}\le \frac{1}{n^\beta}\frac{(t_n-t_1)^{2s}}{n^{2s-1}}\Big\},
  \\
  E_1:=\left \{t_1\le \frac{t}{n^\alpha}\right \},
  \qquad E_3:=\left \{t-t_n\le \frac{t}{n^\alpha}\right\}.
\end{gathered}
\end{equation}
$E_2$ is the region where $\sigma_n^{2s}$ is close to its expected
value; and $E_1$, $E_3$ are the regions where $t_1$, $t_n$ are not
very far above or below, respectively, from their expected values.
With this interpretation it is reasonable to expect the probability of
their complements to be small.

\begin{osservazione}
  Our analysis also implicitly relies on a fourth condition: that the
  number of collisions $n$ is close to the mean of the Poisson
  distribution ($n \approx t$). We do not define an explicit set $E_0$
  for this but, instead, we handle this ``hidden'' region directly
  using Poisson tail bounds (as in Propositions \ref{prop:
    boundtailPoiss} and \ref{prop: poiss2}).
\end{osservazione}

We have
\begin{equation*}
  \E\big[ D_n \big]
  =
  \P(E)\, \E\big[ D_n \mid E
  \big]
  + \P(E^{\mathrm{c}})\,
  \E\big[ D_n \mid E^{\mathrm{c}}
  \big].
\end{equation*}
\subsubsection{The probable region}
The most interesting estimate is that in the region $E$. 
From the very definition, we may write
\begin{equation*}
  \Phi_n - \M(v) \M_{(\diffC t)^{1/(2s)}}*\rho_0
  =
  \M(v) \left(
    \M_{\sigma_n}*\rho[T_{t_1}f_0]*\delta_{(t-t_n)v}
    - \M_{(\diffC t)^{1/(2s)}}*\rho_0
  \right).
\end{equation*}
The rest of the section is devoted to proving the following estimate:
\begin{proposizione}\label{prop: finalProb}For all  $\beta\in(0,1/2)$,
    in the probable region $E=E_1\cap E_2\cap E_3$ defined in \eqref{eq: defEi}, one has
    $$
    \int_{E} \norm{\Phi_n-u(t,x)\M^s(v)}_{L^p_{x,v}}\d\mu\lesssim t^{-q}t^{-\beta},
    $$
    
\end{proposizione}
We split this difference into four terms by adding and substracting intermediate terms.
\begin{equation*}
  \M^s_{\sigma_n}*\rho[T_{t_1}f_0]*\delta_{(t-t_n)v}
  -  \M^s_{(\diffC t)^{1/(2s)}  } * \rho_0
  = \NN_1 + \NN_2 + \NN_3+\NN_4,
\end{equation*}
where
\begin{align*}
  \NN_1 :=&
  \M^s_{\sigma_n}*\rho[T_{t_1}f_0]*\delta_{(t-t_n)v}-\M^s_{(\diffC{n}^{1-2s})^{1/(2s)}t}*\rho[T_{t_1}f_0]*\delta_{(t-t_n)v}
  \\
  \NN_2 :=&
  \M^s_{(\diffC{n}^{1-2s})^{1/(2s)}t}*\rho[T_{t_1}f_0]*\delta_{(t-t_n)v}-\M^s_{{(\diffC t)^{1/(2s)}  }}*\rho[T_{t_1}f_0]*\delta_{(t-t_n)v}
  \\
  \NN_3 :=&
  \M^s_{(\diffC t)^{1/(2s)}}*\rho[T_{t_1}f_0]*\delta_{(t-t_n)v}-\M^s_{(\diffC t)^{1/(2s)}}*\rho[T_{t_1}f_0]
  \\
  \NN_4 :=&
    \M^s_{(\diffC t)^{1/(2s)}}*\rho[T_{t_1}f_0]-
    \M^s_{(\diffC t)^{1/(2s)}}*\rho_0
  .
\end{align*}

Let us start by taking the norm in the space variable. In this section, since it is dedicated to the probable region, to lighten the notation we write $L^p$ meaning $L^p(E).$
\medskip

\textbf{Estimate of $\NN_1$.}
Using Young's inequality and applying Lemma \ref{lemma: diststable}
\begin{multline}\label{eq: lasttermF}
	\| \NN_1\|_{L^p_x}
	\le
	\norm{\M^s_{\sigma_n}-\M^s_{(\diffC{n}^{1-2s})^{1/(2s)}t}}_{L^p_x}\norm{\rho[T_{t_1}f_0]*\delta_{(t-t_n)v}}_{L^1_x}
	=\norm{\M^s_{\sigma_n}- \M^s_{(\diffC{n}^{1-2s})^{1/(2s)}t}}_{L^p_x}
	\\
	\lesssim
	 t^{-2s(q+1)}(\diffC n^{1-2s})^{-(q+1)}\abs{\sigma_n^{2s}-\Gamma(2s+1)t^{2s}n^{1-2s}}\\
    \le t^{-2s(q+1)}(n^{1-2s})^{-(q+1)}\Big(\abs{\sigma_n^{2s}-\mathfrak{c}_{n,s}}+\abs{\mathfrak{c}_{n,s}-\Gamma(2s+1) t^{2s}n^{1-2s}}\Big).
\end{multline}
Combining the conditions from $E_1$ and $E_3$, one has $t_n-t_1>t(1-2n^{-\alpha})$ which gives 
\begin{equation}\label{eq: e1e3}
t^{2s}-(t_n-t_1)^{2s}=t^{2s}\left(1-\left(1-2n^{-\alpha}\right)^{2s}\right)\le t^{2s}\left(1-(1-4n^{-\alpha})\right)=4\frac{t^{2s}}{n^\alpha}
\end{equation}
and 
$$
\abs{\mathfrak{c}_{n,s}-\Gamma(2s+1)t^{2s}n^{1-2s}}=\Gamma(2s+1)n^{1-2s}\abs{t^{2s}-(t_n-t_1)^{2s}}\le 4\Gamma(2s+1)\frac{t^{2s}}{n^{2s-1+\alpha}}
$$
Moreover, the constraint defining $E_2$ allows us to estimate the difference
$$\abs{\sigma_n^{2s}-\mathfrak{c}_{n,s}}
\lesssim
\frac{t^{2s}}{n^{2s-1+\beta}}.
$$
Choosing  $\alpha>\beta$,
\eqref{eq: lasttermF} becomes
\begin{equation}\label{eq: N1frac1}
    \norm{\NN_1}_{L^p_x}\lesssim t^{-2sq}t^{-2s}n^{-q(1-2s)}n^{2s-1}t^{2s}n^{1-2s-\beta}=t^{-2sq}n^{-q(1-2s)}n^{-\beta}
 \end{equation}

\noindent
\textbf{Estimate of $\NN_2$.} By Young's inequality and a straightforward change of variable one obtains
\begin{equation*}
	\norm{\NN_2}_{L^p_x}\le \norm{\M^s_{(\diffC n^{1-2s})^{1/(2s)}t} -\M^s_{(\diffC t)^{1/(2s)}}}_{L^p_{x}}\lesssim t^{-q}\norm{\M^s_{\left(\frac{t}{n}\right)^{1-\frac{1}{2s}}}-\M^s}_{L^p_{x}}
\end{equation*}

\medskip \noindent
\textbf{Estimate of $\NN_3$}
We want to apply Lemma \ref{dist: deltas} with $h=(t-t_n)v$, and the fact that   $t -t_n \leq t / n^\alpha$. However for $s\in(0,1/2]$ we have to lower the power of $h$ in order to have $\norm{|v|^k\M^s}$ finite, i.e. with $k<2s(q+1)$.
In this case we have to  interpolate between the bound given by Lemma \ref{dist: deltas}
\begin{equation*}
    \|\NN_3\|_{L^p_x}
    \le
    \|\M^s_{(\tilde{\kappa}t)^{1/(2s)}}*\tilde{\rho}(\cdot-(t-t_n)v)-\M^s_{(\tilde{\kappa}t)^{1/(2s)}}*\tilde{\rho}(\cdot)\|_{L^p_x} 
    \\\lesssim
    t^{-q-\frac{1}{2s}} (t-t_n)|v| 
    \lesssim
    t^{-q-\frac{1}{2s}+1}n^{-\alpha}|v|.
\end{equation*}
and the trivial bound
$$
\|\M^s_{(\tilde{\kappa}t)^{1/(2s)}}*\tilde{\rho}(\cdot-(t-t_n)v)-\M^s_{(\tilde{\kappa}t)^{1/(2s)}}*\tilde{\rho}(\cdot)\|_{L^p_x} \le 2\norm{\M^s_{(\tilde{\kappa}t)^{1/(2s)}}}_{L^p_x}
$$
Therefore, we define the exponent $\theta=\min\{1,k\}$. Note that  for $s>\frac{1}{2}$ we can always choose $k$ larger than $1$, and thus take $\theta=1$.
In full generality, we write
\begin{multline*}
\|\NN_3\|_{L^p_x}
    \le
    \|\M^s_{(\tilde{\kappa}t)^{1/(2s)}}*\tilde{\rho}(\cdot-(t-t_n)v)-\M^s_{(\tilde{\kappa}t)^{1/(2s)}}*\tilde{\rho}(\cdot)\|_{L^p_x} 
    \\\lesssim
    t^{-q-\frac{\theta}{2s}} (t-t_n)^\theta|v|^\theta \lesssim 
    t^{-q-\theta(\frac{1}{2s}-1)} n^{-\theta\alpha}|v|^\theta.
\end{multline*}

\medskip \noindent
\textbf{Estimate of $\NN_4$}
For $\NN_4$, with reasoning similar to the previous case  we use  Lemma \ref{lemma: distinitialdata} and the fact that $t_1\le tn^{-\alpha}$ to obtain 
    \begin{equation*}
    \norm{\NN_4}_{L^p_x}=\norm{\M^s_{(\tilde{\kappa}t)^{1/(2s)}}*{\rho}[T_{t_1}f_0]-\M^s_{(\tilde{\kappa}t)^{1/(2s)}}*\rho_0}_{L^p_x}\le C_{d,p,s} t^{-q-\frac{\nu}{2s}}t_1^{\nu}\le C_{d,p}t^{-q-\nu(\frac{1}{2s}-1)} n^{-\nu\alpha}
\end{equation*}
Let us notice that $\nu\le \theta.$
    
\medskip \noindent
\textbf{Final estimate for the probable region.}
We now take into account all the four terms $\NN_1,\NN_2,\NN_3,\NN_4$ over the probable region $E$ and again we denote the velocity norms as
$C_0=\norm{\M^s(v)}_{L^p_v}$ and $C_\theta=\norm{\M^s(v)|v|^\theta}_{L^p_v}$.
Summing up these contributions against the Poisson measure $e^{-t} \frac{t^n}{n!}$, we obtain:
\begin{multline*}
\int_{E} \norm{\Phi_n-\M^s_{(\tilde{\kappa}t)^{1/(2s)}}(x)\M(v)}_{L^p_{x,v}}\d\mu =
	\int_{ E} \norm{\M(v)\left(\NN_1+\NN_2+\NN_3+\NN_4 \right)}_{L^p_{x,v}}\d\mu
    \\    
	\lesssim C_0e^{-t}\sum_{n\ge
		0} \int_{\mathcal{S}_n(t)}\norm{\NN_1}_{L^p_x}
        +C_0e^{-t}\sum_{n\ge
		0} \int_{\mathcal{S}_n(t)}\norm{\NN_2}_{L^p_x}
        \\
         +e^{-t}\sum_{n\ge
		0} \int_{\mathcal{S}_n(t)}\norm{\M^s(v)\norm{\NN_3}_{L^p_x}}_{L^p_v}
        +C_0e^{-t}\sum_{n\ge
		0} \int_{\mathcal{S}_n(t)}\norm{\NN_4}_{L^p_x}
    \\
	\lesssim C_0t^{-2sq}e^{-t}\sum_{n\ge
		0} \frac{t^n}{n!}n^{-q(1-2s)-\beta}
        +C_0t^{-q}e^{-t}\sum_{n\ge
		0} \frac{t^n}{n!}\norm{\M^s_{\left(\frac{t}{n}\right)^{1-\frac{1}{2s}}}-\M^s}_{L^p_x}
        \\
        +t^{-q-\theta\left(\frac{1}{2s}-1\right)}\norm{\M(v)|v|^\theta}_{L^p_v}e^{-t}\sum_{n\ge
		0} \frac{t^n}{n!}n^{-\theta\alpha}
        +C_0t^{-q-\nu\left(\frac{1}{2s}-1\right)}e^{-t}\sum_{n\ge
		0} \frac{t^n}{n!}n^{-\nu\alpha}
        \\
        \lesssim
        C_0 t^{-q-\beta}+C_0t^{-q-\frac{1}{2}}+C_\theta t^{-q-\theta\left(\frac{1}{2s}-1+\alpha\right)}+C_0 t^{-q-\nu\left(\frac{1}{2s}-1+\alpha\right)}
\end{multline*}
where we used Proposition \ref{prop: boundtailPoiss} for the first, third and fourth terms, and Proposition \ref{prop: poiss2} for the second one after treating separately the exponentially small $n=0$ case.
Since the second term  always  decays faster than the first, and the third always faster than the fourth, we only need to  equate the exponents of the first and fourth terms, which leads us to choose $\alpha$: 
\begin{equation}\label{eq: constrainalpha}
0<\max\left\{\beta,\frac{\beta}{\nu}+1-\frac{1}{2s}\right\}<\alpha<1
\end{equation}

Therefore in the probable region we have
$$\int_{E} \norm{\Phi_n-u(t,x)\M^s(v)}_{L^p_{x,v}}\d\mu\lesssim t^{-q}t^{-\beta}.$$
 
\subsubsection{The improbable region}
\label{sec:improbable}
In this region, we have to bound
$$
\int_{E^c}
    \norm{\Phi_n-\M^s(v)u(t,x)}_{L^p}\d \mu.
$$
We  split this into two terms and show  that each of them is small (so the fact that $D_n$ is a small
difference does not play any role in this region). We write
\begin{equation}\label{eq:ir1}
\int_{E^c}
    \norm{\Phi_n-\M^s(v)u(t,x)}_{L^p}\d \mu
    \le \int_{E^c}
    \norm{\M^s(v)u(t,x)}_{L^p}\d \mu+  \int_{E^c}
    \norm{\Phi_n}_{L^p}\d \mu
\end{equation}
We will show that both of these terms decay like $t^{-q}t^{-(1-2\beta)}$ when $s\in(1/2,1)$, while the second term decays more slowly in the lower fractional case $s\in(0,1/2)$ depending on the Lebesgue exponent $p$.

\medskip \noindent
Regarding the first term, we will show that the region $E^c$ is improbable in the sense that
\begin{equation}\label{eq:mainobjectiveImprobable}
  \P(E^{\mathrm{c}}) \leq C t^{-(1 - 2\beta)}
\end{equation}
and refer to Lemma \ref{lem:improbable-region} for the proof of this. The
value of the $L^p$ norm in the first term in \eqref{eq:ir1} can be
explicitly estimated: we have
\begin{equation*}
  \norm{ \M^s(v) u(t,x) }_{L^p}
  = \left( \ird \ird \M^s(v)^p u(t,x)^p \d x \d v \right)^{\frac{1}{p}}
  \leq C t^{-q},
\end{equation*}
where  $C = C(p,d)$ is a constant. Together, we deduce
\begin{equation}\label{eq:ir2}
  \int_{E^c}\norm{\M^s(v)u(t,x)}_{L^p}=\P(E^{\mathrm{c}}) \, \norm{ \M^s(v) u(t,x) }_{L^p}
  \leq C t^{-q}t^{-(1 - 2\beta)},
\end{equation}
whose proof is contained in Proposition \ref{prop:gauss_improbable}.

\medskip \noindent For the second term,
it is easier to use the integral expression from \eqref{eq:wild2},
written in a more compact way in \eqref{eq:wild3}. We will prove in Proposition \ref{prop:A_n-improbable} below that 
\begin{equation}\label{eq:ir3}
   \int_{E^\mathrm{c}} \| \Phi_n \|_{L^p} \d \mu
  \leq
  C t^{-2sq}t^{q(2s-1)_+}t^{-(1-2\beta)},
\end{equation}
which in particular, for $s\in[1/2,1]$,  reduces to
$$
\int_{E^\mathrm{c}}\norm{\Phi_n}_{L^p}\,d\mu
\le Ct^{-q}t^{-(1-2\beta)}.
$$
This gives the needed estimate of this part.

Using \eqref{eq:ir2} and \eqref{eq:ir3} in \eqref{eq:ir1} we complete
the proof of our estimate for the improbable region. The rest of this
section is devoted to proving the needed estimates to complete the
above argument.

\medskip
\noindent
Let us start with the estimate of the first term in \eqref{eq:ir1}.

\begin{lemma}
  \label{lem:improbable-region}
 There exists $C$ such that \eqref{eq:mainobjectiveImprobable} holds.
\end{lemma}

\begin{proof}We  present a short proof since the technique is similar, although much simpler than the one used to estimate the $\Phi_n$ part (see Proposition \ref{prop:A_n-improbable}). However, since it  already  presents some of the difficulties, we use it as a warm-up.
Since 
\begin{equation}\label{eq: improbableSetInter}
    E^{\mathrm{c}} = (E_1 \cap E_2 \cap E_3)^{\mathrm{c}}
    = E_1^{\mathrm{c}} \cup
    E_2^{\mathrm{c}} \cup E_3^{\mathrm{c}}\
  \end{equation}
 then
  \begin{equation*}
    \P(E^{\mathrm{c}})
    \leq
         \P (E_1^{\mathrm{c}})
    +    \P (E_2^{\mathrm{c}})
    +    \P (E_3^{\mathrm{c}}),
  \end{equation*}
  and we estimate each of the terms separately

  \smallskip\noindent
  \textbf{Estimate of $\P (E_1^{\mathrm{c}})$.} Since $t_1$ is the
  minimum of $n$ random uniform samples on $[0,t]$, for any $\lambda
  \in [0,t],$ we have
  \begin{equation*}
    \P (t_1 \geq \lambda) = \left( 1-\frac{\lambda}{t} \right)^n.
  \end{equation*}
  Hence, for fixed $n \geq 0$, for any $\tilde{\alpha}>0$ there exists a constant $C$ depending on $\alpha$ and $\tilde{\alpha}$ such that 
  \begin{equation*}
    \P \left(t_1 \geq \frac{t}{n^\alpha}\right)
    = \left( 1- \frac{1}{n^\alpha} \right)^n
    \leq \exp\left( -n^{1-\alpha} \right)
    \leq C_\alpha \, n^{-\tilde\alpha}.
  \end{equation*}
  Thus, Proposition \ref{prop: boundtailPoiss} gives
  \begin{equation*}
    \P (E_1^{\mathrm{c}})
    \lesssim e^{-t}
    \sum_{n=0}^\infty \frac{t^n}{ n!} n^{-\tilde\alpha}
    \lesssim
    t^{-\tilde{\alpha}}.
  \end{equation*}
   \textbf{Estimate of $\P (E_3^{\mathrm{c}})$.}
   The estimate is analogous to the previous one.

  \smallskip\noindent
  \textbf{Estimate of $\P (E_2^{\mathrm{c}})$.}
  Let us start our sum from $n\ge 3$ (the first three terms decay exponentially thanks to the $e^{-t}$ factor).
    We will pass from $\mathcal{S}_n(t)$ to the simplex $\Delta_{n-2}$, by
  changing the ``outer'' variables $u=\frac{t_1}{t}$, $w=\frac{t-t_n}{t}$ and  define the $n-2$ ``inner'' variables $x_i=\frac{t_{i+1}-t_i}{t}$ for $i=1,\dots, n-1$. Observe that the vector $\mathbf{x}=(x_1,\dots, x_{n-1})$ defines the simplex $\Delta_{n-2}(1-u-w)$ and the Jacobian of the total transformation is $t^n$. We thus write
  \begin{multline*} 
   I:=\P(E_2^c)\le e^{-t}\sum_{n\ge 3} \int_{\mathcal{S}_n(t)}\mathds{1}_{E^c_2}(t_1,\dots, t_n)
      =e^{-t}\sum_{n\ge 3} \int_{\Delta_n(t)}\mathds{1}_{E^c_2}(x_1,\dots, x_{n+1}) \d x_1\dots \d x_{n+1}
      \\
      =e^{-t}\sum_{n\ge 3}t^n\int_0^{1}\int_{0}^{1-u}\int_{\Delta_{n-2}(1-u-w)}\mathds{1}_{E^c_2}(\mathbf{x})\d\mathbf{x} \d u \d w
      \end{multline*}
      Applying   Corollary \ref{coroll: dirichlet_frac}, we deduce
      \begin{multline*}
      I\le e^{-t}\sum_{n\ge 3}t^n\int_0^{1}\int_{0}^{1-u}\frac{(1-u-w)^{n-2}}{(n-2)!}\frac{1}{n^{1-2\beta}}\d u\d w
      \\
      =e^{-t}\sum_{n\ge 3}t^n\frac{1}{n^{1-2\beta}}\int_0^{1}\int_{0}^{1-u}\frac{z^{n-2}}{(n-2)!}\d u\d z
      \\
      =
      e^{-t}\sum_{n\ge 3}t^n\frac{1}{n^{1-2\beta}}\int_0^1\frac{(1-u)^{n-1}}{(n-1)!}\d u=e^{-t}\sum_{n\ge 3}\frac{1}{n^{1-2\beta}}\frac{t^n}{n!}
      \lesssim t^{-(1-2\beta)}
  \end{multline*}
where we used Proposition \ref{prop: boundtailPoiss} in the last inequality.
Considering the other terms, the result follows.
\end{proof}
As a consequence, we  finally find the correct decay of the first term in the
improbable region, already mentioned in \eqref{eq:ir2}.
\begin{proposizione}
  \label{prop:gauss_improbable}  
  $$
  \int_{E^c}\norm{\M^s(v)u(t,x)}_{L^p_{x,v}}\le t^{-q-(1-2\beta)}
  $$
\end{proposizione}
\begin{proof}
 $$
  \int_{E^c}\norm{\M^s(v)u(t,x)}_{L^p_{x,v}}=\P(E^{\mathrm{c}}) \, \norm{ \M^s(v) u(t,x) }_{L^p}$$
  By Lemma \ref{lem:Lp-norm-of-stable} $\norm{\M(v)u(t,x)}_{L^p_{x,v}}\lesssim t^{-q}$ and by the previous Lemma, we obtain the additional decay $t^{-(1-2\beta)}$.
\end{proof}

\medskip
\noindent
We now  bound the second term in \eqref{eq:ir1}. Let us recall the definition
\begin{equation*}
  \Phi_n (t, x, v; t_1, \dots, t_n)
  =  \big[\M^s(v)   \M^s_{\sigma_n} *
    \rho[T_{t_1} f_0] \big]
  \left( x - (t-t_n)v \right).
\end{equation*}
\begin{proposizione}
  \label{prop:A_n-improbable}
  There exists a constant $C$ such that 
   \begin{equation*}
    \int_{E^\mathrm{c}} \| \Phi_n \|_{L^p} \d \mu
    \leq
    C t^{-2sq}t^{{q(2s-1)_+}}t^{-(1-2\beta)}
  \end{equation*}  
where $a_+:=\max\{0,a\}.$
  In particular, for $s\in[1/2,1],$
  \begin{equation*}
    \int_{E^\mathrm{c}} \| \Phi_n \|_{L^p} \d \mu
    \leq
    C t^{-q} t^{-(1 - 2\beta)}.
  \end{equation*}
\end{proposizione}

\begin{proof}
Let $N_q=\cl{2sq}+1$.
  For $n=0$, $\norm{\Phi_n}_{L^p_{x,v}}=\norm{f_0}_{L^p_{x,v}}$ and $1\le n\le N_q$  we will use the simple bound
  $$\norm{\Phi_n}_{L^p_{x,v}}\le
\norm{\M^s}_{L_v^p}
\norm{\M^s_{\sigma_n}*\rho[T_{t_1}f_0]}_{L_x^p}\le
\norm{\M^s}_{L_v^p}
\norm{f_0}_{L_v^1L_x^p}
$$
  This  gives
  $$
  e^{-t} \sum_{n=0}^{N_q} 
    \int_{0}^t \int_{0}^{t_n} \dots
    \int_{0}^{t_2}\norm{\Phi_n}_{L^p}\mathds{1}_{E^c}
    \d t_1 \dots \d t_n
    \lesssim C e^{-t} \sum_{n=0}^{N_q} \frac{t^n}{n!}
    \le N_qe^{-t}t^{N_q}
    \lesssim t^{-q} t ^{-\tilde{\alpha}}
  $$
  for any $\tilde{\alpha}>0$.
For $n>N_q$, we proceed as follows: from the above recalled expression of $\Phi_n$, using Young's inequality we have
  \begin{multline}\label{eq: i0}
    \| \Phi_n \|_{L^p}^p
    =
    \ird \M^s(v)^p  \ird \Big( 
      (\M^s_{\sigma_n} *
        \rho[T_{t_1} f_0])
      ( x - (t-t_n)v)\Big)^p
    \d x \d v
    \\
    =
    \ird \M^s(v)^p  \ird  \left(
      \left( \M^s_{{\sigma_n}} *
        \rho[T_{t_1} f_0] \right)
      \left( x \right)
    \right)^p
    \d x \d v
    \\
    \leq
    \| \M^s_{\sigma_n}(x) \|_{L^p_x}^p
    \, \| \rho[T_{t_1}f_0] \|_{L^1_x}^p
    \, \| \M^s(v) \|_{L^p_v}^p
    \lesssim
     \left(\sigma_n\right)^{-d (p - 1)},
  \end{multline}
  where, in the last step, we used known facts of the decay of norm of  scaled functions --- see Lemma \ref{lem:Lp-norm-of-stable} and the
  fact that
  $\| \rho[T_{t_1}f_0] \|_{L^1_x} = \| T_{t_1}f_0
  \|_{L^1_{x,v}} = \| f_0 \|_{L^1_{x,v}} = 1$.

Let us compute a lower bound for the $\sigma_n$ term.
For $s\in(1/2,1]$, the
real function $z\mapsto z^{2s}$ is convex. Therefore, by Jensen's inequality
\begin{equation*}
 \sigma_n^{2s} = \sum_{i=2}^{n} (t_{i} - t_{i-1})^{2s}
    \ge
    \frac{1}{(n-1)^{2s-1}} \left( \sum_{i=2}^{n} (t_{i} - t_{i-1}) \right)^{2s}
    =
    \frac{\left( t_n - t_1 \right)^{2s}}{(n-1)^{2s-1}}.
\end{equation*}
For $s=1$, this is just Cauchy--Schwarz inequality, while the case $s=1/2$ is trivial.

For $s\in(0,1/2)$ the function $z \mapsto z^{2s}$ is concave, so the
minimum of  $\sigma_n^{2s}$ is attained when the vector of jump times
is  uneven  as possible (opposite situation of the convex case, where the
minimum is attained when  all the jump times are equal). It can be
seen as a special case of Karamata's inequality. This implies that
\begin{equation*}
 \sigma_n^{2s} = \sum_{i=2}^{n} (t_{i} - t_{i-1})^{2s}
    \ge\left( t_n - t_1 \right)^{2s}.
\end{equation*}
Thus, summarising
\begin{equation}\label{eq: lowerboundSigma}
    \sigma_n^{2s}\ge \frac{(t_n-t_1)^{2s}}{(n-1)^{(2s-1)_+}}.
\end{equation}
  Taking into account the above arguments, we estimate the second term of \eqref{eq:ir1} using \eqref{eq: improbableSetInter}, which means to estimate the integral over each of the regions
  $E_i^{\mathrm{c}}$, $i = 1, 2, 3$.
  For $n>N_q$, by equations \eqref{eq: lowerboundSigma} and \eqref{eq: i0}, we have
  \begin{equation}\label{eq:iec2FRAC}
     \int_{E_i^\mathrm{c}} \| \Phi_n \|_{L^p} \d \mu\le e^{-t} \sum_{n=N_q+1}^\infty (n-1)^{q(2s-1)_+}
    \int_{0}^t \int_{0}^{t_n} \dots
    \int_{0}^{t_2}\left( t_n - t_1 \right)^{-2sq} \mathds{1}_{E_i^c}
    \d t_1 \dots \d t_n.  
  \end{equation}

   \medskip \noindent
  \textbf{Estimate over $E_1^{\mathrm{c}}$.} 
  From equation \eqref{eq:iec2FRAC}, we estimate
  \begin{multline*}
    e^{-t} \sum_{n=0}^\infty (n-1)^{q(2s-1)_+}
    \int_{0}^t \int_{0}^{t_n} \dots
    \int_{0}^{t_2}\left( t_n - t_1 \right)^{-2sq} \mathds{1}_{E_1^c}
    \d t_1 \dots \d t_n
    \\
    =   e^{-t} \sum_{n=0}^\infty (n-1)^{q(2s-1)_+}
    \int_{t/n^\alpha}^t \int_{t/n^\alpha}^{t_n} \dots
    \int_{t/n^\alpha}^{t_{2}}\left( t_n - t_1 \right)^{-2sq}
    \d t_1 \dots \d t_n.
  \end{multline*}
  We compute the integrals first. Changing the order of integration,
  we have
  \begin{multline*}
    \int_{t/n^\alpha}^t
    \int_{t/n^\alpha}^{t_n}
    \dots
    \int_{t/n^\alpha}^{t_{2}}
    \left( t_n - t_1 \right)^{-2sq}
    \d t_1 \dots \d t_n
    \\
    =
    \int_{t/n^\alpha}^t \int_{t/n^\alpha}^{t_n}
    \int_{t_1}^{t_n}
    \dots
    \int_{t_1}^{t_{3}}
    \left( t_n - t_1 \right)^{-2sq}
    \d t_2 \dots \d t_{n-1} \d t_1 \d t_n
    \\
    =
    \frac{1}{(n-2)!}
    \int_{t/n^\alpha}^t \int_{t/n^\alpha}^{t_n}
    \left( t_n - t_1 \right)^{-2sq + n - 2}
    \d t_1 \d t_n
    \\
    =
    \frac{1}{(n-2sq) (n-1-2sq) (n-2)!}
    \left( t - t/n^\alpha \right)^{n - 2sq }
    \\
    \lesssim
    \frac{1}{n!}
    t^{n-2sq}
    \left(1 - \frac{1}{n^\alpha} \right)^{n-2sq}
    \lesssim
    \frac{1}{n!}
    t^{n-2sq}
    \exp\left( -\frac{n-2sq}{n^\alpha} \right)
    \lesssim
    \frac{1}{n!}
    t^{n-2sq}
    \exp\left( -c n^{1-\alpha} \right)
  \end{multline*}
 
  Let us notice that the integral
  $$
  \int_{t/n^\alpha}^t \int_{t/n^\alpha}^{t_n}
    \left( t_n - t_1 \right)^{-2sq + n - 2}
    \d t_1 \d t_n<\infty
  $$
  since $n>N_q$.
  Carrying out the sum, thanks to Proposition \ref{prop: boundtailPoiss} and the very definition of $N_q$,  we  finally have
\begin{multline*}
e^{-t} \sum_{n=N_q+1}^\infty (n-1)^{q(2s-1)_+}
    \int_{t/n^\alpha}^t \int_{t/n^\alpha}^{t_n} \dots
    \int_{t/n^\alpha}^{t_{2}}\left( t_n - t_1 \right)^{-2sq}
    \d t_1 \dots \d t_n
    \\
    \leq
     e^{-t} \sum_{n=0}^\infty 
    \frac{t^n}{n!}
    t^{-2sq}
    n^{q(2s-1)_+}
    \exp\left( -c n^{1-\alpha} \right)
    \lesssim
     e^{-t} \sum_{n=0}^\infty 
    \frac{t^n}{n!} n^{q(2s-1)_+}
    t^{-2sq}
    n^{-\tilde{\alpha}}\lesssim {t^{-2sq}} t^{q(2s-1)_+}t^{-\tilde{\alpha}},
\end{multline*}
  for all $\tilde{\alpha}>0$. 
  By \eqref{eq:iec2FRAC}, this means 
  \begin{equation*}
    \int_{E_1^\mathrm{c}} \| \Phi_n \|_{L^p} \d \mu
    \lesssim
   t^{-2sq+q(2s-1)_+}t^{-\tilde{\alpha}},
  \end{equation*}
  which is the intended estimate. Let us notice that for $s\in(1/2,1]$, this reduces to $t^{-q} t^{-\tilde{\alpha}}$ which has the correct first order decay.
   
  \medskip \noindent
  \textbf{Estimate over $E_2^{\mathrm{c}}$.} 
  Recall the definition of the set $ \mathscr{E}_{n,s}(z) $ in \eqref{eq:defE_stable}.

We perform the following change of variable:
let $u=\frac{t_1}{t}$ and $w=\frac{t-t_n}{t}$ such that $\frac{t_n-t_1}{t}=1-u-w$.
Now define the $n-1$ ``inner'' variables $x_i=\frac{t_{i+1}-t_i}{t}$ for $i=1,\dots, n-1$. Notice that the vector $\mathbf{x}=(x_1,\dots, x_{n-1})$ defines the simplex $\Delta_{n-2}(1-u-w)$. Moreover, the Jacobian of the total transformation is $t^n$. The integral becomes
\begin{multline}\label{eq: i4}
   \int_0^t\int_0^{t_{n-1}}\dots \int_0^{t_2}(t_n-t_1)^{-2sq}\mathds{1}_{E_2^c}\d t_1\dots\d t_n
    \\
    = t^{-2sq}t^n\int_0^{1}\int_0^{1-u}(1-u-w)^{-2sq}\left[\int_{\Delta_{n-2}(1-u-w)}\mathds{1}_{{\mathscr{E}_{n-2}^c(1-u-w)}}(\mathbf{x})\d \mathbf{x}\right]\d w\d u:=I.
\end{multline}
Applying Corollary \ref{coroll: dirichlet_frac} (up to a harmless multiplicative change in the radius) with dimension $n-2$ and $z=1-u-w$,
we have
$$
\int_{\Delta_{n-2}(1-u-w)}\mathds{1}_{{\mathscr{E}_{n-2}^c(1-u-w)}}(\mathbf{x})\d \mathbf{x}
\le\frac{(1-u-w)^{n-2}}{(n-2)!}n^{-(1-2\beta)}.
$$
Substituting back in \eqref{eq: i4}, one has
\begin{equation*}
   I\le  t^{n-2sq}\frac{n^{-(1-2\beta)} }{(n-2)!}
    \int_0^{1}\int_0^{1-u}(1-u-w)^{n-2-2sq}\d w\d u
\end{equation*} Now, for $n>N_q$, the  term inside the integral is
integrable and equal to
 $$
 \int_0^1\int_0^{1-u}(1-u-w)^{-2sq+n-2}\d u \d w
 =\frac{1}{(n-2sq)(n-2sq-1)},
 $$
 and therefore
 \begin{multline*}
    e^{-t} \sum_{n=N_q+1}^\infty (n-1)^{q(2s-1)_+}   \int_0^t\int_0^{t_n}\cdots\int_0^{t_2}
    \left( t_n - t_1 \right)^{-2s q}
    \mathds{1}_{E_2^\mathrm{c}}\d t_1\dots \d t_n
    \\
    \le
     e^{-t} \sum_{n=N_q+1}^\infty  t^n t^{-2sq}(n-2)^{-(1-2\beta)}\frac{(n-1)^{q(2s-1)_+}}{(n-2sq)(n-2sq-1)(n-2)!}
     \\
     \lesssim t^{-2sq}e^{-t}\sum_{n=0}^\infty\frac{t^n}{n!}n^{{-(1-2\beta)+q(2s-1)_+}}\lesssim t^{-2sq}t^{{q(2s-1)_+}}t^{-(1-2\beta)}
  \end{multline*}
where in the last inequality we used   Proposition \ref{prop: boundtailPoiss}.
Again notice that for the case $s\in(1/2,1)$ the decay is $t^{-q}t^{-(1-2\beta)}$.

\medskip \noindent
  \textbf{Estimate over $E_3^{\mathrm{c}}$.}
  As in the case $E_1^{\mathrm{c}}$, we proceed as follows for the improbable region: after a change of order of integration and a similar computation, we obtain
  \begin{multline*}
  \int_0^{t}\int_0^{t_n}\cdots \int_0^{t_2}(t_n-t_1)^{-2sq}
  \mathds{1}_{E_3^{\mathrm{c}}}\d t_1\dots\d t_n=
  \\
  =\frac{1}{(n-2)!}\int_0^t\int_0^{t_n}(t_n-t_1)^{-2sq+n-2}\mathds{1}_{E_3^{\mathrm{c}}}\d t_1\d t_n=\frac{1}{(n-2)!}\int_0^{t(1-n^{-\alpha})}\int_0^{t_n}(t_n-t_1)^{-2sq+n-2}\d t_1\d t_n
  \\
  =\frac{t^{n-2sq}}{(n-2sq)(n-2sq-1)(n-2)!}\Big(1-\frac{1}{n^\alpha}\Big)^{n-2sq}\lesssim \frac{1}{n!}t^{n-2sq}\exp(-cn^{1-\alpha})\lesssim\frac{1}{n!}t^{n-2sq}n^{-\tilde{\alpha}}
  \end{multline*}
  for any $\tilde{\alpha}>0$.
  Applying the previous estimate, Proposition \ref{prop: boundtailPoiss}, and \eqref{eq: i0}, we finally have 
  \begin{equation*}
\int_{E_3^\mathrm{c}}\norm{\Phi_n}_{L^p} \d \mu \lesssim \int_{E_3^\mathrm{c}}\sigma_n^{-2sq}\d \mu \lesssim e^{-t}\sum_{n=0}^\infty n^{q(2s-1)_+}\frac{1}{n!}t^{n-2sq}n^{-\tilde{\alpha}}\lesssim {t^{-2sq}} t^{q(2s-1)_+}t^{-\tilde{\alpha}}.
  \end{equation*}
Summing up the results for the three regions, the claim is proven.
\end{proof}

\subsubsection{Proof of Theorem \ref{thm:mainBGKStable}}
Following the previous two sections, we have
\begin{multline*}
    \norm{f(t)-\M^s(v)u(t,x)}_{L^p}\\
    \le  
    \int_{\Omega_t} \norm{\Phi_n-\M^s(v)u(t,x)}_{L^p}\d \mu= \int_E \norm{\Phi_n-\M^s(v)u(t,x)}_{L^p}\d \mu+ \int_{E^c}
    \norm{\Phi_n-\M^s(v)u(t,x)}_{L^p}\d \mu
    \\
    \le \int_E \norm{\Phi_n-\M^s(v)u(t,x)}_{L^p}\d \mu+ \int_{E^c}
    \norm{\M^s(v)u(t,x)}_{L^p}\d \mu+  \int_{E^c}
    \norm{\Phi_n}_{L^p}\d \mu
    \\
    \lesssim t^{-q-\beta}+t^{-q-(1-2\beta)}+t^{-2sq}t^{{q(2s-1)_+}}t^{-(1-2\beta)}
\end{multline*}
For $s\in(1/2,1]$, optimising on  $\beta$, i.e. choosing  $\beta=1-2\beta$, gives
$$
 \norm{f(t)-\M^s(v)u(t,x)}_{L^p}\le Ct^{-q}t^{-\frac13}
$$

For $s\in(0,1/2)$, the estimate for the improbable region gets worse as $s\to 0$ for larger $p$. Indeed optimising one gets
$$\norm{f(t)-\M^s(v)u(t,x)}_{L^p}\le t^{-q-\beta}
$$
with $
\beta=\frac13(1-(1-2s)q)
$
which is larger than zero for 
$$
\frac{1}{p} > \frac{(1-2s)d - 2s}{(1-2s)d}
$$
Since $1/p\ge 0$, the previous inequality holds true for every $p\in[1,\infty]$ when
$s>\frac{d}{2(d+1)}$.
For $s\in\left(0,\frac{d}{2(d+1)}\right]$, the decay rate is meaningful only for
\begin{equation}\label{eq: pmax}
p<\frac{(1-2s)d}{(1-2s)d-2s}:=p_{{\max}},
\end{equation}
thus proving the main statement. Taking into account the constraint  \eqref{eq: constrainalpha} and recalling that $\alpha<1$, for $s\in(0,1/2)$ we require that
\begin{equation}\label{eq:costrain_nu}
    \nu>\frac{2}{3}s(1-q(1-2s)_+).
\end{equation}
which is satisfied by the condition $\nu\in(s,2s).$

\subsection{Extension to a general equilibrium \texorpdfstring{$F$}{F}}\label{sec: FNOstable}
Let us outline the third case $(iii)$ described at the beginning of Section \ref{sec:BGK}, where we consider an equilibrium $F$ which is in the basin of attraction for a  (generalised) Central Limit Theorem. We do not aim  at a complete characterisation of the assumptions on $F$ under which the result holds; rather, we formulate the corresponding Berry–Esseen-type estimate as an assumption and indicate how the previous proof can be adapted under this hypothesis.
 We assume $f$ solves the following BGK equation
\begin{equation}\label{eq: genBGK}
    \partial_t f+ v\cdot \nabla_x f= F(v)\ird f(t,x,w) \d w- f.
    \end{equation}
The Fourier transform of $F$ has the following expansion:
\begin{equation}\label{eq: assJ}
  \widehat{F}(\eta)=1-\frac{|\eta|^{2s}}{2s}+R(\eta)\qquad\text{with}\quad R(\eta)/|\eta|^{2s}\to 0 \text{ as }|\eta|\to 0
\end{equation}
This roughly means that $F$ has finite second moment if $s=1$ or
$F(v)\sim |v|^{-d-2s}$ for large $|v|$ if $s\in(0,1)$. We make the following Berry-Esseen-type assumption:

\begin{assumption}\label{ass: BE}
Let $s\in(0,1]$ and let $F$ be as in \eqref{eq: assJ}.
For $\sigma_1,\dots,\sigma_n>0$, set
$$
F^{\sigma_1,\dots,\sigma_n}
:=F_{\sigma_1}*\cdots*F_{\sigma_n},
\qquad
\bar{\sigma}^{2s}:=\frac{1}{n}\sum_{i=1}^n\sigma_i^{2s}
$$
We assume that 
\begin{enumerate}[(i)]
    \item There exist $N\in\N$, $C_{BE}>0$ and
$\bar\delta>0$ such that, for every $n\ge N$ 
$$
\norm{
F^{\sigma_1,\dots,\sigma_n}
-\M^s_{n^{1/(2s)}\bar\sigma}
}_{L^p}
\le C_{BE} \, n^{-\frac{d}{2s}\left(1-\frac{1}{p}\right)} n^{-\frac{\bar\delta}{2s}}
\bar\sigma^{-d(1-\frac1p)}\qquad \text{on a suitable probable region }E
$$
\item 
$$
\norm{F^{\sigma_1,\dots,\sigma_n}}_{L^p}\lesssim n^{-\frac{d}{2s}(1-\frac{1}{p})}\bar\sigma^{-d(1-\frac1p)}\qquad \text{on its complement } E^c
$$
\end{enumerate}
The regions $E$, $E^c$ refer to the set of $\sigma_1, \dots, \sigma_n$
where the estimates hold. In general,  $E$ includes the conditions specified (with a slight abuse of notation) in 
(\ref{eq:defE}), (\ref{eq: defEi}) and may have to be further restricted depending on the particular Berry–Esseen theorem available, in order to rule out pathological cases in which the Theorem does not apply.
We assume that these additional excluded configurations are sufficiently improbable, at a level comparable to those obtained in the previous section.
\end{assumption}

Under the previous assumption we have the following result:
\begin{teorema}\label{thm: mainBGKgeneral}
  Let $f$ be a solution to \eqref{eq: genBGK}. Assume $F\in L^p$ satisfying a
  Berry-Esseen Theorem as in Assumption \ref{ass: BE} and assume
  $f_0\in L^1_vL^p_x\cap L^p_{x,v}$ such that
  $$M_\nu=\ird\ird |v|^\nu f_0(x,v)\d v \d x < \infty$$ with $\nu=1$ if
  $s\in(1/2,1]$, or $\nu\in(s,2s)$ if $s\in(0,1/2]$. Let $u$ be the
  solution of the (fractional) heat equation \eqref{eq: fracheat} with
  diffusivity constant $\kappa=\Gamma(2s)$ and macroscopic initial
  density $ u_0(x)=\int f_0(x,w)\d w.$ Then, the solution $f$ of
  equation \eqref{eq: genBGK} with initial data $f_0$ satisfies
  $$
  \norm{f(t,\cdot,\cdot)-F(\cdot)u(t,\cdot)}_{L^p_{x,v}}\lesssim t^{-q}t^{-\omega}
  $$
  The Lebesgue exponent $p$ is such that $p\in[1,\infty]$ if
  $s\in\left(\frac{d}{2(d+1)},1\right]$; and $p\in[1,p_{\max})$,
  with $$p_{\max} = \frac{(1-2s)d}{(1-2s)d -2s}$$ otherwise, and
  \begin{equation*}
    \omega
    =\begin{cases}
       \min\left\{\frac{\bar{\delta}}{2s},\frac{1}{3}\right\}
       &\text{
         if
         }s\in(1/2,1]
       \\
       \min\left\{\frac{\bar{\delta}}{2s},\bar\gamma\right\}&\text{ if }s\in(0,1/2]
     \end{cases}
   \end{equation*}
   where $\bar\gamma$ is as in Theorem \ref{thm:mainBGKStable}
   $$\bar\gamma = \frac{1}{3}\!\left(1 - q(1-2s)_+\right)$$
   and $\bar\delta$ as in Assumption \ref{ass: BE}.    
\end{teorema}

\begin{osservazione}[Discussion on Assumption \ref{ass:
    BE}]\label{rmk: addAss}
    Since this extension is presented only at the level of a proof sketch, we do not try to characterise the probable region $E$ or to give minimal sufficient  conditions ensuring the above Berry--Esseen estimate.
  One could possibly give additional assumptions on $F$---depending on   the chosen $L^p$ space---to show that Assumption \ref{ass: BE} is
  satisfied.  A unified presentation for these Berry--Esseen results is
  to our knowledge not available, and though we believe it would be
  valuable, it falls outside the scope of the present work.

  For the $L^\infty$ framework in the classical case $s=1$, we refer
  to \cite[Thm 3.6]{canizotassi2024}. For the fractional case, one
  could adapt the proof \cite[Thm.2.16]{tassi2026_fractional} for
  $s\in(1/2,1)$ to the whole $s\in(0,1)$ case.  In general, results in
  the $H^k$ framework are available in
  \cite{goudon_junca_toscani_BE_2002}, although one should adapt the
  proof to the non i.i.d case.  In previous results the convergence
  rate is $\bar{\delta} \equiv \delta$ for both $L^2$ and $L^\infty$,
  which extends to all $p \in [2, \infty]$ via standard interpolation.

  Generally, establishing convergence in the $L^1$ framework (or, equivalently,  bounding the total variation distance between the distributions)
  requires stronger hypotheses, since Fourier methods are not available.  For the classical case $s=1$, we cite
\cite{BallyCaramellino2016_CLTtotalvariation} and for the stable
  case $s\in(0,1)$ \cite{Xiang_Xu_Yang_StableCLTTotalvariation},
  although one should adapt the proof to the non i.i.d case. Finally,
  interpolation between $L^1$ and $L^2$, one can prove the
  quantitative result in the intermediate spaces  $p \in [1,2]$.
\end{osservazione}
    
    \begin{proof}[Sketch of the proof of Theorem \ref{thm: mainBGKgeneral}]
    Writing the Wild sum formulation of the solution, one has

\begin{equation*}
  \Phi_n - F(v) \M^s_{(\diffC t)^{1/(2s)}}*\rho_0
  =
  F(v) \left(
    F^{\tau_2,\dots, \tau_n}*\rho[T_{t_1}f_0]*\delta_{(t-t_n)v}
    - \M^s_{(\diffC t)^{1/(2s)}}*\rho_0
  \right).
\end{equation*}
where, according to the notation in Assumption \ref{ass: BE}, 
$$F^{;\tau_2,\dots, \tau_n}=F_{\tau_2}*F_{\tau_3}\dots\cdots F_{\tau_n}$$
and $\tau_i=(t_i-t_{i-1})$. In particular, the scale parameter (standard deviation for $s=1$) of $F^{\tau_2,\dots, \tau_n}$  is $\left(\sum_{i=2}^n (t_i-t_{i-1})^{2s}\right)^{1/(2s)}=\sigma_n$.
Therefore, we split the term of the suitable probable region $E$ into five terms:
\begin{equation*}
  F(v) \left(
    F^{\tau_2,\dots, \tau_n}*\rho[T_{t_1}f_0]*\delta_{(t-t_n)v}
    - \M^s_{(\diffC t)^{1/(2s)}}*\rho_0
  \right)
  = F(v)(\NN_0+\NN_1 + \NN_2 + \NN_3+\NN_4),
\end{equation*}
where 
\begin{align*}
\NN_0 :=& F^{\tau_2,\dots, \tau_n}*\rho[T_{t_1}f_0]*\delta_{(t-t_n)v}-\M^s_{\sigma_n}*\rho[T_{t_1}f_0]*\delta_{(t-t_n)v}
\\
  \NN_1 :=&
\M^s_{\sigma_n}*\rho[T_{t_1}f_0]*\delta_{(t-t_n)v}-\M^s_{(\diffC{n}^{1-2s})^{1/(2s)}t}*\rho[T_{t_1}f_0]*\delta_{(t-t_n)v}
  \\
  \NN_2 :=&
  \M^s_{(\diffC{n}^{1-2s})^{1/(2s)}t}*\rho[T_{t_1}f_0]*\delta_{(t-t_n)v}-\M^s_{{(\diffC t)^{1/(2s)}  }}*\rho[T_{t_1}f_0]*\delta_{(t-t_n)v}
  \\
  \NN_3 :=&
  \M^s_{(\diffC t)^{1/(2s)}}*\rho[T_{t_1}f_0]*\delta_{(t-t_n)v}-\M^s_{(\diffC t)^{1/(2s)}}*\rho[T_{t_1}f_0]
  \\
  \NN_4 :=&
    \M^s_{(\diffC t)^{1/(2s)}}*\rho[T_{t_1}f_0]-
    \M^s_{(\diffC t)^{1/(2s)}}*\rho_0.
\end{align*}
By Assumption \ref{ass: BE}(i), the new term $\NN_0$ is such that 
$$
\int_E \norm{F(v)\NN_0}_{L^p_{x,v}}\,\d\mu
\lesssim
t^{-q-\frac{\bar\delta}{2s}}.
$$
The other four terms are bounded similarly as before.

For the improbable region $E^c$, we have by Assumption \ref{ass: BE}(ii)
$$
\norm{F^{\tau_2,\dots, \tau_n}}_{L^p_{x}}\le (n-1)^{-\frac{d}{2s}(1-\frac{1}{p})}
\bar\tau^{-d(1-\frac1p)}
=
\sigma_n^{-d(1-\frac1p)}.
$$
Taking into account the ``smallness'' of the improbable region, this concludes the argument.
    \end{proof}

\section{Nonlocal  Kinetic Fokker--Planck}\label{sec: NLKFP}
 We want to study the time asymptotics of the solution of the following evolution  PDE
\begin{equation}\label{eq: NLKFP}
    \partial_t f(t,x,v)+v\cdot\nabla_x f=G *_vf-f+\div_v(vf)
\end{equation}
where $G \: \R^d \to \R$ is the isotropic Gaussian distribution with variance $2\Id$, that  is
$$
G(v):=(4\pi)^{-d/2} \exp\left\{-\frac{1}{4}|v|^2\right\}.
$$  
For further reference, we define the drift/transport semigroup appearing in equation \eqref{eq: NLKFP} by 
\begin{equation}
    \label{eq:TranspSG}
    T_t f(x,v):=e^{dt}f\left(x-(e^t-1)v, e^t v\right):=\bar{b}^df(x-\bar{a}v, \bar{b}v)
\end{equation}
with $\bar{a}=e^{t}-1$ and $\bar{b}=e^t$. Our default function spaces are the mixed $(p,r)$
Lebesgue spaces.
These spaces are Banach spaces when endowed with the norm defined in
\eqref{eq:LpxLrv}; see \citep{BenedekPanzone_1961_MixedLebesgue}.
  
For mixed Lebesgue spaces, Minkowski Integral's inequality
\cite[Thm. 202]{Hardy1988} implies that for $1\le r\le p\le \infty$,
we have
\begin{equation}\label{eq:mixedLebesgueInclusion}
\norm{f}_{L^p_xL^r_v}\le \norm{f}_{L^r_vL^p_x}.
\end{equation}
Specifically, we will be mainly interested in the two spaces
$L^1_vL^p_x\subset L^p_xL^1_v$ by
\eqref{eq:mixedLebesgueInclusion}. In $L^1_vL^p_x$ the semigroup
$T_t:L^1_vL^p_x\to L^1_vL^p_x$  is an isometry:
\begin{multline*}
    \norm{T_tf}_{L^1_vL^p_x}=\ird \abs{\ird \abs{e^{dt}f\left(x-(e^{t}-1)v,e^{t}v\right)}^{p}\d x}^{1/p}\d v
    \\
    =\ird e^{dt} \abs{\ird \abs{f\left(y,e^{t}v\right)}^{p}\d y}^{1/p}\d v=
    \ird  \abs{\ird \abs{f\left(y,w\right)}^{p}\d y}^{1/p}\d w=\norm{f}_{L^1_vL^p_x}
\end{multline*}
The same claim is not true in $L^p_xL^1_v$ with $p>1$. For a counterexample, take
$t_0 > 0$, $f(x,v) := T_{-t_0}[ \varphi \otimes \psi ]$, with
$\varphi = \varphi(x)$ an $L^1$ function with compact support, which
is not in $L^p$ and $\psi = \psi(v)$ a smooth, compactly supported
function.  Then one can check that $f \in L^p_x L^1_v$, while
$T_{t_0} f = \varphi \otimes \psi$ is not in $L^p_x L^1_v$. However,
by \eqref{eq:mixedLebesgueInclusion}, if $f \in L^1_vL^p_x$, we can
then control the $L^p_xL^1_v$ norm by
$$
\norm{T_t f }_{L^p_xL^1_v}\le\norm{T_t f }_{L^1_vL^p_x} = \norm{f}_{L^1_vL^p_x}.
$$
We decided to state our results and work whenever possible in the more physically meaningful space $L^p_xL^1_v$ because its norm  is the $L^p$ norm macroscopic density $\rho[f]$.
We observe that all our results in case \eqref{eq:defL-nonlocalFP} could have been stated more generally in 
$L^1_v L^p_x$, with the $L^p_x L^1_v$ result then following as a simple consequence. 

\bigskip
    
We now state the main theorem of this section:
\begin{teorema}\label{thm: mainNLFP}
  Let $p\in[1,\infty]$ and $f$ be the solution of \eqref{eq: NLKFP} with probability
  initial data $f_0\in 
  L^1_vL^p_x$ such that
    $$
    \ird \ird f_0(x,w)|w|\d x\d v=M_1<\infty.
    $$
    Let $u$ be the solution of the heat equation \eqref{eq: fracheat}  with  diffusivity constant $\kappa=1$ and initial data $u_0(x)=\int f_0(x,w)\d w$.
    Then, for all $0<\gamma<\frac{1}{2}$, there exists a constant $C$ depending on $\norm{f_0}_{L^1_vL^p_x}, p,d, M_1,\gamma$, such that
    \begin{equation*}
        \norm{f(t)-u(t,x)F(v)}_{L^p_xL^1_v}\le C t^{-\frac{d}{2}\left(1-\frac{1}{p}\right)} t^{-\gamma}
    \end{equation*}
    where $F$ is the unique probability equilibrium of the operator 
    $\mathcal{L}[f]=G*f-f+\div_v(vf)$.
\end{teorema}

\begin{osservazione}
Let us emphasise  why the adjustment of working in mixed Lebesgue space is necessary: the equilibrium $F(v)$ does not 
belong to $L^p_v$ in general. Through a rescaling argument, however, the operator $\mathcal{L}^\eta$ 
takes the form $\mathcal{L}^\eta f:= \frac{1}{\eta^2}(G_\eta*f-f)+\operatorname{div}_v(vf)$.
This operator is studied in \cite{canizotassi2024} and in 
\cite{tassi2026_fractional}, where we proved that $F(v)$ belongs to $L^p$ depending 
on $\eta$ as follows: it is in $L^1$ for all $\eta\in(0,1]$; in $L^p$ with 
$p\in(1,2)$ if $\eta<\frac{1}{\sqrt{d/2}}$; and in $L^p$ with $p\in[2,\infty]$ if 
$\eta<\frac{1}{\sqrt{2q}}$. Therefore, for $\eta$ sufficiently small and $f_0\in L^1_vL^p_x\cap L^p_{x,v}$, the main theorem  holds true in $L^p_{x,v}$. In this case,  the previous discussion simplify and 
the semigroup $T_t: L^p_{xv}\to L^p_{xv}$ coincides with that of the local  Fokker--Planck case.
    \end{osservazione}

\subsection{Preliminary results}
\subsubsection{The star convolution}\label{sec: starconv}
In this section, we introduce the star-convolution operator, which can be viewed as a sort of twisted convolution. Many of the definitions and properties we show remain true even for the case
$s<1$, and thus, for the sake of completeness, we  include these cases in our analysis as well.
 \begin{definizione}
    Let $g\in L^1(\R^d)$ and $f\in L^1(\R^d \times\R^d)$, $a,b\in \R$. The star convolution is the application $\star: L^1(\R^d)\times L^1(\R^d\times\R^d)\to L^1(\R^d\times \R^d) $ defined by
\begin{equation}\label{eq:defStarConv}
g\starconv{a}{b}f(x,v)=\ird g(y) f(x+ay,v-by)\d y.
\end{equation}
\end{definizione} 
The definition can be extended to general mixed Lebesgue spaces in a
straightforward way. For
our purposes, we will use it in $L^p_xL^1_v$, where we have the
following version of Young's inequality:
\begin{lemma}
    Let $g\in L^1$, $f\in L^p_xL^1_v$.
    Then $g\starconv{a}{b}f(x,v)\in L^p_xL^1_v$ and 
    $$
    \norm{g\starconv{a}{b}f}_{L^p_xL^1_v}\le \norm{g}_{L^1}\norm{f}_{L^p_xL^1_v}.
    $$
\end{lemma}
\begin{proof}
Let us assume both $g,f\ge 0$. Applying Fubini's Theorem, one has
    \begin{multline*}
    \ird g\starconv{a}{b}f(x,v)\d v\le \ird g(y)\ird f(x+ay,v-by)\d v\\
    =
    \ird g(y)\ird f(x+ay,v') \d v'
    =
   \ird g(y)\rho(x+ay)\d y=\tilde{g}_a*\rho(x)
\end{multline*}
where $\tilde{g}_a=|a|^{-d}g(-z/a)$ and $\norm{g}_{L^1}=\norm{\tilde g}$.
Notice that in our application $g=G$ is generally an even function, and therefore
$\tilde g_a=g_a$. Then, applying Young convolution in standard $L^p_x$, we get
$$
 \norm{g\starconv{a}{b}f(x,v)}_{L^p_xL^1_v}=\norm{\tilde g_a*\rho[f]}_{L^p_x}\le\norm{g}_{L^1}\norm{\rho[f]}_{L^p_x}=\norm{g_a}_{L^1}\norm{f}_{L^p_xL^1_v}.
$$
\end{proof}
The star convolution enjoys similar properties to the classical convolution, and it easier to understand in the Fourier variables.
\begin{proposizione}\label{prop: proprietaStarConv}
    Let   $g\in L^1(\R^d)$ and $f\in L^1(\R^d\times \R^d)$. The Fourier transform $\mathcal{F}(g\starconv{a}{b} f)\in C_0$ and 
    \begin{equation}
        \mathcal{F}(g\starconv{a}{b} f)(\xi,\eta)=\widehat{g}(b\eta-a\xi)\widehat{f}(\xi,\eta).
    \end{equation}
\end{proposizione}
\begin{proof}
We just compute
    \begin{equation*}
        \begin{split}
             \mathcal{F}(g\starconv{a}{b} f)(\xi,\eta)&
             =\ird g(y)\ird \ird e^{-i(x\cdot\xi+v\cdot\eta)}f(x+ay,v-by)\d x\d v\d y \\
             &=\widehat{f}(\xi, \eta)\ird g(y)e^{-iy\cdot(b\eta-a\xi)}\d y\\
             &=\widehat{f}(\xi,\eta)\widehat{g}(b\eta-a\xi).
             \qedhere
        \end{split}
    \end{equation*}
\end{proof}
As an immediate generalisation, we have the following corollary
\begin{corollario}\label{coroll: nStarConv}
    Let $g_1,\dots g_n\in L^1(\R^d)$ and $f_0\in L^1(\R^d\times \R^d)$. Define 
    $$\mathcal{G}_n[f_0]= g_n \ \starconv{\bar{a}_n}{\ \bar{b}_n} (g_{n-1} \ \starconv{\bar{a}_{n-1}}{\ \bar{b}_{n-1}}(g_{n-2}(\cdots (g_1 \
       \starconv{\ \bar{a}_1}{\ \bar{b}_1} f_0)\cdots)))$$
       Then
    \begin{enumerate}[(i)]
    \item
    \begin{equation*}
        \mathcal{F}(\mathcal{G}_n[f_0])(\xi,\eta)=\left(\prod_{i=1}^n \widehat{g}_i(\bar{b}_i\eta-\bar{a}_i\xi)\right)\widehat{f}_0(\xi,\eta).
    \end{equation*}
    \item If for all $i$, $g_i=G^{d,s}$,  the   $d$-dimensional isotropic  $2s$-stable law
    $$
    \widehat{G}^{d,s}(\xi)=\exp\left\{-\abs{\xi}^{2s}\right\},
    $$
     then
    $$
    \mathcal{F}(\mathcal{G}_n[f_0])(\xi,\eta)=\left(\exp\{-\sum_{i=1}^{n}|\bar{b}_i\eta-\bar{a}_i\xi|^{2s}\right)\widehat{f}_0(\xi,\eta).
    $$   
    That is, the star convolution maps  stable laws to multivariate stable laws through  a ``twisted'' convolution.
    Isotropy  is thus lost and one enters the realm of multivariate stable laws, functions with characteristic exponent $\Psi_n(\xi,\eta)=\sum_{j=1}^{n}|\bar{b}_i\eta-\bar{a}_i\xi|^{2s}$, which for $s\ne 1$, are defined through a spectral measure;
    we define $G^{2d,s}_{\bar{a},\bar{b}}$ to be the anisotropic stable law  with characteristic exponent $\Psi_n$, that is
     $$
    \mathcal{G}_n[f_0] =G_{\bar{a},\bar{b}}^{2d,s}*f_0.
    $$
    \item 
   If $s=1$, that is if $g_i=G$, the spectral measure  can be uniquely represented by the matrix $\Lambda\Lambda^T$ (which is the  covariance matrix of the multivariate stable law)  that is
    $$
    \mathcal{G}_n[f_0] =\mathbf{G}_{\Lambda}^{2d}*f_0.
    $$
    This is a $2d$-dimensional Gaussian whose  covariance matrix is the  $2d\times 2d$ block matrix
        $$\Lambda\Lambda^T:=\begin{pmatrix}
        \sum_{i=1}^n\bar{a}_i^2&-\sum_{i=1}^n\bar{a}_i\bar{b}_i\\
        -\sum_{i=1}^n\bar{a}_i\bar{b}_i& \sum_{i=1}^n\bar{b}_i^2
        \end{pmatrix}.
        $$   
    \end{enumerate}
    \end{corollario}
\subsubsection{Concentration inequalities}\label{sec: concIneq}
\begin{proposizione}\label{prop: A11Bernstein}
     Let  $U_i\sim\mathrm{Unif}[0,t]$ i.i.d., and define $X_i=\left(e^{-2(t-U_i)}-2 e^{-(t-U_i)}\right)$.
    Then, for all $\alpha>0$
    $$
    \P\left(  \abs{\sum_{i=1}^n\left(X_i-\E(X_i)\right)}>n^\alpha\right)\le 2\exp\left\{-\frac{n^{2\alpha}/2}{\frac{11n}{12t}+\frac{1}{3}n^{\alpha}}\right\}
    $$    
    Let us notice that for a fixed $t$ this is exponentially small for all $\alpha>1/2$, but for $t\sim n$ it is exponentially small for all $\alpha>0$.
\end{proposizione}
\begin{proof}
    We apply Bernstein's inequality:
    if $Y_1,\dots, Y_n$ are independent zero-mean random variable with $|Y_i|\le M$ almost surely, then 
    $$
    {\displaystyle \mathbb {P} \left(\abs{\sum _{i=1}^{n}Y_{i}}\geq \epsilon\right)\leq 2\exp \left(-{\frac {{\tfrac {1}{2}}\epsilon^{2}}{\sum _{i=1}^{n}\mathbb {E} \left[Y_{i}^{2}\right]+{\tfrac {1}{3}}M\epsilon}}\right).}
    $$
    Set $Y_i=X_i-\E(X_i)$. Since  $|X_1|\le 1$ then $ |Y_i|\le 1:=M$. A direct computation gives $\Var(X_i)$, which gives the bound $\Var(X_i)\le \frac{11}{12}\frac{1}{t}$.
\end{proof}

\begin{proposizione}\label{prop: ratioImpr}
Define the random variable $\mathcal{R}$ as 
$$
\mathcal{R}:=\frac{S_1}{\sqrt{S_2}}\qquad\text{where}\qquad 
S_k=\sum_{i=1}^ne^{-k(t-U_i)},
$$ and $U_i\sim\text{Unif}\,[0,t]$. 
Then, for $t>e^3n^{1-\beta}$ with $0<\beta<\frac{1}{2}$
    $$    \P\left(\mathcal{R}>n^{\beta}\right)\le C
    e^{-n^{\beta}}
    $$
\end{proposizione}
\begin{proof}
Let us consider $n\ge 2$ (otherwise is trivial)
  Denote by $U_{(1)},\dots, U_{(n)}$ the order statistics of
  $U_1, \dots, U_n$. Recalling that
  $\sum_{i=1}^nf(U_i)=\sum_{i=1}^nf(U_{(i)})$ we have
  \begin{equation*}
    \mathcal{R}=\frac{S_1}{\sqrt{S_2}}
    {=}
    \frac{\sum_{i=1}^n e^{-(t-U_{(i)})}}
    {\left(\sum_{i=1}^ne^{-2(t-U_{(i)})}\right)^{1/2}}
    =
    \frac{e^{-(t-U_{(n)})}\left(
        1+\sum_{i=1}^{n-1}e^{-(U_{(n)}-U_{(i)})}
      \right)}
    {e^{-(t-U_{(n)})}\left(1+\sum_{i=1}^{n-1}e^{-2(U_{(n)}-U_{(i)})}\right)^{1/2}}
    \le
    1+\sum_{i=1}^{n-1} e^{-(U_{(n)}-U_{(i)})}.
\end{equation*}
Now 
$$
\text{law}\left( \sum_{i=1}^{n-1} e^{-(U_{(n)}-T_i)}\big | U_{(n)}=\tau\right)=\text{law}\left(\sum_{i=1}^{n-1}e^{-U_i}\right),\qquad U_i\stackrel{\text{i.i.d}}{\sim} \mathrm{Unif}[0,\tau].
$$
Let us call $\tilde{S}=\sum_{i=1}^{n-1}Z_i:= \sum_{i=1}^{n-1}e^{-U_i}$.
Thus 
$$
\P(\mathcal{R}>n^{\beta})\le \int_0^t \P\left(\tilde{S}>n^{\beta}-1|U_{(n)}=\tau\right)f_{U_{(n)}}(\tau)\d\tau
$$
where $f_{U_{(n)}}$ is the density of the last order statistics.

Let us study $\P\left(\tilde{S}\,|\,U_{(n)}=\tau\right)$.
Conditioned to the event $\{U_{(n)}=\tau\}$ we now have a sum of i.i.d variables, so we prove a Chernoff-type bound
Since $Z_i\le 1$ a.s., 
$
e^{\lambda{Z_i}}\le 1+(e^{\lambda}-1)Z_i,
$
and hence
\begin{multline*}
   \E\left(e^{\lambda\tilde{S}}\right)=\left(\E(e^{\lambda Z_1})\right)^{n-1}
   \le
   \left(\E(1+(e^{\lambda}-1)Z_1)\right)^{n-1}
   =
   \left(1+(e^{\lambda}-1)\E(Z_1)\right)^{n-1}
   \\
   \le\left(1+(e^{\lambda}-1)\tau^{-1}\right)^{n-1}\le \exp\left\{(n-1)\frac{e^{\lambda}-1}{\tau}\right\} 
\end{multline*}
Thus applying  Chernoff's bound
$$
\P\left(\tilde{S}>n^{\beta}-1|U_{(n)}=\tau\right)\le e^{\lambda}e^{-\lambda n^{\beta}}\E(e^{\lambda \tilde{S}})\le  e^{\lambda}\exp\left\{-\lambda n^{\beta}+(n-1)\frac{e^{\lambda}-1}{\tau}\right\} 
$$
Taking $e^\lambda=\frac{n^{\beta}\tau}{n-1}$, when $\tau>e^2{n^{1-\beta}}$, we have that $e^{\lambda}> e^{2}$. Thus  
\begin{equation*}
\P\left(\tilde{S}>n^{\beta}-1|U_{(n)}=\tau\right)\le e^{\lambda} \exp\left\{- n^{\beta}\log\left(\frac{n^{\beta}\tau}{e(n-1)}\right)-\frac{n-1}{\tau}\right\} \le e^2e^{-n^{\beta}}
\end{equation*}

Moreover let us recall that 
\begin{equation*}
    \P(U_{(n)}<x)= \left(\frac{x}{t}\right)^n=\exp\left\{-n\log\left(\frac{t}{x}\right)\right\}
\end{equation*}

Therefore,  when $t>e^3n^{1-\beta}$, 
\begin{multline*}
\P(\mathcal{R}>n^{\beta})\le
    \int_0^t \P\left(\tilde{S}>n^{\beta}-1|U_{(n)}
    =\tau\right)f_{U_{(n)}}(\tau)\d\tau
    \\
    =\int_0^{e^2n^{1-\beta}}    \P\left(\tilde{S}>n^{\beta}-1|U_{(n)}    =\tau\right)f_{U_{(n)}}(\tau)\d\tau
    +\int_{e^2n^{1-\beta}}^t \P\left(\tilde{S}>n^{\beta}-1|U_{(n)}=\tau\right)f_{U_{(n)}}(\tau)\d\tau
    \\    
    \le \int_0^{e^2n^{1-\beta}}
    f_{U_{(n)}}(\tau)\d\tau+e^2\int_{e^2n^{1-\beta}}^t e^{-n^{\beta}}f_{U_{(n)}}(\tau)\d\tau
    \le 
    \P(U_{(n)}\le e^2n^{1-\beta})+e^2e^{-n^{\beta}}\int_{e^2n^{1-\beta}}^tf_{U_{(n)}}(\tau)\d\tau
    \\
    \le \exp\left\{-n\log \left(\frac{t}{e^2n^{1-\beta}}\right)\right\}+e^2e^{-n^{\beta}}
    \le e^{-n}+e^2 e^{-n^{\beta}}
    \le C
    e^{-n^{\beta}}
\end{multline*}
concluding the proof.
\end{proof}
Building on the previous proof we can show the following Proposition. Since the reasoning is very similar we will only give a sketch of the proof.
\begin{proposizione}\label{prop:upperboundS2}
    For $t>\frac{e^2}{2}n^{1-\tilde\beta}$,
    $$
    \P(S_2>n^{\tilde\beta})\le e^{-n^{\tilde\beta}}
    $$
\end{proposizione}
\begin{proof}
    Let us define as before, $$S_2=\sum_{i=1}^nZ_i=\sum_{i=1}^n e^{-2U_i}\qquad \text{in law}$$
    and again applying the same kind of estimate as before, a Chernoff-type bound gives
    $$    \P\left(S_2>n^{\tilde{\beta}}\right)
    \le 
    \exp\left\{-\lambda n^{\tilde\beta}+\frac{n}{2t}(e^{\lambda}-1)\right\}
    \le
    \exp\left\{-n^{\tilde\beta}\log\left(\frac{2tn^{\tilde\beta-1}}{e}\right)-\frac{n}{2t}\right\}\le 
    e^{-n^{\tilde\beta}}
    $$
    after optimising in $e^{\lambda}=2t n^{\tilde\beta-1}$ and since by hypothesis  $t>\frac{e^2}{2}n^{1-\tilde\beta}$.
\end{proof}
In our results, we will need $\tilde\beta=2\beta$.
The conditions $t>e^3n^{1-\beta}$ and $t>\frac{e^2}{2}n^{1-2\beta}$ of the two previous results are not very restrictive. Indeed, the complementary event happens with probability super-exponentially small, as a consequence of standard tail bounds for Poisson random variables. A non-sharp bound is given in the following trivial lemma, by a simple application of the standard Chernoff bound.
\begin{lemma}\label{lemma: easypoisson}
Let $X\sim\mathrm{Pois}(t)$. There exists $t_0>0$ such that for $t\ge t_0$
    $$
    e^{-t}\sum_{n\ge 0}\frac{t^n}{n!}\mathds{1}_{\{t\le e^3n^{1-\beta}\}} =P\left( X\ge e^{-\frac{3}{1-\beta}}t^{\frac{1}{1-\beta}}
\right)\le Ce^{-t}.  $$
\end{lemma}

\subsubsection{Perturbation of Multivariate Gaussian Distributions}
In what follows, we denote 
$$
\mathbf G(C;z)
=
(2\pi)^{-d}(\det C)^{-1/2}
\exp\left(-\frac12 z^TC^{-1}z\right),.
$$ since working directly with the covariance matrix, rather than the standard deviation is easier and $\mathbf{G}$ denotes  the standard Gaussian instead of the one having variance $2\Id$, in order to avoid heavier notation.

\begin{proposizione}\label{prop: boundMultGauss}
    Let $A$, $B$ be two  symmetric positive definite matrices and assume $B$ is block-diagonal $$B=\begin{pmatrix}
        B_{11}&0\\0&B_{22}
    \end{pmatrix}$$
    Define $\Delta(A,B)=B^{-1/2} (A-B)B^{-1/2}$, and assume that 
    $\norm{\Delta(A,B)}_{op}\le 1/2$.
    Then, for $1\le p,r\le\infty$
    the following holds$$
    \norm{\mathbf{G}(A;\cdot)-\mathbf{G}(B;\cdot)}_{L^p_xL^r_v}\le C_{d,p,r}[\det(B_{11})]^{-\frac12\left(1-\frac1p\right)}[\det(B_{22})]^{-\frac12\left(1-\frac{1}{r}\right)}\norm{\Delta(A,B)}_F
    $$
     where  $\norm{\cdot}_F$ denotes the Frobenius norm.
\end{proposizione}
Let us recall $\norm{D}_F:=\sqrt{\sum_{i=1}^N\lambda_i^2}$, with  $\lambda_1,\dots\lambda_N$ the eigenvalues of $D$.

\begin{proof}
    Let $z=(x,v)\in\R^{2d}$ and let $\Sigma_0=B^{-1/2}AB^{-1/2}$. By a classical change of variable 
    \begin{equation*}\label{eq: GaussDiff1}
     \norm{\mathbf{G}(A;\cdot)-\mathbf{G}(B;\cdot)}_{L^p_xL^r_v}=[\det(B_{11})]^{-\frac12\left(1-\frac1p\right)}[\det(B_{22})]^{-\frac12\left(1-\frac1r\right)}\norm{\mathbf{G}(\Id;\cdot)-\mathbf{G}(\Sigma_0;\cdot)}_{L^p_xL^r_v}.
    \end{equation*}
    Now define the interpolation matrix between $\Sigma_0$ and $\Id$ as
    $$
    \Sigma_\lambda=\Id+(1-\lambda)(\Sigma_0-\Id)=\Id+(1-\lambda)\Delta(A,B)
    $$
    and the function $h:\lambda\to \mathbf{G}(\Sigma_\lambda;\cdot)$.
    Clearly 
    \begin{equation}\label{eq: LpcontrolMult}
    \norm{\mathbf{G}(\Id;\cdot)-\mathbf{G}(\Sigma_0;\cdot)}_{L^p_xL^r_v}=\norm{h(1)-h(0)}_{L^p}=\norm{\int_0^1\partial_\lambda h(\lambda)\d\lambda}_{L^p_xL^r_v}\le 
    \int_0^1    \norm{\partial_\lambda h(\lambda)}_{L^p_xL^r_v}\d \lambda
    \end{equation}
    and $$\partial_\lambda h =\trace \left[\partial_\lambda \Sigma_\lambda (\partial_\Sigma \mathbf{G}(\Sigma; z)\right]=-
    \trace \left[\Delta(A,B)(\partial_\Sigma \mathbf{G}(\Sigma; z))\right].$$
    Then, omitting the constants
    \begin{multline}\label{eq: CovarDeriv}
        \partial_\Sigma \mathbf{G}(\Sigma;\cdot)=\partial_\Sigma\left([\det(\Sigma)]^{-1/2}\exp\left\{-\frac{1}{2}x^T\Sigma^{-1}x\right\}\right)
        \\
        =\left(\partial_\Sigma\exp\left\{-\frac{1}{2}\log [\det(\Sigma)]\right\}\right)\exp\left\{-\frac{1}{2}x^T\Sigma^{-1}x\right\}+[\det(\Sigma)]^{-1/2}\left(\partial_\Sigma\exp\left\{-\frac{1}{2}x^T\Sigma^{-1}x\right\}\right)
        \\
        =\mathbf{G}(\Sigma;\cdot)\left(\partial_\Sigma\left[-\frac{1}{2}\log [\det(\Sigma)]\right]\right)
        +\mathbf{G}(\Sigma;\cdot)
        \left(\partial_\Sigma\left[-\frac{1}{2}\trace \left(x^T\Sigma^{-1}x\right)\right]\right)
        \\
        =\mathbf{G}(\Sigma; \cdot)\left [-\frac12\Sigma^{-1}-\frac12\left(\partial_\Sigma\left[\trace \left(\Sigma^{-1}zz^T\right)\right]\right)
        \right]=\frac{1}{2}\mathbf{G}(\Sigma; \cdot)\left [\Sigma^{-1} z z^T\Sigma^{-1}-\Sigma^{-1}
        \right].           
    \end{multline}
    In the previous derivation, we made use of standard identities from matrix differential calculus for symmetric matrices (see e.g., \cite{Petersen2012}).
    The last term  of \eqref{eq: CovarDeriv} is exactly one half of the Hessian of $\mathbf{G}(\Sigma;\cdot)$.
    Indeed,
    $$
    \nabla_z \mathbf{G}(\Sigma;z)=-\mathbf{G}(\Sigma;z)\Sigma^{-1}z
    $$
    $$
    \hess_z \mathbf{G}(\Sigma;z)=\nabla_z\left(\nabla_z^T \mathbf{G}(\Sigma;z)\right)=\nabla_z (-\mathbf{G} )z^T\Sigma^{-1}-\mathbf{G}\nabla_z(z^T\Sigma^{-1})=\mathbf{G}\left(\Sigma^{-1}zz^T\Sigma^{-1}-\Sigma^{-1}\right).    $$
    Let us notice that  that  
    \begin{equation}
    \label{eq:unifboundSigma}
\frac12\Id\preceq\Sigma_\lambda\preceq\frac32\Id\end{equation}
    Then
     \begin{equation*}
         \abs{\partial_\lambda h(\lambda)}=\frac{1}{2}\abs{ \trace \left[
\Delta(A,B) \hess_z \mathbf{G}(\Sigma;z)\right]}\le \frac{1}{2}\norm{\Delta(A,B)}_F\norm{\hess_z \mathbf{G}(\Sigma;z)}_F
    \end{equation*}By the previous calculation and \eqref{eq:unifboundSigma}
    \begin{multline*}
        \norm{\hess_z \mathbf{G}(\Sigma;z)}_F\le\mathbf{G}(\Sigma_\lambda;z)\norm{|\Sigma_\lambda^{-1}z|^2)+\norm{\Sigma_\lambda^{-1}}_F}
        \\
        \lesssim \mathbf{G}(\Sigma_\lambda;z)(4|z|^2+\sqrt{d})
        \lesssim \mathbf{G}
        \left(\frac{3}{2}\Id;z\right)(|z|^2+1).
    \end{multline*}
    Therefore substituting in
\eqref{eq: LpcontrolMult}, 
$$
\norm{\mathbf{G}(\Id;\cdot)-\mathbf{G}(\Sigma_0;\cdot)}_{L^p_xL^r_v}\le \frac{C}{2}\norm{\Delta(A,B)}_F\norm{\mathbf{G}\left(\frac{3}{2}\Id;\cdot\right)(|\cdot|^2+1)}_{L^p_xL^r_v}\le C_{d,p,r}\norm{\Delta(A,B)}_F
$$
one obtains the claim.
\end{proof}
\begin{osservazione}The previous proposition  has a natural extension to non-block-diagonal covariance matrices $B$  by considering the factorisation $\mathbf{G}_B(x,v)=G_{xx}(x)G_{v|x}(v|x)$. Here, the latter is the conditional Gaussian with covariance matrix given by the Schur complement of $B_{11}$ in $B$. Since in this work we apply the proposition exclusively to block-diagonal matrices the statement reduces to the one given above.
\end{osservazione}
\begin{proposizione}\label{prop: DistInitDataMultGauss}
  Let $f_0:\R^d \times \R^d \to \R$ be a probability distribution in $L^1_{x,v}\cap L^p_xL^1_v$,  with
$$
\int_{\R^d} \int_{\R^d} f_0(x,v)\,|v|\d x \d v = M_1 < \infty,
$$
and define
$$
u_0(x) := \int_{\R^d} f_0(x,v)\,\d v
$$
and a nonnegative $\f_0\in L^1(\R^d)$ with
$$
 \ird \f_0(v)\d v=1\qquad\ird \f_0|v|^{1/2}\d v=1\qquad \ird |v|\f_0(v)\d v<\infty.
$$
Consider  the semigroups 
$$
T_t: L^1(\R^d \times \R^d) \to L^1(\R^d \times \R^d), \quad T_t f_0(x,v) = e^{dt} f_0(x-(e^t-1)v, e^{t}v),
$$
$$
D_t: L^1(\R^d) \to L^1(\R^d), \quad D_t \f_0(v) = e^{dt} \f_0(e^t v),
$$
and let be $\mathbf{G}_{\Lambda}$ a Gaussian block- diagonal covariance matrix $\Lambda\Lambda^T$ so  that 
 $\mathbf{G}_{\Lambda}(x,v)=\Gamma_t(x)G_2(v)$, with $\Gamma_t$ the heat kernel. Then, there exists $C=C(d)>0$ such that 
$$
\norm{ \mathbf{G}_\Lambda * (T_t f) - \mathbf{G}_\Lambda * (u_0
  \otimes D_t \f_0) }_{L^p_xL^1_v} \lesssim
  (2+ M_1) t^{-q}t^{-1/2}\left(\norm{\nabla G_{2}}_{L^1_v}^{1/2}+1\right)
$$
\end{proposizione}
\begin{proof}
Denote $G_1=\Gamma_t$, which will be responsible of the time decay. 
   Define as a shorthand $g_h(y,w)=f_0(y-hw,w)-u_0(y)\f_0(w)$  and notice that 
   \begin{equation}\label{eq: propTranslation}
       \ird g_h(y,w)\d w=\rho_h-\rho_0
   \end{equation}
   where, consistently to the definitions in Lemma \ref{lemma: distinitialdata}, $$\rho_h(y)=\int f_0(y-hw,w)\d w\quad \text{ and }\quad \rho_0=\int f_0(y,w)\d w=u_0.$$
   Thus, 
    \begin{multline*}
        \ird\ird G_1(x-y)G_2(v-z) e^{dt}\left[f_0(y-(e^{t}-1)z, e^tz)-u_0(y)\f_0(e^tz)\right]\d y\d z\\
        =
        \ird\ird G_1(x-y)G_2(v-e^{-t}w) f_0(y-(1-e^{-t})w, w)-u_0(y)\f_0(w)\d y\d w\\
        = G_2(v) \ird \ird G_1(x-y)g_h(y,w)\d y \d w
        \\+
         \ird \ird G_1(x-y)\left[G_2(v-e^{-t}w)-G_2(v)\right] g_h(y,w)\d y \d w:=I+II
    \end{multline*}
    with $h=(1-e^{-t})\le 1.$
    Now by definition \eqref{eq: propTranslation}
$$
I=G_2(v)\ird G_1(x-y)\left[\ird g_h(y,w)\d w\right]\d y=G_2(v)\left[G_1*\rho_h-G_1*\rho_0\right]
$$
By  Lemma \ref{lemma: distinitialdata}, 
we have 
\begin{multline*}
    \norm{I}_{L^p_xL^1_v}=\left(\ird\abs{\left(\ird G_2(v)\d v\right)}^p\abs{G_1*\rho_h(x)-G_1*\rho_0(x)}^p\d x\right)^{1/p}
    \\    
    \lesssim M_1(1-e^{-t})t^{-q}t^{-1/2}\lesssim M_1t^{-q}t^{-1/2}.
\end{multline*}
For the second term, we have 
\begin{multline*}
  \norm{II}_{L^p_xL^1_v}
  =\norm{\ird\ird G_1(\cdot-y)\left(G_2(v-e^{-t}w)-G_2(v)\right)g_h(y,w)\d y\d w
  }_{L^p_xL^1_v}
  \\
  \le\ird\ird \abs{g_h(y,w)}\norm{G_1(\cdot-y)}_{L_x^p}
  \norm{G_2(\cdot-e^{-t}w)-G_2(\cdot)}_{L^1_v}\d y\d w.
\end{multline*}
To estimate the latter expression we use  the following two bounds:
\begin{equation*}
  \| G_2(\cdot-\tilde{h})-G_2(\,\cdot\,)\|_{L^1_v} \leq
  |\tilde{h}| \| \nabla G_2 \|_{L^1_v},
  \qquad
  \| G_2(\cdot-\tilde{h})-G_2(\,\cdot\,)\|_{L^1_v} \leq 2 \| G_2 \|_{L^1_v},
\end{equation*}
which, by interpolation, gives
$$
\norm{G_2(\cdot-\tilde{h})-G_2(\,\cdot\,)}_{L^1_v}
\le
\sqrt{2}
|\tilde{h}|^{\frac12}\norm{\nabla G_2}_{L^1_v}^{\frac12}
\norm{G_2}_{L^1_v}^{\frac12}
=
\sqrt{2}
|\tilde{h}|^{\frac12}\norm{\nabla G_2}_{L^1_v}^{\frac12}.
$$
Using this in our bound for $II$ with $\tilde{h}=e^{-t}w$
\begin{multline*}
\ird\ird \abs{g_h(y,w)}\norm{G_1(\cdot-y)}_{L_x^p}\norm{G_2(\cdot-e^{-t}w)-G_2(\cdot)}_{L^1_v}\d y\d w
\\
\lesssim \norm{G_1}_{L^p_x}
e^{-t/2}
\norm{\nabla G_2}_{L^1_v}^{1/2}
\ird\ird \abs{g_h(y,w)} |w|^{1/2} \d y \d w
\lesssim (1+M_{1/2})t^{-q} t^{-\frac{1}{2}}\norm{\nabla G_2}_{L^1_v}^{1/2}.
\end{multline*}
Together with our previous bound for $I$, this proves the result
taking into account that $M_{1/2} \leq M_1 + 1$.
\end{proof}
As explained in Section \ref{sec: wild}, we may write the mild solution of the equation \eqref{eq: NLKFP} in its Wild sum formulation:
\begin{equation}\label{eq: oldwildNLFP}
  \begin{split}
    f(t)&=e^{-t}T_tf_0+e^{-t}\sum_{n\ge1}\int_0^t\int_0^{t_{n}}\cdots\int_0^{t_2}T_{t-t_n}(G*(T_{t_n -t_{n-1}}(G *\cdots G *(T_{t_1}f_0)\cdots )))\d t_1 \dots \d t_n,
  \end{split}
\end{equation}
where the drift/transport operator $T_t$ is defined in \eqref{eq:TranspSG}. We can ``switch'' the jump part and the transport
part
with the help of the following lemma.
\begin{lemma}[Exchange Lemma]\label{lemma: exchangestar}
For any $p \in [1,\infty]$, $J \in L^1(\R^d)$ and $\varphi \in L^p_x
L^1_v$ we have
$$J*(T_t\f) = T_t [J \ \starconv{\bar a}{{\bar b}} \ \f] \qquad \bar a:=e^{t}-1\text{ and } \bar b:=e^{t},$$
and $\star$ is the star convolution operator introduced in  \eqref{eq:defStarConv}.
\end{lemma}
\begin{proof}
  A simple computation gives
  \begin{equation*}
  \begin{split}
    J*(T_t \f)(x,v)&=\bar b^d \ird J(w)\f \Big(x-\bar av+\bar aw, \bar bv-\bar bw\Big)\d w
    \\
    &=\ird J(w) (T_t \f)(x,v-w)\d w
    \end{split}
  \end{equation*}
  and \begin{equation*}
    \begin{split}
      &T_t [J \ \starconv{\bar a}{\bar b}\ \f](x,v)=T_t\Big[\ird J(w)\f \big(x+\bar aw, v-\bar bw\big)\d w \Big]\\
      &=\bar b^d\ird J(w)\f \big(x-\bar av+\bar aw, \bar bv-\bar bw\big)\d w.
        \qedhere
    \end{split}
  \end{equation*}
\end{proof}
With the help of Lemma \ref{lemma: exchangestar}, equation  \eqref{eq: oldwildNLFP} may be rewritten in a more clear way. Taking into account that the star convolution of
Gaussian functions gives again a Gaussian function (see Corollary
\ref{coroll: nStarConv}), we denote a general multivariate Gaussian
distribution in $\R^{2d}$ with covariance matrix $2\Sigma \Sigma^T$ by
\begin{equation*}
  \mathbf{G}_\Sigma (z) := (4\pi)^{-d} \frac{1}{\sqrt{\det(\Sigma \Sigma^T)}}
  \exp\left( - \frac14 z^T (\Sigma \Sigma^T)^{-1} z  \right),
  \qquad z \in \R^{2d}.
\end{equation*}
Observe that in the definition of $\mathbf{G}_\Sigma$ only
$\Sigma \Sigma^T$ appears.  To keep consistency with this behaviour
(and with the notation used for the isotropic case), for any function
$\mathbf{H} \: \R^{2d} \to \R$ and any invertible matrix $\Lambda$ of size
$(2d) \times (2d)$ we write
\begin{equation*}
  \mathbf{H}_\Lambda(z) := \frac{1}{|\det (\Lambda)|} \mathbf{H} (\Lambda^{-1}z),
  \qquad z \in \R^{2d}.
\end{equation*}
If $\mathbf{H}\equiv\mathbf{G}$ this definition is consistent with the
previous one. In this case several different matrices $\Lambda$ can
give the same multivariate Gaussian (since only $\Lambda \Lambda^T$
appears). Wherever this may cause ambiguity we  choose
$\Lambda := \left(\Lambda\Lambda^T\right)^{1/2}$, the unique symmetric
positive definite square root.
We will also use $\Lambda^{-T}$ as a shorthand for $(\Lambda^T)^{-1}$.

For $k \geq 0$ we  introduce the notation
\begin{equation}\label{eq: Sk}
  S_k=\sum_{i=1}^n e^{-k(t-t_i)}, 
\end{equation}
where, in particular, $S_0 = n$.

\begin{teorema}
  The solution $f$ of \eqref{eq: NLKFP} with nonnegative initial
  condition $f_0 \in L^1_{xv}\cap L^1_vL^p_x$ can be written as
  \begin{equation}\label{eq: solWildNLKFP}
    f(t,x,v)=e^{-t}\sum_{n\ge0}\int_0^t\int_0^{t_n}\cdots\int_0^{t_2}\mathbf{G}_{\Sigma}*T_tf_0(x,v)\d t_1,\dots\d t_n
  \end{equation}
  where \begin{equation}\label{eq: defVarCov}
    \Sigma\Sigma^T=
    \begin{pmatrix}
      S_0-2S_1+S_2&S_1-S_2
      \\
      S_1-S_2 & S_2
    \end{pmatrix}.
  \end{equation}
\end{teorema}

\begin{proof}
  Applying iteratively the exchange Lemma \ref{lemma: exchangestar} to the
  term inside the nested integral, then Corollary \ref{coroll: nStarConv} \emph{(iii)}
  and the semigroup property of $T_t$, we can rewrite 
  $f$ as
  \begin{equation}\label{eq: notsooldwild}
    f(t,x,v)=e^{-t}\sum_{n\ge0}\int_0^t\int_0^{t_n}\cdots\int_0^{t_2}T_t[G_\Lambda*f_0](x,v)\d t_1\dots\d t_n
  \end{equation}
  where $\bar{a}_i := e^{t_i}-1$, $\bar{b}_i := e^{t_i}$, and
  $$\Lambda\Lambda^T=
  \begin{pmatrix}
    \sum \bar{a}_i^2&-\sum \bar{a}_i\bar{b}_i
    \\
    -\sum \bar{a}_i \bar{b}_i&\sum \bar{b}_i^2
  \end{pmatrix}
  $$
  Then we apply the following  property 
\begin{equation}\label{eq:semigroupDistrProp}
       T_t(f*g)=(T_tf)*(T_tg)
  \end{equation}
  which is  proved by computing
  \begin{equation*}
  \begin{split}
  T_t (f*g)(x,v)&=b^d(f*g)(x-av,bv)=b^d\ird \ird f(y,w) g(x-av-y,bv-w)\d y \d w)
  \\
  &=b^d b^d\ird \ird f(y,bz)g(x-av-y, b(v-z) \d y \d z
  \\
  &=b^d b^d\ird\ird  f(\bar y-az,bz)g(x-\bar y-a(v-z), b(v-z))\d \bar y \d z
\\
  &=\ird\ird b^d f(\bar y-az,bz)b^d g(x-\bar y-a(v-z) b(v-z)\d \bar y \d z
  \\
  &=\ird\ird T_tf(\bar y,z)T_t g(x-\bar y,v-z)\d \bar y\d z=T_tf*T_tg.
  \end{split}
  \end{equation*}
  This gives
      $$
      T_t[\mathbf{G}_\Lambda*f_0]= (\mathbf{G}_\Lambda)_{M_t^{-1}}*(T_t f_0)=\mathbf{G}_{M_t^{-1}\Lambda}*(T_tf_0),
      $$
      that is $\mathbf{G}_{M_t^{-1}\Lambda}\sim\mathcal{N}(0, 2M^{-1}_t\Lambda
      \Lambda^T M_t^{-T})$. Observe that
      \begin{equation*}
        \Sigma\Sigma^T = M_t^{-1} \Lambda \Lambda^T M_t^{-T},
      \end{equation*}where
      \begin{equation}\label{eq: associatedMatrix}
        M_t^{-1}=
        \begin{pmatrix}
          1&1-e^{-t}\\
          0&e^{-t}
        \end{pmatrix}
      \end{equation}
      one can calculate $\Sigma\Sigma^T$ by direct matrix multiplication, concluding the proof.
\end{proof}

\medskip

The asymptotic profile of equation \eqref{eq: NLKFP} is $F(v)u(t,x)$.
However, the equilibrium does not have a closed form in the Wild sum
formulation.  Instead of looking directly at
$\norm{f-F(v)u(t,x)}_{L^p_xL^1_v}$, by the triangle inequality we
rewrite the main statement as
$$ \norm{f-F(v)u(t,x)}_{L^p_xL^1_v}\le
\norm{f-\f(t,v)u(t,x)}_{L^p_xL^1_v}+\norm{\f(t,v)u(t,x)
  -F(v)u(t,x)}_{L^p_xL^1_v}.$$
where $\f$ solves the Nonlocal Fokker--Planck equation
\begin{equation}\label{Eq:NLFPvel}
    \partial_t\f= G*\f-\f+\div(v\f)
\end{equation}
with initial  data $\f_0$.
  
Thanks to our previous results in
\cite{canizotassi2024}, we know that the second term of the previous
inequality decays exponentially fast to zero (much faster than the
expected algebraic rate) with a constant which depends on
$\int_{\R^d} |v|\,|\varphi_0(v) - F(v)| \d v$. Hence, we are allowed to
focus our attention on
$$
 \norm{f-\f(t,v)u(t,x)}_{L^p_xL^1_v}.
$$
Let us emphasise that we can choose any well-behaved initial data $\f_0$ that suits our analysis, since the solution will eventually converge to the same steady state, the true object of our study.
By applying  the Wild sum representation of $\f$, and then the property \eqref{eq:semigroupDistrProp} we can rewrite the solution $\f$ as
\begin{multline}
\f(t,v)=e^{dt}e^{-t}\sum_{n\ge0}\int_0^t\int_0^{t_n}\cdots\int_0^{t_2}G_{e^{t}S_2^{1/2}}*\f_0(e^{t}v)\d t_1\dots\d t_n\\
=e^{-t}\sum_{n\ge0}\int_0^t\int_0^{t_n}\cdots\int_0^{t_2}G_{S_2^{1/2}}*(D_t\f_0(v))\d t_1\dots\d t_n
\end{multline}
where $S_2$ is as in \eqref{eq: Sk},
and
$D_t$ is just the dilation semigroup $D_t \f(v)=e^{dt}\f(e^tv)$.
Together with the convolution representation of the solution of the heat
equation, we have
$$
u(t,x) \varphi(t,v) =
e^{-t}\sum_{n\ge0}\int_0^t\int_0^{t_n}\cdots\int_0^{t_2}[G_{t^{1/2}}\otimes G_{S_2^{1/2}}]*(u_0\otimes (D_t\f_0))(x,v)\d t_1\dots\d t_n
$$
with $u_0=\int f_0(x,w)\d w$, $D_t$ the dilation semigroup and $\f_0$
of our choice. In particular, we may write
$$
G_{t^{1/2}}\otimes G_{S_2^{1/2}}(x,v)=\mathbf{G}_{\Sigma_\infty}(x,v),
$$
where
$$
\Sigma_\infty\Sigma_\infty^T=
\begin{pmatrix}
  t &0\\
  0& S_2
\end{pmatrix}.
$$
\subsection{Proof of the main result}
We split our object of interest into two parts:
\begin{multline}\label{eq: rigorous}
    \norm{f(t,x,v)-\f(t,v)u(t,x)}
    \le
    e^{-t}\sum_{n\ge 0}\int_0^{t}\int_0^{t_n}\cdots\int_0^{t_2}\norm{\mathbf{G}_{\Sigma}*T_t f_0 -\mathbf{G}_{\Sigma_\infty}*(u_0\otimes D_t\f_0)}\d t_1,\dots\d t_n
    \\
    \le
    e^{-t}\sum_{n\ge 0}\int_0^{t}\int_0^{t_n}\cdots\int_0^{t_2}
    \norm{\mathbf{G}_{\Sigma}*T_t f_0-\mathbf{G}_{\Sigma_{\infty}}*T_t f_0}\d t_1,\dots\d t_n
    \\
    + e^{-t}\sum_{n\ge 0}\int_0^{t}\int_0^{t_n}\cdots\int_0^{t_2}\norm{\mathbf{G}_{\Sigma_\infty}*T_t f_0-\mathbf{G}_{\Sigma_\infty}*(u_0\otimes D_t\f_0)}:=I+II.
\end{multline}
Formally speaking, the first term decays because there is a core
region, again called ``probable'', where the covariance matrices of the
two Gaussians are very close, and the histories of jumps belonging
to the improbable region have exponentially small probability of
occurring.  The second term decays because of the mixing effect of $T_t$
and the decay of the isotropic Gaussian in which one of the variables
behaves like the heat kernel.  We will study each term separately in
the next two paragraphs.

Let us recall some notation: $U_1,\dots, U_n$ denote a sequence of $n$
independent, identically distributed uniform random variables on
$[0,t]$, and $U_{(1)},\dots, U_{(n)}$ denote their order statistics.

In the first term, we measure the distance between a Gaussian with a
(now random) covariance matrix and a Gaussian with ``target''
covariance matrix, whose velocity  entry is also random.
This distance will be in terms of the spectral norm $\norm{\cdot}$,
i.e. the largest eigenvalues norm, or in terms of the (equivalent)
Frobenius norm $\norm{\cdot}_F$.

Let us rewrite our problem in terms of random variables. Since we sum over all the possible
$i\in\{1,\dots, n\}$, we have
$$\sum_{i=1}^nf(U_{(i)})=\sum_{i=1}^n f(U_i).$$
Therefore we define the random matrix $A$
\begin{equation}\label{eq: defA}
A=\begin{pmatrix}
    A_{11}& A_{12}
    \\
    A_{12}& A_{22}
  \end{pmatrix}
     \\
     =\begin{pmatrix}
     S_0-2S_1+S_2&S_1-S_2
     \\
     S_1-S_2 & S_2
     \end{pmatrix}
\end{equation}
where with abuse of notation we  denote the random version of
\eqref{eq: Sk} by the same symbol
$$
S_k=\sum_{i=1}^ne^{-k(t-U_i)}.
$$
We also define the random matrix $B_\infty$ by
\begin{equation}\label{eq:defBinfty}
B_{\infty}=
\begin{pmatrix}
    t &0
    \\
    0& A_{22}
\end{pmatrix}.
\end{equation}
Let us notice that $A$ and $B_\infty$ are the random versions of
$\Sigma\Sigma^T$ and $\Sigma_\infty\Sigma_\infty^T$, respectively.
\begin{osservazione}  The random matrix does not
exhibit ``concentration'' around its mean in all its
entries. We have
$$
\E(A)=\Sigma_{FP}\Sigma_{FP}^T.
$$
where $\Sigma_{FP}\Sigma_{FP}^T$ is the covariance matrix of the local
Kinetic Fokker--Planck equation. A naive ``concentration'' approximation of this
type would suggest $A \sim \Sigma_{FP}\Sigma_{FP}^T$, but if this were
to hold one would expect the velocity component to present Gaussian
tails.  However we know that the actual distribution has heavier
tails.  This anomalous behaviour occurs because concentration takes
place only in the $(x,x)$ component. Importantly, our target
distribution possesses the same tails.
\end{osservazione}

\subsubsection{The first term}\label{sec:NLfirstterm}
For $\alpha\in(0,1),$ $\beta\in(0,1/2)$, we define the regions

\begin{gather*}
E_0=\left\{n\in \mathbb{N}:n<e^{-\frac{3}{1-\beta}}\, t^{1+\frac{\beta}{1-\beta}}\right\},\\[4pt]
E_1=\{\mathcal{R}<n^\beta\}, \qquad
E_2=\{S_2<n^{2\beta}\}, \qquad
E_3=\left\{\abs{A_{11}-\mathfrak{m}}\le n^\alpha\right\},\\[4pt]
E=E_1\cap E_2\cap E_3.
\end{gather*}
where $A_{11}=S_0-2S_1+S_2$, $\mathcal{R}=\dfrac{S_1}{\sqrt{S_2}}$
and 
\begin{equation}
\label{eq: expectedvalueA11}
\mathfrak{m}:=\E(A_{11})=n+\frac{n}{t}\left(-\frac32+2e^{-t}-\frac12 e^{-2t}\right)
\end{equation}
The probable region is $E=E_1\cap E_2\cap E_3$.
The auxiliary set $E_0$ ensures that the hypotheses of
Propositions \ref{prop: ratioImpr} and \ref{prop:upperboundS2} are satisfied. We use it only when estimating the improbable region; indeed, applying de Morgan's laws, we have
\begin{multline}\label{eq:splitProbImprob}
E^c=(E_1\cap E_2)^c\cup E_3^c=(E_0^c \cap (E_1\cap E_2)^c)\cup(E_0\cap (E_1\cap E_2)^c)\cup  E_3^c
\\
\subseteq
E_0^c \cup (E_0\cap E_1^c) \cup (E_0\cap E_2^c) \cup E_3^c
\end{multline}
Thus, to bound the size of the improbable region, we  verify separately that each of $E_0^c,\, E_0\cap E_1^c,\, E_0\cap E_2^c$, and  $E_3^c$  is small. These estimates are contained  in Lemma \ref{lemma: easypoisson}, Proposition \ref{prop: ratioImpr}, 
Proposition \ref{prop:upperboundS2}, and Proposition \ref{prop: A11Bernstein} respectively. Recall that $\mathcal{R}$ denotes the ratio $$\mathcal{R}=\frac{\sum_{i=1}^ne^{-(t-U_i)}}{\sqrt{\sum_{i=1}^ne^{-2(t-U_i)}}}$$ 
Trivially, 
    $
        1\le \mathcal{R}\le \sqrt{n} \qquad a.s.
    $

\begin{osservazione}
    Let us remark that  $\mathfrak{m}$ is of order $n$. Since 
$$
\left(-\frac32+2e^{-t}-\frac12 e^{-2t}\right)\in\left[-\frac32,0\right]
$$
then for $t$ large enough
\begin{equation}\label{eq:boundEA11}
    \frac{n}{2}\le n\left(1-\frac{3}{2t}\right)\le\mathfrak{m}\le n 
\end{equation}
\end{osservazione}

\paragraph{First term, probable region} \begin{proposizione}\label{prop:firstterm_PROB}
    In the probable region $E$, for all (small) $\beta\in(0,1/2)$
    $$
      e^{-t}\sum_{n\ge 0}
\int_{\mathcal S_n(t)\cap E}
    \norm{\mathbf{G}_{\Sigma}*T_t f_0-\mathbf{G}_{\Sigma_{\infty}}*T_t f_0}\d t_1,\dots\d t_n\lesssim t^{-q}t^{\beta- \frac12}
    $$
   
\end{proposizione}

\begin{proof}
To simplify some calculations we introduce an  intermediate matrix:
\begin{equation}\label{eq:MatBT}
B_{D}=\begin{pmatrix}
    \E(A_{11})& 0
    \\
    0& A_{22}
\end{pmatrix}
\end{equation}
Thus,
$$
 \norm{\mathbf{G}_{A^{1/2}}*T_t f_0-\mathbf{G}_{B^{1/2}_{\infty}}*T_t f_0}\le
 \norm{\mathbf{G}_{A^{1/2}}-\mathbf{G}_{B^{1/2}_{\infty}}}
  \le \norm{\mathbf{G}_{A^{1/2}}-\mathbf{G}_{B_D^{1/2}}}+\norm{\mathbf{G}_{B_D^{1/2}}-\mathbf{G}_{B^{1/2}_{\infty}}}=I+II
$$
We consider the two terms separately:

First define 
\begin{equation}
\Delta(A,B_D)=B_D^{-1/2}(A-B_D)B_D^{-1/2}=\begin{pmatrix}
    \frac{A_{11}}{\mathfrak{m}}-1&\frac{A_{12}}{\sqrt{\mathfrak{m}A_{22}}}\\
    \frac{A_{12}}{\sqrt{\mathfrak{m}A_{22}}}& 0.
\end{pmatrix}
\end{equation}
Recalling the definition of $A_{12}$ and $\mathcal{R}=\frac{S_1}{\sqrt{S_2}}$, one can compute
\begin{multline}
    \norm{\Delta(A,B_D)}_F^2=d\left[\left( \frac{A_{11}}{\mathfrak{m}}-1\right)^2+2\left(\frac{A_{12}}{\sqrt{\mathfrak{m}A_{22}}}\right)^2\right]
    \\
    =d\left[\left( \frac{A_{11}}{\mathfrak{m}}-1\right)^2+2\frac{\left(\mathcal{R}-\sqrt{S_2}\right)^2}{\mathfrak{m}}\right]
    \lesssim \mathfrak{m}^{-2}\left(A_{11}-\mathfrak{m}\right)^2+2\mathfrak{m}^{-1}\mathcal{R}^2+2\mathfrak{m}^{-1}S_2
\end{multline}
Since we are in $E$, taking into account the bounds of \eqref{eq:boundEA11},  
one has, by choosing $\alpha=\beta+\frac12$,
\begin{multline*}
\norm{\Delta(A,B_D)}_F\le \left[\mathfrak{m}^{-2}\left(A_{11}-\mathfrak{m}\right)^2+2\mathfrak{m}^{-1}\mathcal{R}^2+2\mathfrak{m}^{-1}S_2\right]^{1/2}
\\
\le \left[4n^{-2(1-\alpha)}+4n^{-(1-2\beta)}+4n^{-(1-2\beta)}\right]^{1/2}\le Cn^{-\left(\frac12-\beta\right)}
\end{multline*}
 To apply Proposition \ref{prop: boundMultGauss}, we also introduce (temporary) another splitting, i.e. when $n\sim t$. Let us define the set $F$ as
\begin{equation}\label{eq:Fprob}
F=\left\{n:|n-t|\le \frac{t}{4}\right\}
\end{equation}
Then $n\ge\frac34t,$ implies that for $t$ large enough the operator norm  $\norm{\Delta(A,B_D)}_{op}\le \frac12$. Thus, by Proposition \ref{prop: boundMultGauss}, we have for $n\in F$
$$
I\lesssim\mathfrak{m}^{-q}\norm{\Delta(A,B_D)}_F\lesssim n^{-q}n^{-\left(\frac12-\beta\right)}
$$
To bound the second term, we compute 

$$
\Delta(B_D,B_\infty)=\begin{pmatrix}
\frac{\mathfrak m}{t}-1&0\\
0&0
\end{pmatrix},\qquad \norm{\Delta(B_D,B_\infty)}_{op}=\frac{|t-\mathfrak{m}|}{t}
$$

On $F$ and taking into account \eqref{eq:boundEA11}, one can check that  $\norm{\Delta(B_D,B_\infty)}_{op}<1/2$.
Therefore if $n\in F$
$$
II \lesssim t^{-q}\frac{|t-\mathfrak{m}|}{t}   \lesssim t^{-q-1}\left[\left|n-t\right|+\frac{n}{t}\left|-\frac{3}{2}+2e^{-t}-\frac{1}{2}e^{-2t}\right|\right]\le t^{-q-1}\left[\left|n-t\right|+\frac32\frac{n}{t}\right]
$$
On $F^c$, by applying Young's inequality
$$
\norm{\mathbf{G}_{\Sigma}*T_t f_0-\mathbf{G}_{\Sigma_{\infty}}*T_t f_0}_{L^p_xL^1_v}\le\norm{\mathbf{G}_{\Sigma}*T_t f_0-\mathbf{G}_{\Sigma_{\infty}}*T_t f_0}_{L^1_vL^p_x}\le 2\norm{f_0}_{L^1_vL^p_x}.
$$

Let $N_t$ be a Poisson random variable with parameter $t$.
We have (by standard calculations and Chernoff Bounds) \begin{gather}
\E(N_t)=t\qquad \text{and}\qquad \E(|N_t-t|)\le (\E(N_t-t)^2)^{1/2}=t^{1/2}
\\
\E(\mathds{1}_{F^c})=\P\left(|N_t-t|>\frac{t}{4}\right)\le Ce^{-ct}
\end{gather}
Hence, on $E\cap F$, we have
$$
I+II\lesssim
n^{-q-(\frac12-\beta)}
+t^{-q-1}\left(|n-t|+\frac32\frac nt\right),
$$
whereas on $E\cap F^c$ we use the previous Young bound and the smallness of $\E(\mathds{1}_{F^c})$. 
Altogether we have 
\begin{multline*}
 e^{-t}\sum_{n\ge 0}\int_{\mathcal S_n(t)\cap E}
    \norm{\mathbf{G}_{\Sigma}*T_t f_0-\mathbf{G}_{\Sigma_{\infty}}*T_t f_0}_{L^p_xL^1_v}\d t_1,\dots\d t_n
    \\
    \lesssim e^{-t}\sum_{n\ge 0}\frac{t^n}{n!}\left [
    \mathds{1}_F\left(n^{-q}n^{-\left(\frac12-\beta\right)}+ t^{-q-1}\left(\left|n-t\right|+\frac32\frac{n}{t}\right)\right)+2\mathds{1}_{F^c}\norm{f_0}_{L^1_vL^p_x}\right]
    \\
    \le
     \left(e^{-t}\sum_{n\ge 1}\frac{t^n}{n!}n^{-q}n^{-\left(\frac12-\beta\right)}
     \right)
     + t^{-q-\frac12}+\frac{3}{2}t^{-q-1}+Ce^{-ct}
    \\
    \lesssim t^{-q}t^{-\left(\frac{1}{2}-\beta\right)}+t^{-q-\frac{1}{2}}+t^{-q-1}
    +e^{-ct}
    \lesssim t^{-q}t^{-\left(\frac{1}{2}-\beta\right)}
\end{multline*}
\end{proof}

\paragraph{First term, improbable region}
The main result of this section is the following proposition:
\begin{proposizione}
    In the improbable region $E^c$, for all $\beta\in(0,1/2)$
    $$
    e^{-t}\sum_{n\ge 0}\int_{\mathcal{S}_n(t)}\norm{\mathbf{G}_{\Sigma}*T_tf_0-\mathbf{G}_{\Sigma_{\infty}}*T_tf_0}_{L^p_xL^1_v}\d t_1\dots\d t_n\lesssim e^{-t^{\beta}}
    $$
\end{proposizione}
\begin{proof}
Applying Young's convolution inequality, together with  the fact that the initial data is in $L^1_v L^p_x,$ where $T_t$ is an isometry gives 
$$
\norm{\mathbf{G}_{\Sigma}*T_tf_0}_{L^p_xL^1_v}\le \norm{\mathbf{G}_{\Sigma}*T_tf_0}_{L^1_vL^p_x} 
\le 
\norm{T_tf_0}_{L^1_vL^p_x}=\norm{f_0}_{L^1_vL^p_x}
 $$
 The same holds for $\mathbf{G}_{\Sigma_{\infty}}$ so their convolved difference is bounded by $2\norm{f_0}_{L^1_vL^p_x}$.

Thanks to the splitting \eqref{eq:splitProbImprob}, we have
\begin{multline*}
     e^{-t}\sum_{n\ge 0}\int_{\mathcal{S}_n(t)\cap E^c}
    \norm{\mathbf{G}_{\Sigma}*T_t f_0-\mathbf{G}_{\Sigma_{\infty}}*T_t f_0}\d t_1,\dots\d t_n
    \\
    \le 2e^{-t} \norm{f_0}_{L^1_vL^p_x}+2 \norm{f_0}_{L^1_vL^p_x} e^{-t}\sum_{n\ge 1}\int_{\mathcal{S}_n(t)\cap E^c}\d t_1,\dots\d t_n
\\
\le  2e^{-t} \norm{f_0}_{L^1_vL^p_x}+2 \norm{f_0}_{L^1_vL^p_x}e^{-t}\sum_{n\ge 0}\int_{\mathcal{S}_n(t)\cap E^c}\d t_1,\dots\d t_n
\\
\le 2e^{-t} \norm{f_0}_{L^1_vL^p_x}+2\norm{f_0}_{L^1_vL^p_x} e^{-t}\sum_{n\ge 0}\frac{t^n}{n!}\left(\P(E_0^c)+\P\left(E_0\cap E_1^c\right)+\P\left(E_0\cap E_2^c\right)+\P(E_3^c)\right)
\end{multline*}
Now applying the concentration inequalities  (with $\tilde\beta=2\beta$ in Proposition \ref{prop:upperboundS2}) gives
\begin{subequations}
    \begin{equation}
     e^{-t}\sum_{n\ge 0}\frac{t^n}{n!}\P(E_0^c)\le Ce^{-t}\qquad\qquad\text{by Lemma \ref{lemma: easypoisson}}
    \end{equation}
    \begin{equation}\label{eq:improb2}
        \P\left(E_0\cap E_1^c\right)\le C e^{-n^\beta}\qquad\qquad\text{by Proposition \ref{prop: ratioImpr}}
    \end{equation}
      \begin{equation}\label{eq:improb3}
        \P\left((E_0\cap E_2^c)\right)\le e^{-n^{2\beta}}\qquad\qquad\text{by Proposition \ref{prop:upperboundS2}}
\end{equation}\begin{equation}\label{eq:improb4}
        \P(E_3^c)\le 2\exp\left\{-c\frac{n^{2\alpha}}{\frac{n}{t}+n^\alpha}\right\}\qquad\text{by Proposition \ref{prop: A11Bernstein}}
    \end{equation}
\end{subequations}
Plugging \eqref{eq:improb2}, \eqref{eq:improb3} and \eqref{eq:improb4} into  the Poisson sum and using the fact that similarly to Proposition \ref{prop: boundtailPoiss}
$$
e^{-t}\sum_{n\ge 0}\frac{t^n}{n!}e^{-n^\beta}\le Ce^{-t^{\beta}}
$$
and 
$$
e^{-t}\sum_{n\ge 0}\frac{t^n}{n!}\exp\left\{-c\frac{n^{2\alpha}}{\frac{n}{t}+n^\alpha}\right\}\lesssim e^{-ct^\alpha}
$$
Recalling that we took $\alpha=\beta+\frac{1}{2}$
then
$$
 e^{-t}\sum_{n\ge 0}\int_{\mathcal{S}_n(t)\cap E^c}
    \norm{\mathbf{G}_{\Sigma}*T_t f_0-\mathbf{G}_{\Sigma_{\infty}}*T_t f_0}\d t_1,\dots\d t_n\lesssim e^{-t^\beta}
$$

\end{proof}
\subsubsection{The second term}
Consider the semigroups
$$
T_t: L^1(\R^d \times \R^d) \to L^1(\R^d \times \R^d), \quad T_t f_0(x,v) = e^{dt} f_0(x-(e^t-1)v, e^{t}v),
$$
$$
D_t: L^1(\R^d) \to L^1(\R^d), \quad D_t \f_0(v) = e^{dt} \f_0(e^t v).
$$
We need to bound
$$II
:= e^{-t}\sum_{n\ge 0}
\int_{\mathcal{S}_n(t)}\norm{\mathbf{G}_{\Sigma_\infty}*T_t
  f_0-\mathbf{G}_{\Sigma_\infty}*(u_0\otimes D_t\f_0)}_{L^p_xL^1_v}.$$
 We recall Proposition \ref{prop: DistInitDataMultGauss}, which gives
$$
\norm{ \mathbf{G}_{\Sigma_\infty} * (T_t f) -
  \mathbf{G}_{\Sigma_{\infty}}* (u_0 \otimes D_t \f_0) }_{L^p_xL^1_v}
\le
CM_1 t^{-q}t^{-1/2}\left(\norm{\nabla \mathbf{G}_{S_2^{1/2}}}_{L^1_v}^{1/2}+1\right).
$$
and since 
$
\norm{\nabla \mathbf{G}_{S_2^{1/2}}}_{L^1_v}\le C_1 S_2^{-1/2},
$
we are left to study (removing for now the pathological $n=0$ case), the following term 
$$
C_1CM_1 t^{-q}t^{-1/2}e^{-t}\sum_{n\ge 1}\int_{\mathcal{S}_n(t)}S_2^{-\frac{1}{4}}\d t_1\dots  \d t_n,
$$
Thus, again with an abuse of notation where $S_2$ denotes both a
function and the corresponding random variable, for any $\gamma>0$
\begin{equation}\label{eq:expectS_2NEG}
e^{-t}\sum_{n\ge 1}\int_{\mathcal{S}_n(t)}S_2^{-\gamma}\d t_1\dots  \d t_n=\E(S_2^{-\gamma};N_t\ge 1)
\end{equation}
Conditionally to the event $\{N_t=n\}$, we bound $\P(S_2<y)$. Since
$S_2$ is a sum of positive (and independent) terms, it is certainly greater than
or equal to the largest of those terms. Using the properties of uniform distributions, one has for $e^{-2t}\le y\le 1$:
\begin{multline}\label{eq:cdfS2}   
  \P(S_2<y|N_t=n)\le \P\left(\max_{i}e^{-2(t-U_i)} < y\right)
  =
  \left(\P\left(e^{-2(t-U_1)} < y\right) \right)^n
  \\
  =\left(\P\left(U_1\le t+\frac12\log(y) \right)\right)^n
  =\left(1+\frac{1}{2t}\log(y)\right)^n:=(z_y)^n,
\end{multline}
while 
$$
 \P(S_2<y|N_t=n)=0\qquad \text{if }0<y<e^{-2t}.
$$
Now we can rewrite equation \eqref{eq:expectS_2NEG} as
\begin{multline}
 \E(S_2^{-\gamma};N_t\ge 1)=\int_0^\infty\P(S_2^{-\gamma}>x;N_t\ge 1)\d x
\\
=
\int_0^1\P(S_2^{-\gamma}>x;N_t\ge 1)\d x+\int_1^\infty\P(S_2^{-\gamma}>x;N_t\ge 1)\d x
\\
\le 
1+\gamma\int_0^1\P(S_2<y;N_t\ge 1)y^{-\gamma-1}\d y.
\end{multline}
Using formula \eqref{eq:cdfS2} and the Fubini-Tonelli theorem,
 \begin{multline}
   \E(S_2^{-\gamma};N_t\ge 1)
   \leq 1+\gamma
   \sum_{n\ge 1}\P(N_t=n)\int_{e^{-2t}}^1 z_y^n y^{-\gamma-1}\d y
   \\
   \le1+\gamma\int_{e^{-2t}}^1y^{-\gamma-1}\left(\sum_{n\ge 0}\P(N_t=n)z_y^n \right)\d y
   =
   1+\gamma \int_{e^{-2t}}^{1}y^{-\gamma-1}\E(z_y^{N_t})\d y.
 \end{multline}
 Since the moment generating function of a Poisson random variable of parameter $\lambda$ is $M_{N_t}(u)=\exp\left\{\lambda (e^{u}-1\right)\}$
 then 
 $$
 \E(z_y^{N_t})=M_{N_t}(\log(z))=\exp\left\{ t\left(e^{\log\left(1+\frac{1}{2t}\log y\right)}-1\right)\right\}=y^{\frac12}.
 $$
 Thus \eqref{eq:expectS_2NEG} becomes
 $$
 \E\left(S_2^{-\gamma};N_t\ge 1\right)\le 1+\gamma\int_0^{1} y ^{-\gamma-1}y^{1/2}\d y.
 $$
 where we removed $e^{-2t}$ for simplicity and the integral is finite in $0$ only for $\gamma<1/2$.
 In our case $\gamma=\frac{1}{4}$,
 thus we can conclude
 \begin{multline*}
 e^{-t}\sum_{n\ge 0}\int_{\mathcal{S}_n(t)}\norm{\mathbf{G}_{\Sigma_\infty}*T_t f_0-\mathbf{G}_{\Sigma_\infty}*(u_0\otimes D_t\f_0)}_{L^p_xL^1_v}
 \\
 \le  C_{f_0}e^{-t}+
 e^{-t}\sum_{n\ge1 }\int_{\mathcal{S}_n(t)}\norm{\mathbf{G}_{\Sigma_\infty}*T_t f_0-\mathbf{G}_{\Sigma_\infty}*(u_0\otimes D_t\f_0)}_{L^p_xL^1_v} 
 \le \tilde{C}t^{-q}t^{-1/2},
 \end{multline*}
 where the constant for the initial data really depends only on $f_0$ since $u_0=\int f_0$ and $\f_0$ is now fixed and its dependence is absorbed into the constant.

Gathering together the estimates for the two terms, this conclude the proof of Theorem \ref{thm: mainNLFP}

\section{Kinetic Fokker--Planck and Fractional Kinetic
   Fokker--Planck}
\label{sec: KFPandFKFP}

In this section we deal with the asymptotic behaviour of the equation
\begin{equation}
  \label{eq: KFP}
  \partial_tf+v\cdot\nabla_x f=-(-\Delta_v)^sf+\div_v(vf).
\end{equation} 

The main result reads as follows.
\begin{teorema}
  \label{thm:mainKFP}
  Let $s\in(0,1]$, $p\in[p_s,\infty]$, with $p_1=1$ and $p_s=2$ if
  $s\in(0,1)$. Let $f$ be the solution of \eqref{eq: KFP} with
  probability initial data $f_0\in L^p_{x,v}$ such that
  \begin{equation}\label{eq:M_nu}
    \ird \ird f_0(x,w)|w|^\nu \d x\d w=M_\nu<\infty,
  \end{equation}
  with $\nu=1$ if $s\in(1/2,1]$ and $\nu\in(s,2s)$ if $s\in(0,1/2]$.  
  Let $u$ be the solution of the (fractional) heat equation \eqref{eq: fracheat} 
  with  diffusivity constant $\kappa=1$ and initial data $u_0(x)=\int f_0(x,w)\d w$.
  Then, there exists a constant $C$ depending on $p,\,d,\,s,\, M_\nu$, $\norm{f_0}_{L^p}$ such that
  \begin{equation*}
    \norm{f(t)-u(t,x)\M^s(v)}_{L^p_{xv}}\le C t^{-q} t^{-\frac{\nu}{2s}}.
  \end{equation*}
\end{teorema}
The restriction $p\in[2,\infty]$ comes from the fact that  the stable laws do not possess an explicit inverse Fourier transform. For the case $p\in(1,2)$, we rely on the following remark.
\begin{osservazione}\label{rmk:interpolationPsmall}
  Since the $L^1$ bound $\norm{f-u(t,x)\M^s(v)}_{L^1}< C$ is well
  known, one may interpolate between this $L^1$ estimate and the $L^2$
  decay obtained in Theorem \ref{thm:mainKFP}. This yields the decay
  $$
  \norm{f-u(t,x)\M^s(v)}_{L^p}
  \lesssim
  t^{-q}t^{-\frac{\nu}{s}\left(1-\frac{1}{p}\right)},  
  $$
  which is clearly not optimal and degenerates as $p\to 1^{+}$.
\end{osservazione}

\begin{osservazione}
  Theorem \ref{thm:mainKFP} is the only one in which the computations
  can be made fully explicit, making it the natural test case for
  understanding the optimality of the second-order term and its
  dependence on the fractional parameter $s$.  As previously noted,
  there is a change of regime at $s=\frac{1}{2}$, which is where the
  map $z\mapsto z^{2s}$ transitions from convex to concave. The speed
  of convergence saturates as $s\to 0^+$ and it never reaches the
  threshold $t^{-1}$.
    
  Additionally, after the rescaling \eqref{eq:scaledf} the diffusive
  limit of the solution $f^\e$ to \eqref{eq:general-scaled} is of first
  order  in $\e$, i.e. as a consequence of Theorem \ref{thm: main},
  one has
  $$ \norm{f^\e-\M^s u}_{L^p}\lesssim\e\qquad s\in(1/2,1], \quad
  \text{uniformly in $[t_0,\infty)$}.
  $$
  On the other hand,
  $$
  \norm{f^\e-\M^s u}_{L^p}\lesssim\e^\nu\qquad s\in(0,1/2], \quad
  \text{uniformly in $[t_0,\infty)$}, $$ for a certain $\nu<2s$.  Indeed, for
  the highly fractional case, the equilibrium $\M^s$ lacks a finite
  first moment and the formal Hilbert expansion cannot be performed at
  the first order.
\end{osservazione}

\paragraph{Explicit solution} The Fourier transform of equation
\eqref{eq: KFP} has an explicit solution in both space and velocity.
Denoting by $\xi$ and $\eta$ the Fourier variables of  $x$ and
$v$, respectively, the transformed function
$\widehat{f} = \widehat{f}(t, \xi, \eta)$ solves
\begin{equation*}
    \partial_t\widehat{f}+\left(\eta-\xi\right)\cdot\nabla_\eta \widehat{f}=-|\eta|^{2s}\widehat{f}.
\end{equation*}
 Solving along the characteristics, we have

\begin{equation*}
     \widehat f(t,\xi,\eta)
     =\widehat f_0\left(\xi, e^{-t}\eta+(1-e^{-t})\xi\right)\exp\left\{-\int_0^t\abs{\left(1-e^{-(t-\tau)}\right)\xi+e^{-(t-\tau)}\eta}^{2s}\d \tau\right\}.
\end{equation*}
First, let us notice that for $s=1$ one gets 
\begin{equation*}
    \exp\left\{-\int_0^t\abs{\left(1-e^{-(t-\tau)}\right)\xi+e^{-(t-\tau)}\eta}^{2}\d \tau\right\}=\widehat{\mathbf{G}} _{\Sigma_{FP}}(\xi,\eta)
\end{equation*}
where the covariance matrix can be written $$\Sigma_{FP}\Sigma_{FP}^T=
\begin{pmatrix}
    t-\frac{1}2\left(3-4e^{-t}+e^{-2t}\right)&\frac{1}{2}(1-e^{-t})^2\\\frac{1}{2}(1-e^{-t})^2 &\frac{1}{2}(1-e^{-2t})
\end{pmatrix}. 
$$
Moreover a straightforward calculation shows that 
$$
\mathcal{F}^{-1}\left[\widehat f_0\left(\xi, e^{-t}\eta+(1-e^{-t})\xi\right)\right]=T_tf_0,
$$
where $T_t$ again is defined  in \eqref{eq:TranspSG}.

Therefore, for $s=1$, the solution can be written as a convolution
$$
f=\mathbf{G}_{\Sigma_{FP}}*(T_tf_0).
$$
For $s\in(0,1),$ we define the characteristic exponent as
$$
\Psi_{FP}(\xi,\eta)=\int_0^t\abs{\left(1-e^{-(t-\tau)}\right)\xi+e^{-(t-\tau)}\eta}^{2s}\d \tau
$$
and thus the solution is
$$
f=\mathbf{G}^{s}_{\Psi_{FP}}*(T_tf_0)
$$
where 
$$
\widehat{\mathbf{G}}^{s}_{\Psi_{FP}}(\xi,\eta)=e^{-\Psi_{FP}(\xi,\eta)}.
$$

\subsection{Proof of Theorem \ref{thm:mainKFP}: \texorpdfstring{$s=1$}{s=1}}
We first study the classical case $s=1$. As we have seen, in this case the fundamental solution is a Gaussian, whose inverse Fourier transform is explicit, therefore  allowing us to obtain the second-order estimates in all $L^p$ spaces.

The asymptotic profile is $u(t,x)\M^s(v)=\mathbf{G}_{\Psi_\infty}*(u_0\otimes \delta_v)$ where the characteristic exponent of $\mathbf{G}_{\Psi_\infty}$ is
$$
\Psi_\infty(\xi, \eta)=t|\xi|^{2s}+\frac{1}{2s}|\eta|^{2s}
$$
The following proposition, analogous to Proposition \ref{prop:
  DistInitDataMultGauss}, is stated for all $s\in(0,1]$, as it will be
needed in the fractional case as well.
\begin{proposizione} \label{prop:DistInitDataStableFP} Let
  $s \in (0,1]$ and $f_0:\R^d\times\R^d\to [0,+\infty)$ a probability
  density initial condition with $f_0 \in L^p_{x,v}$ and with $M_\nu$
  defined in \eqref{eq:M_nu}.  Then, there exists a positive constant
  depending only on $d,s$, $M_\nu,$ and $\nu$, such that
  $$
  \norm{\mathbf{G}^s_{\Psi_\infty}*(T_tf_0)
    -\mathbf{G}^s_{\Psi_\infty}*(u_0\otimes \delta_v)}_{L^p_{x,v}}
  \le
  C_{d,s,\nu}t^{-q} t^{-\frac{\nu}{2s}}
  $$
\end{proposizione}

\begin{proof}
  The proof follows the same ideas as in Proposition \ref{prop:
    DistInitDataMultGauss}.  Since $\mathbf{G}(x,v)$ separates the
  variables, we write $\mathbf{G}(x,v)=G^s_1(x)G^s_2(v)$ and
  $$f_h(y,w)=f_0(y-(1-e^{-t})w,w)\qquad g_h(y,w)=f_h(y,w)-f_0(y,w). $$ Observe that
  $\ird g_h(y,w)\d w=\rho_h(y)-\rho_0(y)$, with $\rho_h(y)=\ird f_h(y,w)\d w$
  consistently to the definition in Lemma \ref{lemma:
    distinitialdata}.  We compute
  \begin{multline*}
    \mathbf{G}^s_{\Psi_\infty}*(T_tf_0)-\mathbf{G}^s_{\Psi_\infty}*(u_0\otimes \delta_v)
    \\
    =   \ird\ird G^s_1(x-y)G^s_2(v-z) e^{dt}f_0(y-(e^t-1)z,e^tz)\d z\d y-G^s_2(v)\ird\ird G^s_1(x-y) f_0(y,w)\d y\d w
    \\
    = \ird\ird G^s_1(x-y)G^s_2(v-e^{-t}w) f_0(y-(1-e^{-t})w,w)\d w\d y-G^s_2(v)\ird\ird G^s_1(x-y) f_0(y,w)\d y\d w
    \\
    =G_2^s(v)\ird G^s_1(x-y)\ird g_h(y,w)\d w\d y 
    +\ird\ird G^s_1(x-y)f_h(y,w) \left(G^s_2(v-e^{-t}w)-G^s_2(v)\right)\d w\d y\\
    =I+II,
  \end{multline*}
  by writing $G^s_2(v-e^{-t}w)=G_2^s(v)+[G_2^s(v-e^{-t}w)-G_2^s(v)]$.
  Taking the $L^p$ norm gives
  $$
  \norm{I}_{L^p_{x,v}}\le \norm{G_2^s}_{L^p_v}\norm{G_1^s*\rho_h-G_1^s*\rho_0}_{L^p_x}\lesssim t^{-q}t^{-\frac{\nu}{2s}}
$$
To bound the second term, applying a Taylor expansion similarly to Proposition \ref{prop: DistInitDataMultGauss}, one has the following estimate
$$
\norm{G_2^s\left(\cdot-e^{-t}w\right)-G_2^s(\,\cdot\,)}_{L^p_v}\le e^{-t}|w|\norm{\nabla G_2^s}_{L^p_v}
$$
and 
$$
\norm{G_2^s\left(\cdot-e^{-t}w\right)-G_2^s(\,\cdot\,)}_{L^p_v}\le 2\norm{G_2^s}_{L^p_v}.
$$
Interpolating when $\nu<1$ we obtain the general estimate
$$
\norm{G_2^s\left(\cdot-e^{-t}w\right)-G_2^s(\,\cdot\,)}_{L^p_v}\le 2^{1-\nu}\norm{G_2^s}_{L^p_v}^{1-\nu}|w|^{\nu}e^{-\nu t}\norm{\nabla G_2}_{L^p_v}^\nu\le \tilde{C}_{d,s,\nu} |w|^{\nu} e^{-\nu t}.
$$
Therefore 
$$
\norm{II}_{L^p_{x,v}}\le \tilde{C}_{d,s,\nu}e^{-\nu t}\norm{G_1^s}_{L^p_x}\ird\ird |f_h(y,w)| |w|^{\nu}\d y\d w\le M_\nu C_{d,s,\nu}t^{-q}t^{-\frac{\nu}{2s}}.
$$
Notice that, differently from the nonlocal case, we do not need to lower the power on the norm of $\nabla G^s_2$, thanks to its better regularity.  Taking into account the first term $I$, the proof is completed.
\end{proof}

We now present the proof of the main theorem.
\begin{proof}[Proof of Theorem \ref{thm:mainKFP}]
We split
$$
\norm{f-\mathcal{M}(v)u(t,x)}_{L^p}\le \norm{\mathbf{G}_{\Sigma_{FP}}*(T_tf_0)-\mathbf{G}_{\Sigma_{\infty}}*(T_t f_0)}_{L^p}+\norm{\mathbf{G}_{\Sigma_{\infty}}*(T_tf_0)-\mathbf{G}_{\Sigma_{\infty}}*(u_0\otimes \delta_v)}_{L^p}
$$
The second term can be easily bounded by  Proposition \ref{prop:DistInitDataStableFP} with $\nu=1$.
For the first one, applying Young's convolution inequality and Proposition
 \ref{prop: boundMultGauss},  one has 
 \begin{multline}\label{eq:aux1}
  \norm{\mathbf{G}_{\Sigma_{FP}}*(T_tf_0)-\mathbf{G}_{\Sigma_{\infty}}*(T_t f_0)}_{L^p}
  \le
  \norm{\mathbf{G}_{\Sigma_{FP}}-\mathbf{G}_{\Sigma_{\infty}}}_{L^p}\norm{T_t f_0}_{L^1}
  \\
  \le
  Ct^{-q}\norm{\Delta(\Sigma_{FP}^T\Sigma_{FP},\Sigma_\infty^T\Sigma_\infty)}_{F},
 \end{multline}
 where in the last step we have applied Proposition \ref{prop: boundMultGauss} with $\Delta(\Sigma_{FP}^T\Sigma_{FP},\Sigma_\infty^T\Sigma_\infty)=\Sigma_\infty^{-1}(\Sigma_{FP}^2-\Sigma_\infty ^2) \Sigma_\infty^{-1}$, 
 having chosen $\Sigma=(\Sigma^T\Sigma)^{1/2}$.
While computing the exact eigenvalues of this matrix, is cumbersome, it is a standard calculation. However, for the sake of clarity, we present an asymptotic approximation: the term $\Delta(\Sigma_{FP}^T\Sigma_{FP},\Sigma_\infty^T\Sigma_\infty)$ can be approximated by $\Delta(A,\Sigma_\infty^T\Sigma_\infty)$
where 
$$
A=
    \begin{pmatrix}
        t-\frac{3}{2}&\frac{1}{2}\\ \frac12&\frac12
    \end{pmatrix}
    \sim \Sigma_{FP}^T\Sigma_{FP}
$$
up to an exponentially decaying-in-time error. We then have
$$
\Delta(A,\Sigma_\infty^T\Sigma_\infty)=
\begin{pmatrix}

            1-\frac{3}{2t}&\frac12\sqrt{\frac2t}\\
            \frac12\sqrt{\frac2t}&1
        \end{pmatrix}-\Id=
        \begin{pmatrix}
            -\frac{3}{2t}&\frac12\sqrt{\frac2t}\\
            \frac12\sqrt{\frac2t}&0
        \end{pmatrix}.
        $$
For the block corresponding to dimension $d=1$ the eigenvalues $\lambda_{\pm}$
        are  solutions to the characteristic equation
        $$\lambda^2-\left(-\frac{3}{2t}\right)\lambda-\left(\frac12\sqrt{\frac2t}\right)^2=0,$$ which yields 
        $$
        \lambda_{\pm}=-\frac{1}{4t}(3\mp\sqrt{9+8t})\sim\pm\frac{\sqrt{2}}{2}t^{-1/2}.        $$
The eigenvalues for the full $d-$dimensional case are such that 
        $$\lambda_1^+=\dots =\lambda_d^+\sim\frac{\sqrt2}{2}t^{-1/2}\qquad\text{and} \qquad\lambda^-_{1}=\dots=\lambda^-_{d}\sim-\frac{\sqrt2}{2}t^{-1/2}.$$
Therefore, plugging in \eqref{eq:aux1}, the final decay of the first term is 
$$
\norm{\mathbf{G}_{\Sigma_{FP}}*(T_tf_0)-\mathbf{G}_{\Sigma_{\infty}}*(T_t f_0)}_{L^p}\le C t^{-q} t^{-\frac{1}{2}},
$$
which completes the proof.
 \end{proof}

\subsection{Proof of Theorem \ref{thm:mainKFP}: \texorpdfstring{$s\in(0,1)$}{s in (0,1)}}
In this section we analyse the fractional case $s\in(0,1)$ of the Kinetic Fokker--Planck equation. 

Recall from the explicit solution computed above that we denote by $\mathbf{G}_\Psi$ the multivariate stable distribution with characteristic exponent $\Psi$, i.e., $\widehat{G}(\zeta)=e^{-\Psi(\zeta)}$.

As in Section \ref{sec: NLKFP}, we split
$$
\norm{f-u(t,x)\M(v)}_{L^p}\le \norm{\mathbf{G}_{\Psi_{FP}}*T_tf_0-\mathbf{G}_{\Psi_{\infty}}*T_tf_0}_{L^p}+ \norm{\mathbf{G}_{\Psi_\infty}*T_tf_0-\mathbf{G}_{\Psi_\infty}*T_t(u_0\otimes\delta_0)}_{L^p}.
$$
The second term is analysed as the  case $s=1$, applying Proposition \ref{prop:DistInitDataStableFP}. For the first one we need the following two technical lemmas.

\begin{lemma}\label{lemma: trick2}
    For all $a,b\in \R^d$,  for all $s\in(0,1]$,
    \begin{equation}\label{eq: claimIneq2}
        \abs{\abs{a+b}^{2s}-\abs{a}^{2s}-\abs{b}^{2s}}\le 
        \begin{cases}
            C_s\min\{|a||b|^{2s-1},|a|^{2s-1}|b|\}&\text{ if }s\in(1/2,1]
            \\
            2\min\{|a|^{2s},|b|^{2s}\}&\text{ if }s\in(0,1/2]
        \end{cases}
    \end{equation}
\end{lemma}
\begin{proof}
For $s=1$, it is the classical Cauchy--Schwarz inequality.
When either $a=0$ or $b=0$, the claim is immediate so let us assume $|a|,|b|\ne0$.

\begin{itemize}
    \item If $s\in(1/2,1]$, defining $x=\frac{a}{|b|}$ and $\theta=\frac{b}{|b|}$, proving 
    $$
     \abs{\abs{a+b}^{2s}-\abs{a}^{2s}-\abs{b}^{2s}}\le  C_s|a||b|^{2s-1}
    $$
    
    is equivalent to proving
    $$
    \abs{\abs{x+\theta}^{2s}-|x|^{2s}-|\theta|^{2s}}\le C_s|x|.
    $$
    For $s\in(1/2,1]$, the function $|\cdot|^{2s}$ is convex and locally Lipschitz with Lipschitz constant $L_s=2s2^{2s-1}\max\{1,|x|^{2s-1}\}$.

    For $|x|\le 1$
    \begin{equation*}
        \abs{\abs{x+\theta}^{2s}-|x|^{2s}-|\theta|^{2s}}\le \abs{\abs{x+\theta}^{2s}-|\theta|^{2s}}+|x|^{2s}
        \\
        \le L_s|x|+|x|^{2s}\le (2s2^{2s-1}+1)|x|
    \end{equation*}
    
    For $|x|>1$
    \begin{multline*}
        \abs{\abs{x+\theta}^{2s}-|x|^{2s}-|\theta|^{2s}}
                \le \abs{\abs{x+\theta}^{2s}-|x|^{2s}}+|\theta|^{2s}
                \\
                \le L_s|\theta|+1\le 2s2^{2s-1}|x|^{2s-1}+1\le (2s2^{2s-1}+1)|x|
    \end{multline*}

  \item If $s\in(0,1/2]$ the claim comes directly from the
    subadditivity of $|\cdot|^{2s}$: dividing
      $$
     \abs{\abs{a+b}^{2s}-\abs{a}^{2s}-\abs{b}^{2s}}\le  2|b|^{2s}
    $$
    by $|b|^{2s}$, one easily proves the analogous inequality
     \begin{equation*}
       \abs{\abs{x+\theta}^{2s}-1-|x|^{2s}}
       \le
       |x|^{2s}+1-|x+\theta|^{2s}\le \left(\abs{x+\theta}^{2s}+1\right)+1- \abs{x+\theta}^{2s}=2
     \end{equation*}
   \end{itemize}
   By symmetry the same inequalities are true if we switch the roles of
   $a$ and $b$ and this gives \eqref{eq: claimIneq2}.
\end{proof}

\begin{lemma}
  \label{lemma: lowerboundinf}
  For all $t>2$, there exists a constant $c>0$ depending only on $s$
  and $d$ such that
  $$
  \Psi_{FP}\ge c\Psi_{\infty}.
  $$
\end{lemma}

\begin{proof}
  Assume without loss of generality that $|\xi|\ne 0$, otherwise the
  proof is straightforward. Define $R=\frac{|\eta|}{|\xi|}$,
  $T=4\log R$, and $t_0=2$.
  \begin{itemize}
  \item For \textit{ $t\ge T$}, i.e. $|\eta|\le e^{t/4}|\xi|$, we have
    $$ \int_0^t\abs{(1-e^{-u})\xi+e^{-u}\eta}^{2s}\d u\ge 
    \int_{0}^{t/4}\abs{(1-e^{-u})\xi+e^{-u}\eta}^{2s}\d u
    +\int_{t/2}^t\abs{(1-e^{-u})\xi+e^{-u}\eta}^{2s}\d u
    $$
    and, for $u\in(t/2,t)$, this implies
    $$
    \abs{(1-e^{-u})\xi+e^{-u}\eta}\ge (1-e^{-u})|\xi|-e^{-u}|\eta|\ge
    (1-e^{-u})|\xi|-e^{-(u-t/4)}|\xi|\ge (1-e^{-t/2}-e^{-t/4})|\xi|,
    $$
    that is,
    \begin{equation*} 
      \int_{t/2}^t\abs{(1-e^{-u})\xi+e^{-u}\eta}^{2s}\d u
      \ge
      \frac{t}{2}(1-e^{-t_0/2}-e^{-t_0/4})^{2s}|\xi|^{2s}
      =c_0t|\xi|^{2s}
    \end{equation*} 
    Now if $|\eta|<2|\xi|$
    $$
    t|\xi|^{2s}=\frac{t}{2}|\xi|^{2s}+\frac{t}{2}|\xi|^{2s}\ge \frac{t}{2}|\xi|^{2s}+\frac{t_0}{2}\frac{1}{2^{2s}}|\eta|^{2s}\ge c_1(t|\xi|^{2s}+|\eta|^{2s})
    $$
    and thus 
    $$
    \int_0^t\abs{(1-e^{-u})\xi+e^{-u}\eta}^{2s}\d u
    \ge
    \int_{t/2}^t\abs{(1-e^{-u})\xi+e^{-u}\eta}^{2s}\d u
    \ge
    c(t|\xi|^{2s}+\frac{1}{2s}|\eta|^{2s}) 
    $$
    If $2|\xi|\le |\eta|\le e^{t/4}|\xi|$, notice that, since $t>t_0$,
     $(0,1)\cap (t/2,t)=\emptyset$. Therefore, for $u\in(0,1)$
    $$
    \abs{(1-e^{-u})\xi+e^{-u}\eta}\ge e^{-u}|\eta|-(1-e^{-u})|\xi|\ge e^{-1}|\eta|-(1-e^{-1})\frac{|\eta|}{2}=c_3|\eta|
    $$
    therefore
    \begin{multline} \int_0^t\abs{(1-e^{-u})\xi+e^{-u}\eta}^{2s}\d
      u\ge \int_{0}^1\abs{(1-e^{-u})\xi+e^{-u}\eta}^{2s}\d u
      +\int_{t/2}^t\abs{(1-e^{-u})\xi+e^{-u}\eta}^{2s}\d u
      \\
      \ge c_4\frac{1}{2s}|\eta|^{2s}+c_0t|\xi|^{2s}.
    \end{multline}

  \item For \textit{$t<T$,} i.e., $|\eta|>e^{t/4}|\xi|$, we  prove the following two facts.
    \begin{enumerate}[(a)]
    \item $\Psi_\infty\le K|\eta|^{2s}$
    \item $\Psi_{FP}\ge c|\eta|^{2s}$
    \end{enumerate}
    which together imply $\Psi_{FP}\ge \frac{c}{K}\Psi_{\infty}$.

    \begin{enumerate}[(a)]
    \item On the one hand we have $$
      t|\xi|^{2s}+\frac{1}{2s}|\eta|^{2s}\le te^{-st/2}|\eta|^{2s}+\frac{1}{2s}|\eta|^{2s}\le K|\eta|^{2s}
      $$
      where $K=\sup_{t\ge t_0}\left(te^{-st/2}+\frac{1}{2s}\right)$
    \item On the other hand, for $u\in(0,t_0/4)$
    \begin{multline*}
    \abs{(1-e^{-u})\xi+e^{-u}\eta}\ge e^{-u}|\eta|-(1-e^{-u})|\xi|\ge |\eta|\left( e^{-u}-(1-e^{-u})e^{-t/4}\right)
    \\
    \ge
   |\eta| \left( e^{-u}-(1-e^{-u})e^{-t_0/4}\right)\ge |\eta|e^{-t_0/2},
    \end{multline*}
     thus 
    $$
    \int_0^t \abs{(1-e^{-u})\xi+e^{-u}\eta}^{2s}\ge  \int_0^{t_0/4} \abs{(1-e^{-u})\xi+e^{-u}\eta}^{2s}\ge e^{-st_0}\frac{t_0}{4}|\eta|^{2s}
    $$
\end{enumerate}

\end{itemize}

\end{proof}

We are now ready to prove Theorem \ref{thm:mainKFP} in the fractional case $s\in(0,1)$.
\begin{proof}
As for the previous case we split
$$
\norm{f-\mathcal{M}^s(v)u(t,x)}_{L^p}\le \norm{\mathbf{G}_{\Psi_{FP}}*(T_tf_0)-\mathbf{G}_{\Psi_{\infty}}*(T_t f_0)}_{L^p}+\norm{\mathbf{G}_{\Psi_{\infty}}*(T_tf_0)-\mathbf{G}_{\Psi_{\infty}}*(u_0\otimes \delta_v)}_{L^p}
$$
and Proposition \ref{prop:DistInitDataStableFP} gives that the second term decays as 

$$
\norm{\mathbf{G}_{\Psi_{\infty}}*(T_tf_0)-\mathbf{G}_{\Psi_{\infty}}*(u_0\otimes \delta_v)}_{L^p}\le Ct^{-q}t^{-\frac{\nu}{2s}}
$$
For the first term we consider first the case $p=2$, then $p=\infty$,
and we interpolate between those spaces to obtain the decay in the
intermediate spaces.

For $p=2$, we apply Plancherel's theorem, obtaining
$$
\norm{G_{\Psi_{FP}}-G_{\Psi_{\infty}}}^2_{L^2_{x,v}}=C\norm{e^{-\Psi_{FP}}-e^{-\Psi_{\infty}}}_{L^2_{\xi,\eta}}^2
$$
We consider first the case $s\in(1/2,1]$.
Applying Lemma \ref{lemma: trick2}, one can show that 
\begin{multline*}
    \abs{\Psi_{FP}(\xi,\eta)-\Psi_{\infty}(\xi,\eta)}\\
    \le
        \int_0^t\abs{\abs{(1-e^{-u})\xi+e^{-u}\eta}^{2s}-\abs{(1-e^{-u})\xi}^{2s}-\abs{e^{-u}\eta}^{2s}} \d u\\
     +\abs{
     \int_0^t\abs{(1-e^{-u})\xi}^{2s}\d u-t|\xi|^{2s}}+\abs{\Big(\int_0^te^{-2su}\d u \Big)|\eta|^{2s}-\frac{1}{2s}|\eta|^{2s}}
     \\
     =\int_0^t\abs{\abs{(1-e^{-u})\xi+e^{-u}\eta}^{2s}-\abs{(1-e^{-u})\xi}^{2s}-\abs{e^{-u}\eta}^{2s}} \d u+
     \\
     \left[t-\int_0^t(1-e^{-u})^{2s}\d u)\right]|\xi|^{2s}+\left[\frac{1}{2s}-
     \int_0^t e^{-2s u}\d u
     \right]|\eta|^{2s}
     \\
     \le C_s |\xi||\eta|^{2s-1}     
 \int_0^t\left((1-e^{-u})\right)e^{-(2s-1)u}\d u +2s(1-e^{-t})|\xi|^{2s}+\frac{1}{2s}e^{-2st}|\eta|^{2s}\\
 \le C_s(|\xi||\eta|^{2s-1}+|\xi|^{2s}+e^{-2st}|\eta|^{2s}):=I+II+III.
\end{multline*}
Moreover
applying Lemma \ref{lemma: lowerboundinf}, we have 
$$
e^{-\min\{\Psi_\infty,\Psi_{FP}\}}\le e^{-c\Psi_{\infty}}
$$ and hence
\begin{multline*}
    \norm{G_{\Psi_{FP}}-G_{\Psi_{\infty}}}^2_{L^2_{x,v}}\lesssim\norm{e^{-\Psi_{FP}}-e^{-\Psi_{\infty}}}_{L^2_{\xi,\eta}}^2\\
    =
    \ird\ird\abs{(\Psi_{FP}(\xi,\eta)-\Psi_{\infty}(\xi,\eta))\int_0^1 e^{-((1-\theta)\Psi_{FP}+\theta\Psi_\infty)}\d\theta}^2\d \xi\d\eta
    \\
    \le
     \ird\ird\abs{(\Psi_{FP}(\xi,\eta)-\Psi_{\infty}(\xi,\eta))e^{-\min\{\Psi_\infty,\Psi_{FP}\}}}^2\d \xi\d \eta
     \\
     \le 
     \ird\ird\abs{e^{-c(t|\xi|^{2s}+\frac{1}{2s}|\eta|^{2s})}(\Psi_{FP}(\xi,\eta)-\Psi_{\infty}(\xi,\eta))}^2\d \xi\d \eta    
     \\
     \le 
     \ird\ird\abs{e^{-c(t|\xi|^{2s}+\frac{1}{2s}|\eta|^{2s})}(I+II+III)}^2\d \xi\d \eta.
\end{multline*}
The first term is
\begin{multline*}
     \ird\ird\abs{e^{-c(t|\xi|^{2s}+\frac{1}{2s}|\eta|^{2s})}c_1|\xi||\eta|^{2s-1}}^2\d \xi\d \eta
 \\
 \lesssim
 \ird\ird\abs{e^{-c(|\omega|^{2s}+\frac{1}{2s}|\eta|^{2s})}t^{-\frac{1}{2s}}|\omega||\eta|^{2s-1}}^2t^{-\frac{d}{2s}}\d \omega\d \eta\\
    \lesssim t^{-\frac{d}{2s}}t^{-\frac{1}{s}}\ird\ird e^{-2c(|\omega|^{2s}+\frac{1}{2s}|\eta|^{2s})}|\omega||\eta|^{2s-1}\d \omega\d \eta
    \lesssim t^{-\frac{d}{2s}}t^{-\frac{1}{s}},
\end{multline*}
after the change of variable $\omega=t^{1/(2s)}\xi$.
 With a similar computation one can see that $II$ decays like $t^{-2}t^{-d/(2s)}$, which decays faster than $I$ for $s>\frac{1}{2}$, while $III$ decays exponentially.
 Taking the square root, one has the claim for $s\in(1/2,1]$.

 The case $s\in(0,1/2]$ can be proven in a similar way from the second estimate of  Lemma \ref{lemma: trick2} which then
 concludes the proof for $p=2$.

 The case $p=\infty$ relies on the simple estimate
 \begin{multline*}
    \abs{f_{FP}(x,v)-f_\infty(x,v)}=\abs{\ird\ird e^{i(x,v)\cdot (\xi,\eta)}\widehat{f}_{FP}(\xi,\eta)-\widehat{f}(\xi,\eta)_\infty\d\xi\d\eta}
    \\
    \le
    \ird\ird \abs{e^{-\Psi_{FP}(\xi,\eta)}-e^{-\Psi_\infty(\xi,\eta)}}\d\xi\d\eta.
    \end{multline*}
    Then, applying the same kind of argument as before, one has
    $$
    \ird\ird \abs{e^{-\Psi_{FP}(\xi,\eta)}-e^{-\Psi_\infty(\xi,\eta)}}\d\xi\d\eta\le
      \ird\ird\abs{e^{-c(t|\xi|^{2s}+\frac{1}{2s}|\eta|^{2s})}(I+II+III)}\d \xi\d \eta,
    $$
    which gives the claim in the $L^\infty$ case.
    For all $p\in(2,\infty)$ one concludes by interpolation, since the
    second order exponent is the same for both $L^2$ and $L^\infty$.
\end{proof}

\section*{Acknowledgments}

We would like to thank Cyril Imbert and Clément Mouhot, whose comments
improved the results in this paper. Part of this work was carried out
during the visit of J.A.C.~and N.T.~to the program \emph{``Kinetic Theory:
Novel Statistical, Stochastic and Analytical Methods''} at the Simons
Laufer Mathematics Institute.

J.A.C. and N.T. acknowledge support from the ``Maria de Maeztu''
Excellence Unit IMAG, reference CEX2020-001105-M, funded by
MCIN/AEI/10.13039/501100011033/. They were also supported by grant
PID2023-151625NB-I00 and the research network RED2022-134784-T from
MCIN/\-AEI/\-10.13039/501100011033/.  The project that gave rise to
these results received the support of a fellowship from the ``la
Caixa'' Foundation (ID 100010434). The fellowship code is
LCF/BQ/DI22/11940032.

\section*{Data Availability Statement}
No datasets were generated or analysed during the current study.
\section*{AI Declaration}
The scientific ideas and mathematical formulations in this paper were conceived solely by the authors. During the preparation of this work, the authors used  ChatGPT (OpenAI) and Claude (Anthropic) for language editing, as an additional check of the proofs, and for assistance in identifying and correcting minor issues. After using these tools, the authors independently verified all the results and take full responsibility for the accuracy and originality of this work.
\addcontentsline{toc}{section}{References}
\bibliography{bibliography}
\end{document}